\documentclass[a4paper,11pt,twoside]{amsart}
\usepackage[english]{babel}
\usepackage[utf8]{inputenc}

\usepackage[a4paper,inner=3cm,outer=3cm,top=4cm,bottom=4cm,pdftex]{geometry}
\usepackage{fancyhdr}
\usepackage{titlesec}
\usepackage{enumerate}
\usepackage{color}
\usepackage{bold-extra}
\titleformat{\section}{\normalfont\scshape\centering}{\thesection}{1em}{}
\titleformat{\subsection}{\bfseries}{\thesubsection}{1em}{}

\usepackage{comment}
\usepackage{graphics}
\usepackage{aliascnt}
\usepackage[pdftex,citecolor=green,linkcolor=red]{hyperref}
\usepackage{mathrsfs} 

\usepackage{amsmath}
\usepackage{amsfonts}
\usepackage{amssymb}
\usepackage{amsthm}
\usepackage{comment}
\usepackage{mathtools}

\newtheorem{theorem}{Theorem}[section]
\newtheorem{corollary}[theorem]{Corollary}

\newtheorem{lemma}[theorem]{Lemma}
\newtheorem{proposition}[theorem]{Proposition}
\theoremstyle{definition}
\newtheorem{definition}[theorem]{Definition}
\newtheorem{convention}[theorem]{Convention}
\newtheorem{remark}[theorem]{Remark}

\numberwithin{equation}{section}

\newcommand{\supp}{\textup{supp}}
\newcommand{\lcm}{\textup{lcm}}

\newcommand{\re}{\textup{Re}}
\renewcommand{\Im}{\textup{Im}}
\renewcommand{\Re}{\textup{Re}}
\renewcommand\d{\textnormal{d}}
\newcommand{\md}[1]{\ensuremath{(\mbox{mod}\, #1)}}
\newcommand{\N}{\textnormal{N}}
\usepackage{soul}
\newcommand{\psum}{\sideset{}{^{'}}\sum}

\newcommand{\C}{15.1}
\newcommand{\Co}{19.2}
\newcommand{\constant}{20/363}

\begin{document}

\title[Gaussian almost primes in almost all narrow sectors]{Gaussian almost primes in almost all narrow sectors}

\author[Olli J\"arviniemi]{Olli J\"arviniemi}
\address{Department of Mathematics and Statistics, University of Turku, 20014 Turku, Finland}
\email{olli.jarviniemi@gmail.com}

\author[Joni Ter\"{a}v\"{a}inen]{Joni Ter\"{a}v\"{a}inen}
\address{Department of Mathematics and Statistics, University of Turku, 20014 Turku, Finland}
\email{joni.p.teravainen@gmail.com}

\begin{abstract}
 We show that almost all sectors of the disc $\{z \in \mathbb{C}: |z|^2\leq X\}$ of area $(\log X)^{\C}$ contain products of exactly two Gaussian primes, and that almost all sectors of area $(\log X)^{1 + \varepsilon}$ contain products of exactly three Gaussian primes. The argument is based on mean value theorems, large value estimates and pointwise bounds for Hecke character sums.
\end{abstract}

\maketitle

\section{Introduction}

Our aim in this paper is to establish results on the distribution of Gaussian almost primes in very small sectors. The ring $\mathbb{Z}[i]$ of Gaussian integers is a unique factorization domain, so we have a unique representation for a Gaussian integer as a product of primes, up to factors that are powers of $i$. 

In what follows, for symmetry reasons we restrict our Gaussian integers to
\begin{align*}
\mathbb{Z}[i]^{*}\coloneqq\{n\in \mathbb{Z}[i]\setminus \{0\}:\,\, 0\leq \arg(n)<\pi/2\},    
\end{align*}
i.e., the set of Gaussian integers in the first quadrant. The primes in $\mathbb{Z}[i]^*$ are precisely $1+i$, the rational primes $\equiv 3 \pmod 4$, and elements $a+bi$  with $a,b>0$ whose norm $\N(a+bi)\coloneqq a^2+b^2$ is an odd prime. By a product of $k$ Gaussian primes (or loosely speaking a Gaussian almost prime) we mean an element $n \in \mathbb{Z}[i]^*$ of the form $n = up_1\cdots p_k$, where $p_i \in \mathbb{Z}[i]^*$ are Gaussian primes and $u\in \{\pm 1,\pm i\}$ is a unit. 

We shall investigate the angular distribution of the Gaussian almost primes. Thus, we consider the measure of $\theta\in [0,\pi/2)$ for which a narrow sector
\begin{align*}
S_{\theta}\coloneqq\{n\in \mathbb{Z}[i]^*,\N(n)\leq X:\,\, \theta \le \arg(n)<\theta+h/X\}    
\end{align*}
contains no Gaussian almost primes, with $h$ as small as possible in terms of $X$. 
In this setting, we say that a property $P_{\theta,X}$ holds for \emph{almost all} $\theta\in [0,\pi/2)$ if the Lebesgue measure of those $\theta$ for which $P_{\theta,X}$ fails is $o_{X\to \infty}(1)$.  

For $h<X^{1/2}$, it is easy to see that there exist sectors (in particular $S_{\theta}$ for $\theta$ close enough to $0$) which contain no Gaussian integers, let alone Gaussian almost primes. This is in contrast to the situation of primes in short intervals, where Cram\'er's conjecture predicts for $h=(\log X)^{2+\varepsilon}$ the existence of primes in $[X,X+h]$ for any $X\geq X_0(\varepsilon)$. One can also easily see (just by cardinality considerations) that if $h=o((\log X)/(\log \log X)^{k-1})$, then almost all sectors $S_{\theta}$ contain no products of $k$ Gaussian primes. Our first main theorem shows that this is essentially sharp; as soon as we have a sector of slightly larger width $(\log X)(\log \log X)^{C}/X$, with $C$ suitably large, it does almost always contain 
products of three Gaussian primes. 

\begin{theorem}
\label{thm:E3}
Let $h = (\log X)(\log \log X)^{\Co}$. Almost all sectors $\{n\in \mathbb{Z}[i]^*,\N(n) \le X : \theta \le \arg n < \theta + \frac{h}{X}\}$ contain a product of exactly three Gaussian primes.
\end{theorem}

When it comes to products of two Gaussian primes, we are able to find them in almost all narrow sectors of ``logarithmic width'' $(\log X)^{C}/X$ for some explicit $C>1$. 

\begin{theorem}
\label{thm:E2}
Let  $h = (\log X)^{\C}$. Almost all sectors $\{n\in \mathbb{Z}[i]^*,\N(n) \le X : \theta \le \arg n < \theta + \frac{h}{X}\}$ contain a product of exactly two Gaussian primes.
\end{theorem}

We in fact prove a quantitative bound for the number of $p_1p_2$ or $p_1p_2p_3$ (with $\N(p_i)$ belonging to suitable intervals) in almost all narrow sectors; see Theorem~\ref{thm_variance}. 

\subsection{Previous works}

\label{sec:prev_work}

A central problem in the study of the distribution of Gaussian primes is to count primes in sectors $\{n\in \mathbb{Z}[i]^{*} : \N(n) \le X, \alpha \le \arg n < \beta\}$. An asymptotic formula for the number of primes has been established by Ricci~\cite{ricci} for sectors of area $X^{7/10 + \varepsilon}$, and a positive lower bound has been given by Harman and Lewis~\cite{harman-lewis} for sectors of area $X^{0.619}$.

The problem becomes easier if one only considers almost all sectors. Huang, Liu and Rudnick show in~\cite{huang-liu-rudnick} that almost all sectors of area $X^{2/5 + \varepsilon}$ contain the expected number of primes. Under GRH, works of Rudnick--Waxman~\cite{rudnick-waxman} and Parzanchevski--Sarnak ~\cite{p-sarnak} show that almost all sectors of area $(\log X)^{2 + \varepsilon}$ contain Gaussian primes for any fixed $\varepsilon>0$.

Another problem of interest is counting Gaussian primes in small circles. This corresponds to imposing both angular and norm constraints on Gaussian primes. Harman, Kumchev and Lewis~\cite{harman-kumchev-lewis} have shown that the distance to the nearest Gaussian prime from any point $z\neq 0$ is $\ll |z|^{0.53}$. Lewis has improved this to $|z|^{0.528}$ in his thesis~\cite{lewis}. Previous works in this area include Coleman's papers~\cite{coleman7/12, coleman0.55}. Asymptotic formulas for the number of primes satisfying both angular and norm constraints are given by Stucky~\cite{stucky}.

See also Chapter 11 of Harman's book~\cite{harman-book} for more on the topic and Duke's work~\cite{duke} for some related problems over general number fields.

\subsection{Overview of the method}

The overall strategy of our argument follows the approach of the second author~\cite{teravainen-primes} to almost primes in almost all short intervals, which in turn borrows ideas from the work of Matom\"aki and Radziwi{\l}{\l}~\cite{matomaki-radziwill} on multiplicative functions in short intervals. 
However, adapting these methods efficiently to the Gaussian primes requires several additional ideas.

By a simple Fourier argument (Lemma~\ref{lem:reduct}) and separation of variables, we reduce the task of bounding the variance of products of exactly two Gaussian primes in narrow sectors to mean square estimates of the shape 
\begin{align}\label{eq_meansq}
\sum_{0<m\leq T}|P_1(m)|^2|P(m)|^2\ll (\log X)^{-2-\varepsilon},    
\end{align}
where the \emph{Hecke polynomials}\footnote{We use this term to emphasize the analogy with Dirichlet polynomials.} $P_1(m)$, $P(m)$ are essentially of the form
\begin{align}
P_1(m)=\sum_{P_1 < \N(p)\leq 2P_1}\frac{\lambda^m(p)}{\N(p)},\quad\quad   P(m)=\sum_{X/P_1 < \N(p)\leq 2X/P_1}\frac{\lambda^m(p)}{\N(p)},   
\end{align}
with $P_1\approx h/(\log X)$ and $T\approx X/h$, and with $\lambda^m(z)=(z/|z|)^{4m}$ the \emph{angular Hecke characters}. The Hecke polynomial $P(m)$ is decomposed with Heath-Brown's identity as a product of several ``smooth'' Hecke polynomials (partial sums of Hecke $L$-functions), as well as some harmless very short Hecke polynomials, and one then splits the summation over $m$ into regions depending on the sizes of $P_1(m)$ and the factors coming from Heath-Brown's identity, different regions being handled by different arguments. 

We then attack the problem of bounding~\eqref{eq_meansq} by using various mean value theorems, large value estimates and pointwise bounds for Hecke polynomials. However, some complications arise when adapting such methods from the integers to the Gaussian integers. 

The main source of complications is that less is known about the \emph{Hecke $L$-functions} $$L(1/2+it,\lambda^m)=\sum_{n\in \mathbb{Z}[i]\setminus\{0\}}\lambda^m(n)\N(n)^{-1/2-it}$$ 
in the $m$-aspect than about the Riemann zeta function $\zeta(1/2+it)$ in the $t$-aspect. In particular, while for the Riemann zeta function one has estimates for twisted fourth moments (such as Watt's theorem~\cite{watt} that was employed in~\cite{teravainen-primes},~\cite{mato-tera-primes}), not even the fourth moment 
$$\sum_{m\leq T}|L(1/2,\lambda^m)|^4\ll T^{1+o(1)}$$
(or any moment higher than the second) is currently known for the Hecke $L$-functions. Furthermore, as remarked in e.g.~\cite{harman-lewis}, there is no good analogue of the Hal\'asz--Montgomery inequality (that was used in~\cite{matomaki-radziwill},~\cite{teravainen-primes},~\cite{mato-tera-primes}) for Hecke polynomials. This is ultimately because the $L$-function $L(s, \lambda^m)$ is of degree two, so that the pointwise estimates for it are essentially quadratic compared to the integer case (for instance, we have $|L(0,m)|\ll m$, whereas $|\zeta(it)|\ll |t|^{1/2}$, and we have
$|L(1/2,\lambda^m)|\ll m^{1/3+o(1)}$ by~\cite{sohne}, whereas we have $|\zeta(1/2+it)|\ll |t|^{1/6-\delta+o(1)}$ for $\delta=1/84$ by~\cite{bourgain}).

To overcome these limitations, we provide three tools: 
\begin{enumerate}
    \item An inequality of Hal\'asz--Montgomery type for Hecke polynomials that gives nontrivial bounds even for short polynomials (Proposition~\ref{prop:HM}).
    \item An improved mean value theorem for prime-supported Hecke polynomials, which takes into account the sparsity of the Gaussian primes (Proposition~\ref{prop:mvtapplication}).
    \item An improved large value theorem for short (of length $(\log X)^a$) prime-supported Hecke polynomials (Corollary~\ref{cor:newLVT}).
\end{enumerate}

For (1), our first aim is to obtain a power-saving bound for the sum
$$\sum_{\N(n) \sim N} \lambda^m(n).$$
We do this via the theory of exponential pairs. Writing $n = x+iy$, the sum at hand is a two-dimensional exponential sum with the phase function $m\arctan(y/x)/(\pi/2)$. By the triangle inequality it then suffices to obtain bounds for one-dimensional sums, to which the theory of exponential pairs may be applied. However, we encounter a technical complication: some of the higher order partial derivatives of $\arctan(y/x)$ vanish on certain lines $y = kx$. Hence, we must restrict our sums outside of the resulting problematic narrow sectors. As a result we obtain bounds of the form
\begin{align}\label{eq:delta}
\left|\sum_{\substack{\N(n)\sim N \\ \arg n \not\in I_1 \cup \ldots \cup I_r}}\lambda^m(n)\right|\ll N^{1-\delta}    
\end{align}
for certain (very short) intervals $I_i$, in the full range $N = m^{\alpha}, 0 < \alpha < 1$, with $\delta = \delta(\alpha) > 0$ explicit (and reasonable); this is Proposition~\ref{prop:zeta}(i). (We note that we also employ another approach based on Hecke $L$-functions and Perron's formula, which gives us a certain pointwise bound without any problematic sectors -- see Proposition~\ref{prop:zeta}(ii).)

By the usual Hal\'asz--Montgomery method, we then obtain an inequality of the form
\begin{align*}
\sum_{m\in \mathcal{T}}\left|\sum_{\substack{\N(n)\sim N \\ \arg n \not\in I_1 \cup \ldots \cup I_r}}a_n\lambda^m(n)\right|^2\ll (N+|\mathcal{T}|N^{1-\delta})(\log(N+T))^{O(1)}\sum_{\N(n)\sim N}|a_n|^2  
\end{align*}
that we need for adapting the Matom\"aki--Radziwi\l{}\l{} method (here $\mathcal{T}\subset [-T,T] \cap \mathbb{Z}$ and $\delta=\delta(\alpha)$ with $\alpha=(\log N)/(\log T)$), see Proposition~\ref{prop:HM}. Our exponent of $\log \log X$ or $\log X$ in the main theorems naturally depends on the values of $\delta$ that we obtain in~\eqref{eq:delta}, so we optimize the step where we apply exponent pairs. We consider the exponent pairs obtained from the application of $A$- and $B$-processes to the exponent pair $(0, 1)$.

For (2), we provide a mean value theorem for Hecke polynomials (Lemma~\ref{lem:IMVT}) that takes into account the sparsity of the coefficient sequence as in~\cite[Lemma 4]{teravainen-primes}. The mean value theorem itself is rather simple to derive, but to bound the resulting expression in the case of prime-supported sequences we need sharp upper bounds for sums over Gaussian primes of the type
\begin{align*}
\sum_{\substack{p_1, p_2 \\ \N(p_1), \N(p_2) \le X \\ |\arg p_1 - \arg p_2| \le h/X}} 1.
\end{align*}

In the integer case, the corresponding sum (with $|\arg p_1 - \arg p_2| \le h/X$ replaced by $|p_1 - p_2| \le h$) may be bounded quite directly with Selberg's sieve, but our problem here is more involved. Writing $p_1 = a+bi, p_2 = c+di$, the conditions in the sum translate (more or less) to $a^2 + b^2$ and $c^2 + d^2$ being primes with $a, b, c, d \le \sqrt{X}$ and $|ad - bc| \le h$. We wish to apply a sieve, and we thus consider, for various values of $|k| \le h$ and $T_1, T_2\leq X^{\delta}$, the sums
\begin{align}
\label{eq:shift_conv}
\sum_{\substack{a, b, c, d \le \sqrt{X} \\ ad - bc = k \\ T_1 \mid a^2 + b^2 \\ T_2 \mid c^2 + d^2}} 1.
\end{align}
This is similar to the divisor correlation
$$\sum_{\substack{a, b, c, d \in \mathbb{Z}_+ \\ bc \le X \\ ad - bc = k}} 1 = \sum_{n \le X} \tau(n)\tau(n+k),$$
albeit with slightly different boundary conditions and additional congruence conditions on the variables. We adapt (in Section~\ref{sec:divisor}) the work of Deshouillers and Iwaniec~\cite{DI} on divisor correlations to evaluate~\eqref{eq:shift_conv} with a power-saving error term for $T_1, T_2$ less than a small power of $X$. For the sieve approach to work, it is crucial that there is indeed a good error term and uniformity in all the parameters.

The application of this improved mean value theorem then importantly saves us a few factors of $\log X$ in certain parts of the argument, and this significantly reduces the value of the exponent that we obtain.

For (3), we prove a large value estimate for a prime-supported polynomial $P(m) = \sum_{p \sim P} a_p\lambda^m(p)$, where $P = (\log X)^a$, by applying a large value theorem to a suitable moment of $P(m)$. Such a method was used in~\cite[Lemma 8]{matomaki-radziwill}, where a moment of length $\approx X$ was used, together with a simple large value theorem arising from the usual mean value theorem. In contrast, we raise $P$ to a moment of length $X^\alpha$ for suitable $0 < \alpha < 1$, and apply a Huxley-type large value theorem (see Corollary~\ref{cor:newLVT}). This gives improved results for the number of large values $m$ for which $|P(m)|\geq P^{-o(1)}$ when $a > 6$.

\begin{remark}
We believe that there is no fundamental obstacle in also establishing an analogue of the Matom\"aki--Radziwi{\l}{\l} theorem for cancellations of multiplicative functions in almost all narrow sectors\footnote{Note that the problem reduces to the integer case if one considers Gaussian integers with their \emph{norm} in a short interval, i.e. sums of the form $\sum_{x < \N(n) \le x + h} f(n)$ with $f : \mathbb{Z}[i] \to \mathbb{C}$ multiplicative. Indeed, if one writes $g(k) = \sum_{n : \N(n) = k} f(n)$, then $g$ is multiplicative, divisor-bounded if $f$ is, and one has $\sum_{x < \N(n) \le x + h} f(n) = \sum_{x < k \le x + h} g(k)$. The analogous remark holds for multiplicative functions defined on the ideals of any number field.} by using our lemmas on Hecke polynomials in place of the Dirichlet polynomial lemmas in~\cite{matomaki-radziwill}. 

It is plausible that the methods used in this paper could be adapted to finding almost primes in almost all very small circles, too. Indeed, finding Gaussian primes in circles tends to be easier than for sectors (since we do have tools like the Hal\'asz--Montgomery inequality and Watt's bound for Hecke polynomials when averaging over both $t$ and $m$). For example, as mentioned in Section~\ref{sec:prev_work}, one can find Gaussian primes of norm less than $X$ in circles of area $X^{0.528}$, whereas for sectors the best result works for an area of $X^{0.619}$.

It should be possible to improve the exponent in Theorem~\ref{thm:E2} by incorporating Harman's sieve into our argument; to avoid complicating the arguments further, we do not pursue this improvement here. 
\end{remark}

\subsection{Notation}
\label{sec:notation}

\begin{convention} Unless otherwise stated, summation variables $n, n_i$, etc. are always restricted to lie in $\mathbb{Z}[i]^{*}$.
\end{convention}

By $y \sim X$ and $y \asymp X$ we mean $X<y\leq 2X$ and $X \ll y \ll X$, respectively.

The norm $a^2 + b^2$ of $n = a+bi \in \mathbb{Z}[i]$ is denoted by $\N(n)$. For Gaussian integers $n, m$, we write $n \equiv m$ if $n = um$ for some unit $u$. We denote by $\mathbb{P}_{\mathbb{Z}[i]}$ the set of all Gaussian primes.

For $n \in \mathbb{Z}[i]^{*}$, we let $\arg n$ take values $\md{\pi/2}$, so for example $\arg n \in [c, d]$ if and only if $\arg n \in [c-\pi/2, d - \pi/2]$, and the statement ``$|\arg (1 + 100i)| \le 1/10$'' is true.

We define analogues of usual multiplicative functions for Gaussian integers as follows. If $n = up_1 \cdots p_k$, where $p_i \in \mathbb{Z}[i]^*$ are Gaussian primes and $u \in \{1, i, -1, -i\}$ is a unit, we let $\mu(n) = 0$ if $p_i = p_j$ for some $i \neq j$, and otherwise $\mu(n) = (-1)^k$. If $n$ is a unit times the power of a Gaussian prime $p$, then we let $\Lambda(n) = \log \N(p)$ and otherwise $\Lambda(n) = 0$. We let $\tau(n)$ denote the number of $d \in \mathbb{Z}[i]^*$ for which there exists $m \in \mathbb{Z}[i]$ with $n = dm$.

If $k\neq 0$ is an integer and $p$ is a rational prime, we use $v_p(k)$ to denote the largest integer $a$ such that $p^a\mid k$.

The angular Hecke characters are given by
$$\lambda^m(n) = \left(\frac{n}{|n|}\right)^{4m},$$
with $m\in \mathbb{Z}$ and the corresponding Hecke $L$-function is given by
$$
L(s,\lambda^m)=\sum_{n\in \mathbb{Z}[i]^{*}}\frac{\lambda^m(n)}{\N(n)^s}
$$
for $\Re(s)>1$. One can continue $L(s,\lambda^m)$ meromorphically to the whole complex plane, and the resulting function is entire apart from a simple pole at $s=1$ in the case $m=0$. We denote $\lambda(n) = \lambda^1(n)$.

We write, for $t \in \mathbb{R}$, $e(t) = e^{2\pi i t}$, and thus $$\lambda^m(n) = e\left(\frac{m}{\pi/2}\arg n\right).$$
The distance of $t$ to the closest integer(s) is denoted by $\|t\|$.

\subsection{Structure of the paper} 
In Section~\ref{sec:reduction}, we reduce Theorem~\ref{thm:E2} and Theorem~\ref{thm:E3} to mean square estimates for Hecke polynomials using a standard Fourier expansion. We then derive some basic bounds for Hecke polynomials in Section~\ref{sec:preLem}. In Section~\ref{sec:pointwise}, we establish pointwise bounds for smooth Hecke polynomials, in particular using the theory of exponent pairs. In Section~\ref{sec:HM}, we apply the pointwise bounds from the previous section to obtain a Hal\'asz--Montgomery type estimate for Hecke polynomials and as its consequences several large value estimates for Hecke polynomials, including a large value estimate that works well for very short prime-supported Hecke polynomials. In Section~\ref{sec:factorize}, we show how to factorize mean squares of  Hecke polynomials using the improved mean value theorem, and most importantly, how to bound the error term in the case of Hecke polynomials supported on the primes or almost primes. The bounding of the error term relies on Theorem~\ref{prop_divisorsum}, an additive divisor problem in progressions with power-saving error term, whose 
proof we postpone to Section~\ref{sec:divisor}. Our task in Sections~\ref{sec:proof} and~\ref{sec:proof2} is then to put the above-mentioned tools together to prove Theorems~\ref{thm:E2} and~\ref{thm:E3}, respectively. Finally, in Appendix~\ref{app:A}, we give a proof of a slight generalization of the theory of exponent pairs that was needed in Section~\ref{sec:pointwise}, following work of Ivi\'c.

\subsection{Acknowledgments}

The first author was supported by the Emil Aaltonen foundation and worked in the Finnish Centre of Excellence in Randomness and Structures (Academy of Finland grant no. 346307). The second author was supported by Academy of Finland grant no. 340098, a von Neumann Fellowship (NSF grant \texttt{DMS-1926686}), and funding from European Union's Horizon
Europe research and innovation programme under Marie Sk\l{}odowska-Curie grant agreement No
101058904. We thank Kaisa Matom\"aki for helpful comments, discussions and corrections on the manuscript. We are also grateful to the anonymous referee for helpful corrections.

\section{Reduction to Hecke polynomials}
\label{sec:reduction}

Let $k \in \{2, 3\}$ be fixed, let $\varepsilon>0$ be sufficiently small and fixed, and let 
\begin{align*}
C=\begin{cases}
\C \quad \text{if } k=2,\\
\Co \quad \text{if } k=3.
\end{cases}    
\end{align*}
Let also
\begin{align}\label{eq_h}
h=\begin{cases}
(\log X)^{C}\quad &\text{if } k=2,\\
(\log X)(\log \log X)^{C}\quad &\text{if }k=3.
\end{cases}    
\end{align}
and
\begin{align}\label{eq_k2}
\begin{cases}
P_1 = (\log X)^{C-1}\quad \quad \quad \quad \quad \quad \quad \quad \quad   \textnormal{ if } k=2,\\  
P_1 = (\log \log X)^{C-1},\, P_2 = (\log X)^{\varepsilon^{-1}} \textnormal{ if } k=3.
\end{cases}
\end{align}

For a Gaussian integer $n$, let 
\begin{align}\label{eq_bn}
\beta_n=\begin{cases}\frac{1}{\N(n)},\textnormal{ if } n \equiv p_1\cdots p_k \textnormal{ with } p_1,\ldots, p_k\in \mathbb{Z}[i]^*\textnormal{ primes, } \N(p_j)\in [P_j^{1-\varepsilon}, P_j]\,\,\,\forall\, j<k,\\0,\quad\,\,  \textnormal{ otherwise,}\end{cases}    
\end{align}
where we recall that $a\equiv b$ means that $a=ub$ for some unit $u$.
To prove Theorems~\ref{thm:E3} and~\ref{thm:E2}, it suffices to prove the following. 

\begin{theorem}\label{thm_variance}
Let $k\in \{2,3\}$, and let $\varepsilon>0$ be small but fixed.  Let $h$ be as in~\eqref{eq_h}, let $P_i$ be as in~\eqref{eq_k2}, and let $\beta_n$ be as in~\eqref{eq_bn}. Then 
\begin{align}\label{eq_variance}
\int_{0}^{\pi/2} \left|\sum_{X < \N(n) \le 2X} \beta_n1_{\arg n \in [\theta, \theta + h/X)} - \frac{h}{\pi X/2}\sum_{X < \N(n) \le 2X} \beta_n \right|^2 \d \theta= o\left(\frac{h^2}{X^2(\log X)^2}\right).
\end{align}
In particular, for all $\theta\in [0,\pi/2)$ outside an exceptional set of measure $o_{X\to \infty}(1)$ we have 
\begin{align}\label{eq_lowerbound}
\sum_{\substack{X < \N(p_1\cdots p_k)\leq 2X\\P_j^{1-\varepsilon}\leq \N(p_j)\leq P_j\,\, \forall\,  j\leq k-1\\p_j\in \mathbb{Z}[i]^{*}\textnormal{ prime}}} 1_{\arg n \in [\theta, \theta + h/X)}\gg  \frac{h}{\log X}.  
\end{align}
\end{theorem}

Note that~\eqref{eq_lowerbound} follows from~\eqref{eq_variance} by  Chebyshev's inequality and the prime ideal theorem. Therefore, our task is to prove~\eqref{eq_variance}.

\subsection{Reduction to mean values of Hecke polynomials}

The distribution of Gaussian integers in narrow sectors is governed by the angular Hecke characters $\lambda^m$ with $m\in \mathbb{Z}$ and more precisely the \emph{Hecke polynomials} $\sum_{\N(n)\leq X}a_n\lambda^{m}(n)$.

Recall the definition of $\beta_n$ in~\eqref{eq_bn}.
For $m \in \mathbb{Z}$, define
$$F(m) \coloneqq \sum_{X < \N(n) \le 2X} \beta_n\lambda^m(n).$$ 

\begin{lemma}[Reduction to Hecke polynomials]
\label{lem:reduct}
Let $X$ be large, $h$ be as in~\eqref{eq_h}, and $F$ be as above. Assume that for some function $K=K(X)$ tending to infinity we have
\begin{align}\label{eq_Fbound}
\sum_{0 < |m| \le KX/h} |F(m)|^2 = o\left(\frac{1}{(\log X)^2}\right).    
\end{align}
Then
\begin{align*}
\int_{0}^{\pi/2} \left|\sum_{X < \N(n) \le 2X} \beta_n1_{\arg n \in [\theta, \theta + h/X)} - \frac{h}{\pi X/2}\sum_{X < \N(n) \le 2X} \beta_n \right|^2\d \theta = o\left(\frac{h^2}{X^2(\log X)^2}\right).
\end{align*}
\end{lemma}

\begin{proof}
Let $T = KX/h$. By a truncated Fourier expansion~\cite[Lemma 2.1]{harman-book} with $L = T$, $\delta = \frac{h}{\pi X}$, there exist constants $c_m^+, c_m^-$ for $0 < |m| \le T$ with $|c^{+}_{m}|, |c^{-}_{m}| \ll  h/X$ such that for any $\theta\in [0,\pi/2)$ we have
\begin{align*}
\frac{2h}{\pi X} +S^{-}(\theta)-\frac{1}{T+1}\leq 1_{\arg n \in (\theta, \theta + h/X)}\leq \frac{2h}{\pi X} +S^{+}(\theta)+\frac{1}{T+1},
\end{align*}
where
\begin{align*}
S^{\sigma}(\theta)=\sum_{0 < |m| \le T} c_m^{\sigma}e\left(\frac{m}{\pi/2}\left(\arg n - (\theta + \frac{h}{2X})\right)\right) 
\end{align*}
for $\sigma\in \{-,+\}$. Hence, as $\beta_n \ge 0$ for any $n$, we have
\begin{align*}
&\left|\sum_{X < \N(n) \le 2X} \beta_n\left(1_{\arg n \in (\theta, \theta + h/X)} - \frac{h}{\pi X/2}\right)\right| \\
&\le \max_{\sigma \in \{+, -\}} \left|\sum_{X < \N(n) \le 2X} \beta_n\left(\frac{\sigma}{T+1} + \sum_{0 < |m| \le T}c_m^{\sigma}e\left(\frac{m}{\pi/2}\left(\arg n - \left(\theta + \frac{h}{2X}\right)\right)\right)\right)\right| \\
&\leq  \frac{1}{T+1}\sum_{X < \N(n) \le 2X} \beta_n + \max_{\sigma \in \{+, -\}} \left|\sum_{0 < |m| \le T} c_m^{\sigma}e\left(\frac{-m(\theta + h/(2X))}{\pi/2}\right) \sum_{X < \N(n) \le 2X} \beta_n\lambda^m(n)\right|,
\end{align*}
and thus for some $\sigma \in \{+, -\}$ we have
\begin{align*}
&\int_{0}^{\pi/2} \left|\sum_{X < \N(n) \le 2X} \beta_n1_{\arg n \in [\theta, \theta + h/X)} - \frac{h}{\pi X/2}\sum_{X < \N(n) \le 2X} \beta_n \right|^2\d \theta\\
&\ll \frac{1}{T^2}\left(\sum_{X < \N(n) \le 2X} \beta_n\right)^2 + \int_{0}^{\pi/2} \left|\sum_{0 < |m| \le T} c_m^{\sigma}e\left(-m\frac{\theta + h/(2X)}{\pi/2}\right)\sum_{X < \N(n) \le 2X} \beta_n\lambda^m(n)\right|^2 \d\theta.
\end{align*}
After expanding out the square to obtain a double sum $\sum_{m, m'}$, the terms with $m \neq m'$ vanish as the integral over $\theta$ vanishes in this case. Thus, the previous expression is
\begin{align*}
\frac{1}{T^2}\left(\sum_{X < \N(n) \le 2X} \beta_n\right)^2 + \int_{0}^{\pi/2} \sum_{0 < |m| \le T} |c_m^{\sigma}|^2 \left|\sum_{X < \N(n) \le 2X} \beta_n\lambda^m(n)\right|^2 \d\theta,
\end{align*}
which we bound via the prime number theorem in $\mathbb{Z}[i]$ and the bound $|c_m^{\sigma}| \ll h/X$ as
\begin{align*}
\ll \frac{1}{T^2(\log X)^2} + \frac{h^2}{X^2} \sum_{0 < |m| \le T} |F(m)|^2.
\end{align*}
By the assumption~\eqref{eq_Fbound} and the choice of $T$, this is small enough. 
\end{proof}

In the rest of this paper, our task is to prove~\eqref{eq_Fbound}.

\subsection{Gaussian integers in narrow sectors}

For later use we give the following simple bound for the number of Gaussian integers in a given sector.

\begin{lemma}
\label{lem:countSector}
For any $N, v > 0$ and $n \in \mathbb{Z}[i]$ with $\N(n) \le N$ we have
\begin{align*}
|\{m \in \mathbb{Z}[i] : N(m) \le N,\,\, 0 < |\arg m - \arg n | < v/N\}| \ll v.
\end{align*}
\end{lemma}

\begin{proof}
We may assume without loss of generality that $0<\arg(n)\leq \pi/4$, $v\leq N/10$ and $\gcd(\text{Re}(n), \text{Im}(n)) = 1$. 
Note that we have
$$|x - y| \ll |\arctan(x)-\arctan(y)|$$
for $|x|, |y| \le 2$. Hence if we denote $n = a + bi, m = c + di$ with $(a,b)=1$ and $a,c>0$, then $0 < |\arg m - \arg n| < v/N$ implies that
$$0 < \Big|\frac{b}{a} - \frac{d}{c}\Big| \ll v/N.$$
Let $M = N(n)$. By $0 < \arg(n) \leq \pi/4$ we have $a \asymp \sqrt{M}$. We get
$$0 < \Big|ad - bc\Big| \ll \frac{v\sqrt{M}}{\sqrt{N}},$$
using the fact that $|c| \ll \sqrt{N}$.

Given $0 < |k| \ll v\sqrt{M}/\sqrt{N}$ with $k \in \mathbb{Z}$, any two distinct $(c_1, d_1), (c_2, d_2)$ with $ad_i - bc_i = k$ satisfy $c_1 \equiv c_2 \pmod{a}$ (since $(a,b)=1$). In particular, since $\Re(m) \ll \sqrt{N}$ and $a \asymp \sqrt{M}$, there are only $O(\lceil \sqrt{N}/\sqrt{M} \rceil)$ possibilities for $\Re(m)$. Furthermore, given $k$, the real part uniquely determines $\Im(m)$. Hence the number of $m$ with $\N(m) \le N$ and $0 < |\arg m - \arg n| < v/N$ is
$$\ll \frac{v\sqrt{M}}{\sqrt{N}} \left\lceil \frac{\sqrt{N}}{\sqrt{M}} \right\rceil \ll v,$$
as $\sqrt{M} \ll \sqrt{N}$.
\end{proof}

\section{Lemmas on Hecke polynomials}
\label{sec:preLem}

\subsection{Bounds for Hecke polynomials}

For the proofs of our main theorems, we shall need various estimates for Hecke polynomials $\sum_{\N(n)\leq N}a_n\lambda^{m}(n)$. 

\begin{remark}
Recall our convention that sums over $n$ are taken over $\mathbb{Z}[i]^{*}$. Hence, if $F(m)=\sum_{\N(n_1)\leq N_1}a_{n_1}\lambda^m(n_1)$ and $G(m)=\sum_{\N(n_2)\leq N_2}b_{n_2}\lambda^m(n_2)$, then $F(m)G(m)=\sum_{\N(n)\leq N_1N_2}c_n\lambda^m(n)$ with $c_n=\sum_{n\equiv n_1n_2}a_{n_1}b_{n_2},$ where we recall that $n\equiv a$ means $n=ua$ for some unit $u$.
\end{remark}

We begin with a simple mean value theorem for Hecke polynomials.

\begin{lemma}[Mean value theorem for Hecke polynomials]
\label{lem:MVT}
Let $N,T \ge 1$ and $F(m) = \sum_{\N(n) \leq N} a_n\lambda^m(n)$ with $a_n\in \mathbb{C}$. Then
$$\sum_{|m|\leq T} |F(m)|^2 \ll (N + T) \psum_{\N(n) \leq N} |a_n'|^2,$$
where 
\begin{align}\label{eq_an'}
a_n' \coloneqq \sum_{\substack{\N(v) \leq N\\\arg v = \arg n}} a_v    
\end{align} 
and $\psum$ signifies that the summation is only over primitive Gaussian integers, that is, those $a+bi\in \mathbb{Z}[i]^{\ast}$ with $(a,b)=1$. Moreover, we have
\begin{align}\label{eq_an2}
 \psum_{\N(n) \leq N} |a_n'|^2\ll N\sum_{\N(n)\leq N}\frac{|a_n|^2\tau(n)}{\N(n)}.   
\end{align}
\end{lemma}

\begin{proof}
See~\cite[Lemma 11.1 and Lemma 11.2]{harman-book}.
\end{proof}

The mean value theorem can be improved in the case of sparse coefficient sequences as follows.

\begin{lemma}[Improved mean value theorem]
\label{lem:IMVT}
Let $N,T\geq 1$ and $F(m) = \sum_{\N(n) \leq N} a_n\lambda^m(n)$ with $a_n\in \mathbb{C}$. Then
$$\sum_{|m|\leq T} |F(m)|^2 \ll T\sum_{\N(n)\leq N}|a_n|^2+T  \sum_{\substack{|\arg n_1 - \arg n_2| \le 1/T \\ \N(n_2), \N(n_2) \leq N\\n_1\neq n_2}} |a_{n_1}a_{n_2}|.$$
\end{lemma}

Recall from Section~\ref{sec:notation} that $\arg n$ is only defined up to multiples of $\pi/2$, and thus $|\arg n_1 - \arg n_2| \le 1/T$ is satisfied if $i^{k_1}n_1$ and $i^{k_2}n_2$ lie in the same narrow sector for some integers $k_1,k_2$.

\begin{proof}
Let $g(x) = \max(1 - 10T\|x\|, 0)$ for $x \in \mathbb{R}$ and 
$$\widehat{g}(m) = \int_{0}^{1} g(x)e(-mx) \d x$$
for $m \in \mathbb{Z}$. One has
$$\widehat{g}(m) = \frac{10T}{2\pi^2 m^2}\left(1 - \cos\left(\frac{2\pi m}{10T}\right)\right)$$
for $m \neq 0$ and $\widehat{g}(0) = 1/(10T)$. As $g$ is continuous and the Fourier coefficients $\widehat{g}(m)$ are absolutely summable, it follows that $g(x) = \sum_{m \in \mathbb{Z}} \widehat{g}(m)e(mx)$ for any $x$. Note furthermore that $\widehat{g}(m) \ge 0$ for all $m$ and $\widehat{g}(m) \gg 1/T$ for $|m| \leq T$. Hence,
\begin{align*}
\sum_{|m| \le T} |F(m)|^2 &\ll T \sum_{m \in \mathbb{Z}} \widehat{g}(m)|F(m)|^2\\
&=T \sum_{\N(n_1), \N(n_2) \le N} a_{n_1}\overline{a_{n_2}} \sum_{m \in \mathbb{Z}} \widehat{g}(m)\lambda^m(n_1)\overline{\lambda^m(n_2)}\\
&\le T \sum_{\N(n_1), \N(n_2) \le N} |a_{n_1}a_{n_2}| \left|\sum_{m \in \mathbb{Z}} \widehat{g}(m)e\left(m\frac{\arg n_1 - \arg n_2}{\pi/2}\right)\right|\\
&=T \sum_{\N(n_1), \N(n_2) \le N} |a_{n_1}a_{n_2}|g\left(\frac{\arg n_1 - \arg n_2}{\pi/2}\right)\\
&\le T \sum_{\substack{|\arg n_1 - \arg n_2| \le 1/T \\ \N(n_1), \N(n_2) \leq N}} |a_{n_1}a_{n_2}|,
\end{align*}
as desired.
\end{proof}

\begin{lemma}[A pointwise bound]
\label{lem:vinogradov}
Let $2\leq N\leq N'\leq 2N$, and let $k\geq 1$ be a fixed integer. Let
\begin{align}\label{eq:pointwise}
P(m) = \sum_{N < \N(n_1 \cdots n_k) \leq N'} \frac{g_1(n_1) \cdots g_k(n_k)}{\N(n_1) \cdots \N(n_k)}\lambda^m(n_1 \cdots n_k)
\end{align}
where each $g_i$ is either the M\"obius function (of $\mathbb{Z}[i]$), the characteristic function of Gaussian primes, the von Mangoldt function (of $\mathbb{Z}[i]$), the constant function $1$ or the log-norm function $n\mapsto \log \N(n)$. We have
$$|P(m)| \ll \exp(-(\log N)^{1/10})$$
when $0 < |m| \le \exp((\log N)^{10/9})$.
\end{lemma}

\begin{proof}
By writing the sum over $\N(n)\in (N,N']$ as the difference of a sum over $\N(n)\leq N'$ and a sum over $\N(n)\leq N$, we may assume that the summation in~\eqref{eq:pointwise} is over $\N(n_1\cdots n_k)\leq N$. Moreover, we may assume without loss of generality that $N-1/2$ is an integer.

Consider first the case $k = 1$, $g_1 = 1$. Denote $c=1/\log N$. By the truncated Perron formula 
$$\int_{c-iT}^{c+T}y^s\,\frac{\d s}{s}=1_{y>1}+O\left(\frac{y^c}{T}\right)$$ 
for $y \neq 1$ and the simple bound $L(1+1/\log N,\lambda^m) \ll L(1+1/\log N,1) \ll \log N$, we see that
\begin{align*}
P(m) &= \sum_{\N(n) \leq N} \frac{\lambda^m(n)}{\N(n)} = \frac{1}{2\pi i}\sum_{n \in \mathbb{Z}[i]^{*}} \frac{\lambda^m(n)}{\N(n)} \int_{c - iN}^{c + iN} \frac{N^s}{\N(n)^s}\frac{\d s}{s} + O\left(\frac{\log N}{N}\right)\\
&=\frac{1}{2\pi i} \int_{c - iN}^{c + iN} L(s+1,\lambda^m)N^s\frac{\d s}{s}+ O\left(\frac{\log N}{N}\right).
\end{align*}
Move the integral to the line $\re(s) = -\sigma = -(\log (N + |m|))^{-3/4}$, noting that there is no pole as $m \neq 0$. Let $\mathcal{C}$ be the rectangle having the line segments $[c-iN,c+iN]$ and $[-\sigma-iN,-\sigma+iN]$ as two of its sides. By~\cite[Theorems 1 and 6]{coleman-zeta} (applied with $\mathbf{f} = 1$ and $V = O(m)$) we have $$|L(s+1,\lambda^m)|\ll (\log |m|)^{2/3}$$
for $s\in \mathcal{C}$, 
so the error arising from moving the integral is $O((\log |m|)^{2/3}/N)$, and we have
\begin{align*}
\left|\frac{1}{2\pi i}\int_{-\sigma - iN}^{-\sigma + iN} L(s+1,\lambda^m)N^s\frac{\d s}{s}\right| &\ll N^{-\sigma}\int_{-\sigma - iN}^{-\sigma + iN} (\log |m|)^{2/3}\frac{|\d s|}{|s|}\\
&\ll N^{-\sigma}(\log |m|)^{2/3}(\log N).
\end{align*}
Finally, note that this is $\ll \exp(-(\log N)^{1/10})$ as long as $0<|m| \leq \exp((\log N)^{10/9})$, by our choice of $\sigma$.

The cases with $k = 1$ and $g_i$ being equal to the M\"obius function, the indicator function of Gaussian primes or the log-norm function are handled similarly, noting that if $\mathcal{C}$ is as above, for $s\in \mathcal{C}$ we have
\begin{align*}
\frac{1}{|L(s+1,\lambda^m)|},\, |L'(s+1,\lambda^m)|,\, \frac{|L'(s+1,\lambda^m)|}{|L(s+1,\lambda^m)|} \ll (\log |m|)^{O(1)}    
\end{align*}
by an analogue of the Vinogradov--Korobov zero-free region for $L(s,\lambda^m)$~\cite[Theorem 2]{coleman-zeta}. 

Finally, the cases $k \ge 2$ follow from the case $k = 1$ by decomposing the sum to dyadic intervals $\N(n_i) \in (N_i, 2N_i]$, for fixed $N_1, \ldots , N_k$ summing over the variable $n_i$ for which $N_i$ is largest using the case $k = 1$, and applying the triangle equality to the sum over the other variables.
\end{proof}

\subsection{Heath-Brown's decomposition}

Next, we give a suitable version of Heath-Brown's identity for Hecke polynomials.

\begin{definition}[Smooth Hecke polynomials]\label{def:smooth} We say that a Hecke polynomial $M(m) = \sum_{\N(n) \sim M} a_n\lambda^m(n)/\N(n)$ is \emph{smooth} if for some interval $I\subset [M,2M]$ we have $a_n=1_{I}(\N(n))$ for all $n$ or $a_n = 1_{I}(\N(n))\log \N(n)$ for all $n$.
\end{definition}

\begin{lemma}[Heath-Brown's decomposition]
\label{lem:HB}
Let an integer $k \ge 1$ and a real number $B\geq 1$ be fixed, and let $T \ge 3$. Let $P(m) = \sum_{P < \N(p) \le P'} \lambda^m(p)/\N(p)$ with $ P\leq P'<2P$. Then, for some constant $D=D_{k,B}\geq 1$, we have the decomposition
\begin{align*}
|P(m)| \ll |G_1(m)| + \cdots + |G_L(m)|+E(m) \quad \textnormal{ for all  } 0 < |m| \le T,
\end{align*}
where
\begin{enumerate}
    \item $\sum_{|m|\leq T}|E(m)|^2\ll (T/P+1)(\log P)^{-B}$.
    \item $L \le (\log P)^{D}$.
    \item Each $G_j(m)$ is of the form
$$G_j(m) = \prod_{i \le J_j} M_i(m), \quad J_j \le 2k,$$
with $M_i(m)=\sum_{M_i \leq  \N(n)< 2M_i}a_{n,i}\lambda^m(n)/N(n)$ being Hecke polynomials (which depend on $j$) with $M_i\geq 1$, $|a_{n,i}|\leq 1$ and $P\exp(-(\log T)^{19/20})\leq M_1 \cdots M_{J_j} \leq 2P$. Additionally, $|M_i(m)|\ll \exp(-(\log(2M_i))^{1/10})$ for all $0<|m|\leq T$, and the Hecke polynomial $M_i(m)$ is smooth if $M_i > (2P)^{\frac{1}{k}}$.
\end{enumerate}
\end{lemma}

\begin{proof} We may assume that $P\geq \exp((\log T)^{19/20})$, since otherwise the claim is trivial with $E(m) \equiv 0$, $L = 1$, $J_1 = 1$ and $G_1(m) = M_1(m) \equiv 1$.

By splitting the sum $P(m)$ over $(P,P']$ into subsums over intervals of the form $(Q,Q(1+(\log P)^{-B'})]$ with $B'$ large enough (and one shorter interval) and applying the triangle inequality, it suffices to prove the claim with $P'\leq P(1+(\log P)^{-B'})$. 

Now, we write
\begin{align*}
P(m)=\frac{1}{\log P}\widetilde{P}(m)+E_1(m),    
\end{align*}
where
\begin{align*}
\widetilde{P}(m)=\sum_{P<N(n)\leq P'}\Lambda(n)\frac{\lambda^m(n)}{\N(n)},\quad E(m)=\sum_{P<\N(n)\leq P'}\left(1_{n\textnormal{ prime}}-\frac{\Lambda(n)}{\log P}\right)\frac{\lambda^m(n)}{\N(n)}.    
\end{align*}
By the mean value theorem (Lemma~\ref{lem:MVT}) and the prime number theorem for Gaussian integers with classical error term, $E_1(m)$ satisfies property (1). By writing $\widetilde{P}(m)$ as a difference of sums over $[1,P']$ and $[1,P]$, we see that it suffices to prove the claim for sums of the form
\begin{align}\label{eq:lambdaHB3}
\sum_{\substack{n \in \mathbb{Z}[i] \\ \N(n) \le x}} \Lambda(n)f(n),
\end{align}
where 
$P\leq x\leq 2P$ and $f(n)=\lambda^m(n)/\N(n)$.

Let $\delta=(\log P)^{-A}$, where $A$ is large enough in terms of $B,k$. Fix an integer $m\in [-T,T]\setminus\{0\}$ as in the lemma. 
By Heath-Brown's identity in $\mathbb{Z}[i]$ (which is derived precisely as in the case of $\mathbb{Z}$) and $(1+\delta)$-adic decomposition,~\eqref{eq:lambdaHB3}
may be written as
$$\sum_{\substack{M_1, \ldots , M_{2k}\leq x\\M_{k+1},\ldots, M_{2k}\leq (2P)^{1/k}\\M_i=(1+\delta)^{\ell_i},\ell_i\geq 0}} c_{M_1, \ldots , M_{2k}} \sum_{\substack{M_i\leq \N(n_i)<(1+\delta)M_i, 1 \le i \le 2k \\ \N(n_1\cdots n_{2k})\leq x}} (\log \N(n_1))\prod_{j=k+1}^{2k} \mu(n_{j})\cdot f(n_1 \cdots n_{2k}),$$
where the constants $c_{M_1, \ldots , M_{2k}}$ are bounded in magnitude by $O_k(1)$. By the triangle inequality,~\eqref{eq:lambdaHB3} is thus bounded by
\begin{align*}
 \ll \sum_{\substack{M_1, \ldots , M_{2k}\leq x\\M_{k+1},\ldots, M_{2k}\leq (2P)^{1/k}\\M_i=(1+\delta)^{\ell_i},\ell_i\geq 0}} \left| \sum_{\substack{M_i\leq \N(n_i)<(1+\delta)M_i, 1 \le i \le 2k \\ N(n_1\cdots n_{2k})\leq x}} (\log \N(n_1))\prod_{j=k+1}^{2k} \mu(n_{j})\cdot f(n_1 \cdots n_{2k})\right|.   
\end{align*}

Write $g_1(n) = \log \N(n)$, $g_i(n) = 1$ for $2 \le i \le k$ and $g_i(n) = \mu(n_i)$ for $k+1 \le i \le 2k$, and let $P_0\in \{P,2P\}$. Then, from the above we deduce that
\begin{align}\label{eq:lambdaHB2}\begin{split}
\left|\sum_{\N(n) \le P_0} \frac{\Lambda(n)\lambda^m(n)}{\N(n)}\right| &\ll \Sigma_j(m)+E_{2,j}(m),\\
\Sigma_j(m)&=\sum_{\substack{M_{k+1},\ldots, M_{2k}\leq (2P)^{1/k}\\M_1\cdots M_{2k}\leq P_0\\M_1\cdots M_{2k}\geq \delta^{2k+10}P\\M_i=(1+\delta)^{\ell_i},\ell_i\geq 0}} \prod_{i=1}^{2k}\left|\sum_{M_i\leq \N(n_i)<(1+\delta)M_i} \frac{g_i(n_i)\lambda^m(n_i)}{\N(n_i)} \right|,\end{split}
\end{align}
where $E_{2,j}(m)$ arises from removing the summation condition $\N(n_1\cdots n_{2k})\leq (1+\delta)^j P$, and from inserting the condition $M_1\cdots M_{2k}\geq \delta^{2k+10}P$. One easily sees from the mean value theorem that $E_{2,j}(m)$ satisfies condition (1). We can further estimate the product in~\eqref{eq:lambdaHB2} by bounding trivially as $\ll 1$ all those terms for which $M_i\leq \exp((\log T)^{19/20}/(4k))$; the product of the remaining $M_i$ is $\gg P\exp(-(\log T)^{19/20})$.

We have now arrived at the desired decomposition, since the Hecke polynomials $M_i(m)$ with coefficients $g_i(m)/\N(n)$ in the definition of $\Sigma_j(m)$ satisfy the bound $|M_i(m)|\ll \exp(-(\log(2M_i))^{1/10})$ by Lemma~\ref{lem:vinogradov} since $M_i\geq \exp((\log T)^{19/20}/(4k))$, and additionally if $M_i>(2P)^{1/k}$, then $i\leq k$, so $M_i(m)$ is smooth.
\end{proof}

\section{Pointwise bounds}\label{sec:pointwise}

The goal of this section is to establish Proposition~\ref{prop:zeta}, a pointwise bound for smooth Hecke polynomials. For stating the result, we need the notion of exponent pairs.

\subsection{Exponent pairs}

We define exponent pairs following Ivi\'c~\cite[Chapter 2.3]{ivic-book}, but impose slightly milder conditions on the derivatives of the phase function, since the functions we apply the theory to do not quite satisfy the original definition.

\begin{definition}  Let $A,B,M\geq 1$, and let $R\geq 1$ be an integer. Let $I\subset [B,2B]$ be an interval. We define the set $\mathscr{F}_I(A,B,M,R)$ as the set of those functions $f$ on $I$ that satisfy
\begin{enumerate}
    \item $f \in C^{\infty}(I)$. 
    \item For all $t\in I$ and all integers $1\leq r\leq R$
    \begin{align*}
     M^{-1}AB^{1-r}\leq |f^{(r)}(t)|\leq  MAB^{1-r}.   
    \end{align*}
\end{enumerate}

\end{definition}

\begin{definition}\label{def_gk}
We say that a pair of real numbers $(\kappa,\lambda)$ with $0\leq \kappa\leq 1/2\leq \lambda\leq 1$ is an \emph{exponent pair} if the following holds for some integer $R\geq 1$. 

For any $A,B,M\geq 1$ and any $f\in \mathscr{F}_I(A,B,M,R)$ with $I\subset [B,2B]$ an interval, we have
\begin{align*}
\left|\sum_{n\in I}e(f(n))\right|\ll_{\kappa,\lambda}M^{O_{\kappa,\lambda}(1)}  A^{\kappa}B^{\lambda}.  
\end{align*}
We call the least integer $R$ with this property the \emph{degree} of $(\kappa,\lambda)$.
\end{definition}

The difference between our definition and~\cite{ivic-book} is that there only the case $M = O(1)$ is considered (and the derivative bound is assumed for all $r$).

Trivially, $(0,1)$ is an exponent pair. We recall the $A$ and $B$ processes that allow us to generate infinitely many exponent pairs from a given pair. 

\begin{lemma}[$A$ and $B$ processes]\label{le_gk}
\begin{enumerate}[(A)]
    \item If $(\kappa,\lambda)$ is an exponent pair, so is
    \begin{align*}
   A(\kappa,\lambda):= \left(\frac{\kappa}{2\kappa+2},\frac{1}{2}+\frac{\lambda}{2\kappa+2}\right).    
    \end{align*}
    \item If $(\kappa,\lambda)$ is an exponent pair with $\kappa+2\lambda\geq 3/2$, then
    \begin{align*}
    B(\kappa,\lambda):=\left(\lambda-\frac{1}{2},\kappa+\frac{1}{2}\right)    
    \end{align*}
    is also an exponent pair.
\end{enumerate}
\end{lemma}

\begin{proof} The claim is a slight generalization of~\cite[Lemmas 2.8 and 2.9]{ivic-book} (see also~\cite{phillips}), since our conditions for exponent pairs allow $M$ to be unbounded. The proof works similarly in our case; see Appendix~\ref{app:A} for the details.
\end{proof}

In our proofs we will use the following exponent pairs.

\begin{lemma}
\label{lem:exp_pairs}
The pairs
\begin{align*}
(\kappa_1, \lambda_1) \ = \ &(0.02381, 0.8929), \\
(\kappa_2, \lambda_2) \ = \ &(0.05, 0.825)
\end{align*}
are exponent pairs.
\end{lemma}

\begin{proof}
The first pair is obtained from the pair
$$AAABAAB(0, 1) = (0.0238095\ldots, 0.892857\ldots)$$
by rounding the entries up. The second pair is $AABAAB(0, 1)$.
\end{proof}

\subsection{Pointwise bounds}

For the proof of Proposition~\ref{prop:zeta} below, we will need to evaluate and estimate derivatives of $x\mapsto \arctan(y/x)$. 

\begin{lemma}\label{le_derivative}
Let $n\geq 1$ be an integer and let $y>0$. We have
\begin{align*}
\frac{\partial^n}{\partial x^n}\arctan\left(\frac{y}{x}\right)=(-1)^n\frac{(n-1)!}{(x^2+y^2)^{n}}\Im((x+iy)^{n}).
\end{align*}
\end{lemma}

\begin{proof}
We have
\begin{align*}
\frac{\partial}{\partial x}\arctan\left(\frac{y}{x}\right)=-\frac{y}{x^2+y^2}, 
\end{align*}
which agrees with the claim for $n=1$. Moreover, for $n\geq 1$, we have
\begin{align*}
\frac{\partial}{\partial x}\left(\frac{1}{(x^2+y^2)^{n}}\Im((x+iy)^{n})\right)&=\frac{\partial}{\partial x}\left(\Im((x-iy)^{-n})\right)\\
&=\Im\left(\frac{\partial}{\partial x}(x-iy)^{-n}\right)\\
&=-n\Im((x+iy)^{-n-1})\\
&=-n\left(\frac{1}{(x^2+y^2)^{n+1}}\Im((x+iy)^{n+1})\right).
\end{align*}
The claim now follows by induction. 
\end{proof}

\begin{proposition}[Pointwise bound for smooth Hecke sums]
\label{prop:zeta}
Let $N, N' \geq 2$ with $N \leq N'\leq 2N$, and let  $m\neq 0$ be an integer. \begin{enumerate}[(i)]

    \item For any fixed exponent pair $(\kappa, \lambda)$ and any fixed $\varepsilon>0$ small enough in terms of $(\kappa,\lambda)$, there exists an integer $R$ such that the following holds. If $I\subset [0,\pi/2]$ is an interval such that all the real solutions of $\Im((1+i\tan(t))^k)=0$ with $k=1,\ldots, R$ have distance $\geq N^{-\varepsilon^2}$ to $I$, then we have 
\begin{align}\label{eq_lambdasum}
\left|\sum_{\substack{N < \N(n) \le N'\\\arg(n)\in I}} \lambda^m(n)\right| \ll |m|^{\kappa}N^{(\lambda-\kappa+1)/2+\varepsilon} +N^{3/4+o(1)}.
\end{align}
and if either $a_n = 1/\N(n)$ or $a_n = (\log \N(n))/\N(n)$, then
\begin{align*}
\left|\sum_{\substack{N < \N(n) \le N'\\\arg(n)\in I}} a_n\lambda^m(n) \right| \ll (\log N)|m|^{\kappa}N^{(\lambda - \kappa - 1)/2+\varepsilon} + N^{-1/4+o(1)}.
\end{align*}

\item We have 
\begin{align}\label{eq_lambdasum3}
\left|\sum_{N < \N(n) \le N'} \lambda^m(n)\right| \ll (\log N)^5|m|^{1/3} N^{1/2}+N^{5/8+o(1)},
\end{align}
and if either $a_n = 1/\N(n)$ or $a_n = (\log \N(n))/\N(n)$, then
\begin{align*}
\left|\sum_{N < \N(n) \le N'} a_n \lambda^m(n) \right| \ll (\log N)^6|m|^{1/3}N^{-1/2}+N^{-3/8+o(1)}.
\end{align*}
\end{enumerate}
\end{proposition}

\begin{remark}
One may at first wonder about the need in part (i) to exclude some small sectors. The estimate should be true even without it, but our proof method does not work without this condition. The exponential sum~\eqref{eq_lambdasum} is interpreted as a two-dimensional exponential sum involving the phase function $\frac{m\arctan(y/x)}{\pi/2}$, and to apply the theory of exponent pairs to this function we need to know that its derivatives do not vanish, so we need to exclude certain narrow sectors of the $(x,y)$-plane inside of which the derivatives of some bounded order do vanish. See also Remark~\ref{rem:arctan}.
\end{remark}

\begin{remark}\label{rem:pointwise} Part (i) of the lemma gives us explicit power savings in the range $|m|^{\varepsilon}\leq N\leq |m|$ (using the exponent pair $A^kB(0,1)=(\frac{1}{2^{k+2}-2},1-\frac{k+1}{2^{k+2}-2})$ with $k$ large enough in terms of $\varepsilon$). The most critical case for the proof of our main theorem is $N\in [|m|^{1/2},|m|^{2/3}]$; in this range the estimate of part (ii) is trivial. However, when $N\geq |m|^{1-\delta}$ for somewhat small $\delta$, part (ii) is stronger.
\end{remark}

\begin{proof}
(ii)  By partial summation, it suffices to prove the first claim in part (ii). By writing the sum over $N < \N(n)\leq N'$ as a difference of two sums, it suffices to prove~\eqref{eq_lambdasum3} with the summation condition $N < \N(n)\leq N'$ changed to $\N(n)\leq N$. 

We may assume $|m|\leq N^{3/2}$, since otherwise the claim is trivial. Let $T=|m|+N^{3/8}$. By a truncated form of Perron's formula (\cite[Corollary 2.4 of Section II.2]{tenenbaum} applied to the sequence $a_k=\sum_{\N(n)=k}\lambda^m(n)$), we have
\begin{align*}
\sum_{\N(n) \le N} \lambda^m(n) =\frac{1}{2\pi i}\int_{1+1/\log N-iT}^{1+1/\log N+iT}L(s,\lambda^m)\frac{N^s}{s}\d s+O\left(N^{o(1)}+\frac{N^{1+o(1)}}{T}\right).    
\end{align*}
We shift the integration to the line $\Re(s)=1/2$ and use the estimate
\begin{align}\label{eq:zeta}
|L(\sigma+it,\lambda^m)|\ll (|m|+T)^{2(1-\sigma)/3}(\log(|m|+T))^4,
\end{align}
which follows from~\cite{sohne} and the Phragm\'en--Lindel\"of principle~\cite[Appendix A.8]{ivic-book}, to bound the horizontal integrals. We obtain
\begin{align*}
\sum_{\N(n) \le N} \lambda^m(n) =\frac{1}{2\pi i}\int_{1/2-iT}^{1/2+iT}L(s,\lambda^m)\frac{N^s}{s}\d s+O\left(N^{o(1)}+\frac{N^{1+o(1)}}{T}\right).    
\end{align*}
Using~\eqref{eq:zeta} again to bound the integral, the claim follows.

(i) By partial summation, it suffices to prove the first claim in part (i). By writing the sum over $N < \N(n)\leq N'$ as a difference of two sums, it suffices to prove~\eqref{eq_lambdasum} with the summation condition $N < \N(n)\leq N'$ changed to $\N(n)\leq N$. Furthermore, we may assume that $|m|\geq N^{3/4}$, since otherwise the claim follows directly from part (ii). 

Note that 
$$\lambda(x+iy)=e\left(\frac{1}{\pi/2}\arctan(y/x)\right)$$
if $x\neq 0$. Note also that $\lambda(x+iy)=\overline{\lambda}(y+ix)$. Lastly, observe that the contribution of $n$ of the form $x+ix$ or $x+0i$ to the left-hand side of~\eqref{eq_lambdasum} is $\ll N^{1/2}$, which is admissible. Hence, it suffices to prove~\eqref{eq_lambdasum} with the sum restricted to the region $n=x+iy$ with $1 \le y \le x$.
Thus, our task is to bound
\begin{align*}
S&=\sum_{\substack{x, y \in \mathbb{Z} \\ 1 \le y \leq x \\x^2 + y^2 \le N\\\arctan(y/x)\in I}} e\left(\frac{m}{\pi/2}\arctan\left(\frac{y}{x}\right)\right).
\end{align*}
We can write the condition $\arctan(y/x)\in I$ in the form $x \in yJ$, where $J=\{\frac{1}{\tan t}:\,\,t\in I\setminus\{0\}\}\subset (0,\infty)$ is an interval and $yJ:=\{yt:\,\, t\in J\}$.

By dyadic decomposition, we can bound
\begin{align}\label{eq:SX}
|S|\leq \sum_{\substack{1\leq X\leq \sqrt{N}\\X=2^k,\, k\in \mathbb{N}}}|S_X|,   
\end{align}
where
\begin{align}\label{eq:SX2}
S_X&:=\sum_{\substack{1\leq y\leq \min\{2X,\sqrt{N}\}\\y\in \mathbb{Z}}}\sum_{\substack{y\leq x\leq \sqrt{N-y^2}\\x\sim X\\x\in yJ \cap \mathbb{Z}}}e\left(\frac{m}{\pi/2}\arctan\left(\frac{y}{x}\right)\right).
\end{align}

Now, for a given $y\geq 1$, consider the function 
\begin{align*}
f(x)=\frac{m}{\pi/2}\arctan\left(\frac{y}{x}\right).   
\end{align*}
By Lemma~\ref{le_derivative}, for any $n\geq 1$ we have
\begin{align*}
f^{(n)}(x)=(-1)^{n}(n-1)!\cdot \frac{2m}{\pi}\frac{\Im((x+iy)^{n})}{(x^2+y^2)^{n}}.    
\end{align*}
Expanding out $(x+iy)^n$ and using the triangle inequality, for $1\leq y\leq x$ we obtain
\begin{align*}
|f^{(n)}(x)|\ll_n \frac{|m|y}{x^{n+1}}.    
\end{align*}
On the other hand, $\textnormal{Im}((x+iy)^n)=yx^{n-1}P_n(y/x)$ for some polynomial $P_n(t)$ of degree $\leq n-1$ and constant coefficient $n$, and the zeros of $P_n$ in the region $[0,1]$ are precisely the zeros of $\textnormal{Im}((1+it)^n)=0$. By the assumption on $I$, For any $x\in yJ$ and $1\leq n\leq R$, the number $y/x$ is distance $\gg N^{-\varepsilon^2}$ away from any solution to $\textnormal{Im}((1+it)^n)=0$, so we have $|P_n(y/x)|\gg_n N^{-n\varepsilon^2}$ when $x\in yJ$ (since if $P(t)$ is a monic polynomial of degree $n$ and $t_0$ is at least $\delta > 0$ away from all of the roots $\alpha_i$ of $P$, then $|P(t_0)| = \prod_{1\leq i\leq n} |t_0 - \alpha_i| \ge \delta^n$). Therefore, for $1\leq n\leq R$ and $x\in yJ$ we have 
\begin{align*}
|f^{(n)}(x)|\gg \frac{|m|y}{x^{n+1}N^{R\varepsilon^2}}.     
\end{align*}

We conclude that $f\in \mathscr{F}_{yJ}(A,B,O(N^{R\varepsilon^2}),R)$, where $A=|m|y/X^2$ and $B=X$. We have $A\geq 1$ if $y\geq X^2/|m|$, and in the case $y < X^2/|m|$ we use the trivial estimate for the inner sum in~\eqref{eq:SX2}. Hence, by the definition of exponent pairs, if $\varepsilon>0$ is small enough we have, using $X \leq N^{1/2}$ and $|m| \geq N^{3/4}$,
\begin{align*}
|S_X|\ll \sum_{\substack{X^2/|m|\leq y\leq \min\{2X,\sqrt{N}\}\\y\in \mathbb{Z}}}\left(\frac{|m|y}{X^2}\right)^{\kappa}X^{\lambda}N^{\varepsilon/2}+\frac{X^2}{|m|}\cdot X\ll |m|^{\kappa}X^{\lambda-\kappa+1}N^{\varepsilon/2}+N^{3/4}.  \end{align*}
Substituting this to~\eqref{eq:SX} we see that
\begin{align*}
|S|\ll |m|^{\kappa}\sqrt{N}^{\lambda-\kappa+1+\varepsilon}+N^{3/4+o(1)},    
\end{align*}
as desired. 
\end{proof}

\begin{remark}\label{rem:arctan}
Note that it was important in the proof of Proposition~\ref{prop:zeta}(i) that $I$ contains no solutions to $\Im((1+i\tan(t))^k)=0$. Indeed, otherwise the inner sum over $x$ in~\eqref{eq:SX2} would contain zeros of the $k$th derivative of the phase function $f$, so the theory of exponent pairs would not be applicable. 
\end{remark}

\section{Large value estimates and a Hal\'asz--Montgomery type inequality}\label{sec:HM}

\subsection{Hal\'asz--Montgomery type estimate}

In this section, we employ Proposition~\ref{prop:zeta} to establish large value theorems for Hecke polynomials that will be key to our arguments in Section~\ref{sec:proof}. These large value estimates are based on the following estimate of Hal\'asz--Montgomery type. 

\begin{proposition}[Hal\'asz--Montgomery type inequality with exponent pairs]
\label{prop:HM}
Let $N,T\geq 2$, and let $F(m)=\sum_{\N(n)\leq N}a_n\lambda^m(n)$ with $a_n\in \mathbb{C}$. Let $\mathcal{T} \subset [-T, T] \cap \mathbb{Z}$.
\begin{enumerate}[(i)]
\item Let $(\kappa, \lambda) \neq (0, 1/2)$ be a fixed exponent pair,  and let $J$ be a large enough integer. Let $\varepsilon>0$ be small but fixed, and suppose that $a_n=0$ whenever $\arg(n)$ is within distance $N^{-\varepsilon^2}$ of some real solution to $\Im((1+i\tan(t))^k)=0$ with $k=1,\ldots, J$. Then, we have
$$\sum_{m \in \mathcal{T}} |F(m)|^2 \ll (N + |\mathcal{T}|T^{\kappa}N^{(\lambda - \kappa + 1)/2+\varepsilon})\sum_{\N(n) \le N} |a_n|^2.$$

\item We have
$$\sum_{m \in \mathcal{T}} |F(m)|^2 \ll (N + |\mathcal{T}|T^{1/3}N^{1/2})(\log (N+T))^4\sum_{\N(n) \le N} |a_n|^2.$$
\end{enumerate}
\end{proposition}

\begin{proof}(i) We may assume that $N\leq T^{1+\varepsilon}$ for any fixed $\varepsilon>0$, since otherwise the claim follows directly from the mean value theorem (Lemma~\ref{lem:MVT}).

Let $J$ be an integer large enough in terms of $(\kappa,\lambda)$, and let $\mathcal{S}$ be the set of complex numbers whose argument is at least $N^{-\varepsilon^2}$ away from any solution to $\Im((1+i\tan(t))^k)=0$ with $k=1,\ldots, J$. Let $\mathcal{T}=\{m_r\}_{r\leq R}$ with $R=|\mathcal{T}|$. We may assume that $R\leq T$, as otherwise the claim follows from Lemma~\ref{lem:MVT}. By the duality principle (see e.g.~\cite[Chapter 7]{iw-kow}), the statement is equivalent to the claim that, for any complex numbers $c_r$ and distinct integers $m_r\in [-T,T]$, we have
\begin{align}\label{eq_hal_mon}
\sum_{\substack{\N(n)\leq N\\n\in \mathcal{S}}}\left|\sum_{r\leq R}c_r\lambda^{m_r}(n)\right|^2\ll (N + RT^{\kappa}N^{(\lambda - \kappa + 1)/2+\varepsilon})\sum_{r\leq R} |c_r|^2.    
\end{align}
Opening the square and using $|c_r||c_s|\ll |c_r|^2+|c_s|^2$, the left-hand side of~\eqref{eq_hal_mon} becomes
\begin{align}\label{eq:mrms1}
\ll N\sum_{r\leq R}|c_r|^2+\sum_{s \leq R} \sum_{\substack{r \leq R\\r\neq s}}|c_r|^2\left|\sum_{\substack{\N(n)\leq N\\n\in \mathcal{S}}}\lambda^{m_r-m_s}(n)\right|.    
\end{align}
By Proposition~\ref{prop:zeta}(i) and the fact that $\mathcal{S}$ is a union of $O(1)$ intervals, for $r\neq s$ we have
\begin{align}\label{eq:mrms}
\left|\sum_{\substack{\N(n)\leq N\\n\in \mathcal{S}}}\lambda^{m_r-m_s}(n)\right|\ll T^{\kappa}N^{(\lambda-\kappa+1)/2+\varepsilon}+N^{3/4+o(1)}.    
\end{align}
Note that by definition $0\leq \kappa\leq 1/2\leq \lambda$ for any exponent pair $(\kappa,\lambda)$, and moreover we assumed that $\kappa>0$ or $\lambda>1/2$.
Since $N\leq T^{1+\varepsilon}$, we thus have $$T^{\kappa}N^{(\lambda-\kappa+1)/2}\geq N^{(\kappa+\lambda+1)/2-O(\varepsilon)}\geq N^{3/4+\varepsilon}$$
for $\varepsilon>0$ small enough. Hence, the second term on the right of~\eqref{eq:mrms} can be removed, and the claim follows by substituting~\eqref{eq:mrms} into~\eqref{eq:mrms1}.

(ii) The proof of this part is identical, except that we use Proposition~\ref{prop:zeta}(ii) instead of Proposition~\ref{prop:zeta}(i) and do not restrict to $n \in \mathcal{S}$.
\end{proof}

\subsection{Large value estimates}

We now deduce from Proposition~\ref{prop:HM} a large value estimate, refined using Huxley's subdivision trick.

\begin{lemma}[A large value estimate]
\label{lem:hux_LVT}
Let $N, T \ge 2, V>0$, and $F(m) = \sum_{\N(n) \le N} a_n \lambda^m(n)$ with $a_n \in \mathbb{C}$. Write $G = \sum_{\N(n) \le N} |a_n|^2$, and let $\mathcal{T}$ denote the set of $m \in [-T, T] \cap \mathbb{Z}$ for which $|F(m)| \ge V$.
\begin{enumerate}[(i)]
\item Let $(\kappa, \lambda) \neq (0, 1/2)$ be a fixed exponent pair,  and let $J$ be a large enough integer. Let $\varepsilon>0$ be small but fixed, and suppose that $a_n=0$ whenever $\arg(n)$ is within distance $N^{-\varepsilon^2}$ of some real solution to $\Im((1+i\tan(t))^k)=0$ with $k=1,\ldots, J$. Then, we have
\begin{align*}
|\mathcal{T}| \ll N^{\varepsilon}\left(GNV^{-2}+TN^{1+(\lambda - \kappa + 1)/(2\kappa)}(GV^{-2})^{1+1/\kappa}\right).
\end{align*}
\item We have
\begin{align*}
|\mathcal{T}| \ll (\log(N+T))^{O(1)}\left(GNV^{-2}+TN^{5/2}(GV^{-2})^{4}\right).
\end{align*}
\end{enumerate}
\end{lemma}

\begin{proof}
(i)  We may assume that $N$ and $T$ are large enough, as otherwise the claim is trivial. Let $T_0>0$ be a parameter to be chosen. We combine the Hal\'asz--Montgomery type estimate of Proposition~\ref{prop:HM}(i) with Huxley's subdivision. Thus, we split $\mathcal{T}$ into subsets $\mathcal{T}_j=[jT_0,(j+1)T_0)\cap \mathcal{T}$ with $|j|\ll T/T_0+1$ and estimate
\begin{align}
\label{eq:huxley_subdiv}
|\mathcal{T}|V^2\ll  \sum_j\sum_{m\in \mathcal{T}_j}|F(m)|^2.   
\end{align}
By Proposition~\ref{prop:HM}(i) (applied to the coefficient sequence $a_n\lambda^{jT_0}(n)$), we may bound the right-hand side as
\begin{align}\label{eq:nt}
\ll N^{\varepsilon}\left((T/T_0 + 1)NG + |\mathcal{T}|T_0^aN^bG\right),
\end{align}
where we wrote $a = \kappa$ and $b = (\lambda - \kappa + 1)/2$. Let  $T_0 =V^{2/a}N^{-b/a - 2\varepsilon/a}G^{-1/a}$, so that the second term in~\eqref{eq:nt} contributes $|\mathcal{T}|V^2N^{-\varepsilon}$. We then have from~\eqref{eq:huxley_subdiv}
\begin{align*}
|\mathcal{T}| &\ll N^{\varepsilon}(T/T_0+1)V^{-2}NG \\
&\ll N^{O_{\lambda, \kappa}(\varepsilon)}\left(GNV^{-2}+TN^{1+b/a}(GV^{-2})^{1+1/a}\right),
\end{align*}
which is the desired bound (after adjusting $\varepsilon$).

(ii) The proof of this part is identical, except that we apply Proposition~\ref{prop:zeta}(ii) to obtain~\eqref{eq:nt} also with $a=1/3$, $b=1/2$ and with the $N^{\varepsilon}$ factor replaced by $(\log(N+T))^{O(1)}$.
\end{proof}

We then utilize Lemma~\ref{lem:hux_LVT}(ii) to obtain a large value theorem for short, prime-supported polynomials. When applied to a high moment of a prime-supported polynomial $P(s)^k$ (with $P=(\log T)^{c}$ and $P^k=T^{\alpha-o(1)}$ for suitable $\alpha > 0$), this estimate outperforms~\cite[Lemma 6]{teravainen-primes} or~\cite[Lemma 8]{matomaki-radziwill}.

\begin{corollary}[Large values of prime-supported Hecke polynomials]\label{cor:newLVT}
Let $T \ge 2$ and $\sigma>0$, and let $a>2$ be fixed. Let
$$P(m) = \sum_{\N(p) \sim P} a_p\frac{\lambda^m(p)}{\N(p)}$$
with $|a_p| \le 1$ and $P=(\log T)^{a+o(1)}$. Then the number of $m \in [-T, T]\cap \mathbb{Z}$ such that $|P(m)| \ge P^{-\sigma}$ is
\begin{align}
\ll T^{1/(3a/2-3)+8\sigma+o(1)}.
\end{align}
\end{corollary}

Note that this result gives the bound $\ll T^{1/(3a/2-3) + o(1)}$ when $\sigma = o(1)$ and $a>2$. For $a>6$, this improves on the bound $\ll T^{1/a + o(1)}$ that can be deduced from a Hecke polynomial analogue of~\cite[Lemma 6]{teravainen-primes} (see Remark~\ref{rem:LVT}).

\begin{proof} Let $k=\lceil \alpha\log T / \log P \rceil$, with $0 < \alpha < 1$ to be chosen later.
Let $Q(m) = P(m)^k$ and $Q = P^k = T^{\alpha + o(1)}$. Note that $|Q(m)| \ge P^{-k\sigma} = T^{-(\alpha + o(1))\sigma}$ when $|P(m)| \ge P^{-\sigma}$, and if $q_n$ are the coefficients of $Q(m)$, then
\begin{align*}
G \coloneqq \sum_{Q < \N(n) \le 2^kQ} \left|\frac{q_n}{\N(n)}\right|^2 &\le \sum_{Q < \N(n) \le 2^kQ} \left|\sum_{\substack{p_1 \cdots p_k \equiv n \\ P < \N(p_i) \le 2P}} \frac{1}{\N(p_1) \cdots \N(p_k)}\right|^2 \\
&\le \frac{1}{P^k} \sum_{\substack{p_1 \cdots p_k \equiv q_1 \cdots q_k \\ P < \N(p_i), \N(q_j) \le 2P}} \frac{1}{\N(p_1) \cdots \N(p_k)} \\
&\ll \frac{k!}{P^k}\left(\sum_{P < \N(p) \le 2P} \frac{1}{\N(p)}\right)^{k},
\end{align*}
and hence $G \ll T^{-\alpha + \alpha/a + o(1)}$.

By Lemma~\ref{lem:hux_LVT}(ii), we see that the number of large values in question is
\begin{align*}
\ll T^{o(1)}\left(GNP^{2k\sigma} + TN^{5/2}G^4P^{8k\sigma}\right) &\ll T^{o(1)}\left(T^{\alpha/a + 2\alpha \sigma} + T^{1 + 5\alpha/2 + 4(-\alpha + \alpha/a) + 8\sigma \alpha}\right).
\end{align*}
We choose $\alpha=a/(3a/2-3)$ so that $\alpha/a = 1 + 5\alpha/2 + 4(-\alpha + \alpha/a).$
This gives us the bound
$$\ll T^{o(1)}\left(T^{\alpha/a + 2\alpha \sigma} + T^{\alpha/a + 8\sigma \alpha}\right)\ll T^{o(1)+8\sigma\alpha}T^{1/(3a/2-3)},$$
as desired.
\end{proof}

\begin{remark}\label{rem:LVT}
Applying the same argument, but using in place of Lemma~\ref{lem:hux_LVT} a large value estimate following directly from the mean value theorem (Lemma~\ref{lem:MVT}), gives for the number of large values a bound of
\begin{align*}
\ll T^{1/a+2\sigma+o(1)}    
\end{align*}
for $a>1$.
\end{remark}

\subsection{Density bounds}

We apply Lemma~\ref{lem:hux_LVT} to produce some ``density bounds'' (in the spirit of estimates towards the density hypothesis) for the number of large values of Hecke polynomials. These bounds will be employed in the proof of Theorem~\ref{thm:E2}. In the integer setting, a different density bound was used in~\cite[Lemma 4.1]{mato-tera-primes} to study almost primes in almost all short intervals.

\begin{lemma}[A density bound]\label{le:density1}
Let $\varepsilon>0$ be fixed and small enough. Let $\beta\in [\varepsilon,1-\varepsilon]$, $T\geq 2$ and $P=T^{\beta}$. Let $P(m)=\sum_{P\leq n\leq PT^{\varepsilon}}\frac{b_n}{\N(n)}\lambda^{m}(n)$, where $b_n$ are complex numbers with $\sum_{P\leq \N(n)\leq PT^{\varepsilon}}|b_n|^2\ll T^{\varepsilon^2}$. Then we have
\begin{align*}
|\{m\in [-T,T]\cap \mathbb{Z}:\,\, |P(m)|\geq P^{-\sigma}\}|\ll T^{(2-\varepsilon)\sigma},    
\end{align*}
provided that one of the following holds.
\begin{enumerate}[(i)]

\item We have $\beta\geq 2/5$ and, for some exponent pair $(\kappa,\lambda)$, 
\begin{align}\label{eq:sigmakappa}
\varepsilon C_{\kappa,\lambda} \leq \sigma\leq \frac{\frac{\kappa-\lambda+1}{2\kappa}\beta-1}{(2+2/\kappa)\beta-2}-\varepsilon C_{\kappa,\lambda}
\end{align}
for some large constant $C_{\kappa,\lambda}$, and additionally $b_n=0$ whenever $\arg(n)$ is within $P^{-\varepsilon^4}$ of any real solution to $\Im((1+i\tan(t))^k)=0$ with $k=1,\ldots, J$ for some integer $J$ that is large in terms of $(\kappa,\lambda)$.

\item We have $\beta\geq 2/3$ and
    \begin{align}\label{eq:sigma0}
\varepsilon C_0\leq \sigma\leq \frac{\frac{3}{2}\beta-1}{8\beta-2}-\varepsilon C_0
\end{align}
for some large absolute constant $C_0$.
\end{enumerate}
\end{lemma}

\begin{proof} (i) By Lemma~\ref{lem:hux_LVT}(i), the number of large values in question is
\begin{align}\label{eq:densitybound}
\ll T^{O(\varepsilon^2)}(P^{2\sigma}+TP^{\frac{\lambda-\kappa-1}{2\kappa}}\cdot P^{(2+2/\kappa)\sigma}).    
\end{align}
The first term in~\eqref{eq:densitybound} is
\begin{align*}
\ll T^{O(\varepsilon^2)+(2-2\varepsilon)\sigma}\leq T^{(2-\varepsilon)\sigma}    
\end{align*}
if $\sigma\geq C_0\varepsilon$ for a large enough constant $C_0$. The second term in~\eqref{eq:densitybound} is
\begin{align*}
T^{1+\frac{\lambda-\kappa-1}{2\kappa}\beta+(2+2/\kappa)\beta\sigma+O(\varepsilon^2)},    
\end{align*}
and this is $\leq T^{(2-\varepsilon)\sigma}$ when the second inequality in~\eqref{eq:sigmakappa} holds. (Note that the denominator on the right hand side of~\eqref{eq:sigmakappa} is positive since $\beta \ge 2/5$ and $\kappa \le 1/2$.)

(ii) By Lemma~\ref{lem:hux_LVT}(ii), the number of large values in question is
\begin{align*}
\ll T^{\varepsilon^2}(P^{2\sigma}+TP^{-3/2+8\sigma}).
\end{align*}
The first term here is admissible as in part (i), and the second term is
\begin{align*}
T^{1+(8\sigma-3/2)\beta+\varepsilon^2},    
\end{align*}
and this is $\leq T^{(2-\varepsilon)\sigma}$ when the second inequality in~\eqref{eq:sigma0} holds. (Note that the denominator in~\eqref{eq:sigma0} is positive since $\beta > 1/4$.)
\end{proof}

\begin{lemma}[A density bound using amplification]\label{le:density2} Let $\varepsilon>0$ be fixed and small enough, and let $A\geq 2$ be fixed. Let $\beta\in [2/5,1-\varepsilon]$, $\delta\in [2\varepsilon,1-\beta-\varepsilon]$, $A\geq 2$, $T\geq 2$ and $P=T^{\beta}$. Let $P(m)=\sum_{P\leq \N(n)\leq 10P}a_n\lambda^{m}(n)/\N(n)$, where $a_n$ are complex numbers with $|a_n|\leq \tau(n)$, and let $F(m)=\sum_{F^{1-\varepsilon^2}\leq \N(n)\leq F}b_n\lambda^m(n)/\N(n)$, where $F\in [T^{\varepsilon/2},T^{2\varepsilon}]$ and $\sum_{\N(n)\leq F}|b_n|^2\leq F$. Then we have
\begin{align}\label{eq:densitybound2}
|\{m\in [-T,T]\cap \mathbb{Z}:\,\, |P(m)|\geq P^{-\sigma}\textnormal{ and } |F(m)|\geq F^{-1/(2A)}\}|\ll T^{(2-\varepsilon)\sigma},    
\end{align}
provided that
\begin{align}\label{eq:densitybound3}
\frac{\delta}{A(2-2\beta-\varepsilon)}+C_0\varepsilon\leq \sigma\leq \frac{\frac{3}{2}(\beta+\delta)-1-\frac{4\delta}{A}}{8\beta-2}-C_0\varepsilon    
\end{align}
for some large absolute constant $C_0$.
\end{lemma}

\begin{remark}
In applications, we take $F(m)=Q(m)^k/k!^{1/2}$, where 
$$Q(m)=\sum_{Q/2\leq \N(p)\leq Q}\lambda^m(p)/\N(p),$$
with the sum ranging over Gaussian primes, and with $Q=(\log T)^{A}$ and $k$ chosen so that $Q^k\approx T^{\varepsilon}$. 
\end{remark}

\begin{proof}
Let
\begin{align*}
\ell=\left\lfloor \delta\frac{\log T}{\log F}\right\rfloor.    
\end{align*}
Then $1\leq \ell\leq 2\varepsilon^{-1}$. Consider the Hecke polynomial
\begin{align*}
B(m)=P(m)F(m)^{\ell}=\sum_{B\leq \N(n)\leq 10 T^{2\varepsilon^3\ell}B}c_n\lambda^m(n),    
\end{align*}
where $B=PF^{\ell}T^{-2\varepsilon^3\ell}$ and $c_n$ are the coefficients of $B(m)$.

Now, for any $m$ in the large value set in question, we have
\begin{align*}
|B(m)|\geq P^{-\sigma}F^{-\ell/(2A)}\geq  T^{-\beta \sigma-\frac{\delta}{2A}}\geq B^{-\sigma'}
\end{align*}
for the choice
\begin{align}\label{eq:sigma'}
\sigma'=\frac{\beta \sigma+\frac{\delta}{2A}}{\beta+\delta}+10\varepsilon^2. \end{align}

Note that by the divisor bound $\tau_{\ell+1}(n)\ll |n|^{o(1)}$ and Cauchy--Schwarz we have
\begin{align*}
\sum_{B\leq \N(n)\leq 10T^{4\varepsilon^2}B}|c_n|^2&\leq \frac{1}{B^2}\sum_{B\leq \N(n)\leq 10T^{4\varepsilon^2}B}\left(\sum_{\substack{n\equiv n_0n_1\cdots n_{\ell}\\P\leq \N(n_0)\leq 10P \\F^{1-\varepsilon^2}\leq \N(n_1),\ldots, \N(n_{\ell})\leq F}}|a_{n_0}||b_{n_1}|\cdots |b_{n_{\ell}}|\right)^2\\
&\ll \frac{T^{o(1)}}{B^2}\sum_{B\leq \N(n)\leq 10T^{4\varepsilon^2}B}\sum_{\substack{n\equiv n_0n_1\cdots n_{\ell}\\P\leq \N(n_0)\leq 10P\\F^{1-\varepsilon^2}\leq \N(n_1),\ldots, \N(n_{\ell})\leq F}}|b_{n_1}|^2\cdots |b_{n_{\ell}}|^2\ll \frac{T^{o(1)}}{B^2}PF^{\ell}\\
&\ll \frac{T^{4\varepsilon^2+o(1)}}{B}.
\end{align*}

Using this together with Lemma~\ref{lem:hux_LVT}(ii) and recalling~\eqref{eq:sigma'}, we deduce that the number of large values in question is
\begin{align*}
&\ll T^{O(\varepsilon^2)}(B^{2\sigma'}+TB^{-3/2+8\sigma'})\ll T^{O(\varepsilon^2)}(T^{2(\beta\sigma+\delta/(2A))}+T^{1-(3/2)(\beta+\delta)+8(\beta\sigma+\delta/(2A))}).    
\end{align*}
Elementary manipulation shows that this is $\ll T^{(2-\varepsilon)\sigma}$ when~\eqref{eq:densitybound3} holds.
\end{proof}

\section{Factorizing Hecke polynomials and bounding the error term}\label{sec:factorize}

Our next lemma shows how to factorize certain Hecke polynomials arising in our arguments.

\begin{lemma}
\label{lem:factorize}
Let $T,X\geq 2$ and $S \subset [-T, T] \cap \mathbb{Z}$. Let
$$F(m) = \sum_{\substack{\N(kn) \sim X \\ K < \N(k) \le K'}} \frac{a_kb_n}{\N(kn)}\lambda^m(kn)$$
for some $K' > K \ge 2$ and for some complex numbers $a_k, b_n$. Let $H \ge 1$. Denote
$$A_{v, H}(t) = \sum_{e^{v/H} < \N(k) \le e^{(v+1)/H}}  \frac{a_k}{\N(k)}\lambda^m(k), \quad B_{v, H}(t) = \sum_{\N(\ell) \sim Xe^{-v/H}} \frac{b_{\ell}}{\N(\ell)}\lambda^m(\ell)$$
and
$$c_n = \frac{1}{\N(n)} \sum_{\substack{n \equiv k\ell \\ K < \N(k) \le K'}} |a_k b_{\ell}|.$$
Writing $I = (H \log K - 1, H \log K']$, we have
\begin{align*}
\sum_{m \in S} |F(m)|^2 \ll |I| \sum_{v \in I} \sum_{m \in S} |A_{v, H}(t)B_{v, H}(t)|^2 + T \sum_{\substack{|\arg n_1 - \arg n_2| \leq  1/T \\ \N(n_1), \N(n_2) \in [Xe^{-1/H}, Xe^{1/H}] \textnormal{ or} \\ \N(n_1), \N(n_2) \in [2X, 2Xe^{1/H}]}} c_{n_1}c_{n_2}.
\end{align*}
\end{lemma}

\begin{proof}
The proof is analogous to that of~\cite[Lemma 2]{teravainen-primes}, using the improved mean value theorem (Lemma~\ref{lem:IMVT}) in place of its integer analogue,~\cite[Lemma 4]{teravainen-primes}. 
\end{proof}

We use the following result to handle the error term in Lemma~\ref{lem:factorize}. The proof requires a substantial amount of work, occupying the remainder of this section and the next section.

\begin{proposition}\label{prop:mvtapplication} Let $X,T\geq 2$. Let $r\geq 0$ be a fixed integer, let $\varepsilon>0$ be small enough in terms of $r$, and let $I_1, \ldots , I_r$ be pairwise disjoint intervals of form $I_i = [z_i, z_i^2]$ with $1\leq z_i\leq  X^{\varepsilon}$. Define
\begin{align*}
\alpha_n=\begin{cases}\frac{1}{\N(n)}\quad\,\, \textnormal{ if } p \mid n \implies \N(p) \not\in [2, 2X^{1/2}] \setminus (I_1 \cup \ldots \cup I_r) \\0,\quad \quad \,\textnormal{ otherwise.}\end{cases}
\end{align*}
For any $\eta > 0$ tending to zero sufficiently slowly and $4X^{1/2}< T \le \eta^2 \exp(-(\log \eta)^2) X/\log X$, we have
\begin{align}\label{eq_argsum}
T\sum_{\substack{n_1, n_2 \in \mathbb{Z}[i]^{\ast} \\ |\arg(n_1)-\arg(n_2)|\leq 1/T\\\N(n_1),\N(n_2)\in [X,(1+\eta)X]}}\alpha_{n_1}\alpha_{n_2}\ll \eta^2 \frac{1}{(\log X)^2},
\end{align}
where the implied constant does not depend on $\eta$.
\end{proposition}

\begin{remark}
Our bound~\eqref{eq_argsum} is heuristically optimal up to a constant factor. The upper bound $T \le \eta^2\exp(-(\log \eta)^2) X/\log X$ for $T$ in Proposition~\ref{prop:mvtapplication} is used to guarantee that the contribution of the diagonal case $\arg n_1 = \arg n_2$ to the left-hand side of~\eqref{eq_argsum} is small enough. 
If one restricted the sum to $\arg n_1 \neq \arg n_2$, one could relax the condition to $4X^{1/2} < T$.
\end{remark}

The proof of the proposition uses the fundamental lemma of the sieve together with the following estimate for a divisor correlation of a certain kind, where it is crucial that the moduli are allowed to go up to a power of $x$. This estimate is based on the method of Deshouillers and Iwaniec~\cite{DI} for proving a power-saving estimate for $\sum_{n\leq x}\tau(n)\tau(n+1)$ (with error term $O_{\varepsilon}(x^{2/3+\varepsilon})$).

\begin{proposition}[A divisor problem in progressions with power-saving error term]\label{prop_divisorsum}
Let $\delta > 0$ be a sufficiently small fixed constant. Let $x\geq 2$, and let $T_1, T_2, k$ be integers satisfying $1\leq T_1, T_2, |k| \le x^{\delta}$ with $T_1, T_2$ square-free and $(k, T_1T_2) = 1$. Let $M_1, M_2, M_3, M_4 \in [x^{1-\delta}, 100x]$ and let $b_1, b_2, b_3, b_4$ be non-negative smooth functions with $b_i$ supported in $[M_i, 2M_i]$ and satisfying the derivative bounds $|b_i^{(h)}(t)| \ll_h M_i^{-h}$ for all $h \in \mathbb{Z}_{\ge 0}$. Then
\begin{align*}
& \sum_{\substack{m_1, m_2, m_3, m_4 \in \mathbb{Z} \\ m_1m_2 - m_3m_4 = k \\ T_1 \mid m_1^2 + m_3^2 \\ T_2 \mid m_2^2 + m_4^2}} b_1(m_1)b_2(m_2)b_3(m_3)b_4(m_4) \\
&= \int \left(\int \frac{b_1(s)b_2(t/s)}{s} \d s\right)\left(\int \frac{b_3(s)b_4(t/s)}{s} \d s\right) \d t \cdot \frac{6}{T_1T_2\pi^2} \prod_{p \mid kT_1T_2} f_p(k, T_1, T_2) + O\left(x^{2 - \delta}\right),
\end{align*}
where $f_p(a, b, c)$ is a certain function the values of which only depend on the largest powers of $p$ dividing $a, b$ and $c$ and which is symmetric in $b$ and $c$. Explicitly, we have the following formulas, where $p$ is a prime and $v \ge 0$:
\begin{itemize}
\item For any prime $p$, we have
$$f_p(p^v, 1, 1) = \frac{p^{v+1} - 1}{p^v(p-1)}.$$
(In particular $f_p(1, 1, 1) = 1$.)
\item If $p \equiv 1 \pmod{4}$, then
$$f_p(1, p, 1) = f(1, p, p) = \frac{2p}{p+1}.$$
\item If $p \equiv 3 \pmod{4}$, then
$$f_p(1, p, 1) = f_p(1, p, p) = 0.$$
\end{itemize}
\end{proposition}

\begin{remark}
The result holds even without the condition $(k, T_1T_2) = 1$ (with more complicated formulas for $f_p(a, b, c)$), and we only utilize this assumption at the end of the proof when computing the main term. Likely the result extends to non-square-free $T_1, T_2$ as well. We have presented the simplest result that fits our needs, as computing the main terms in more general cases gets quickly rather laborious.
\end{remark}

Below we show how Proposition~\ref{prop_divisorsum} implies Proposition~\ref{prop:mvtapplication}. Section~\ref{sec:divisor} is then devoted to the proof of Proposition~\ref{prop_divisorsum}.

\subsection{Reduction of Proposition~\ref{prop:mvtapplication} to Proposition~\ref{prop_divisorsum}}

\begin{proof}[Proof of Proposition~\ref{prop:mvtapplication} assuming Proposition~\ref{prop_divisorsum}] We first consider the contribution of the terms with $\arg n_1=\arg n_2$ to the left-hand side of ~\eqref{eq_argsum}. Given $n_1, n_2 \in \mathbb{Z}[i]^{\ast}$, let $n \in \mathbb{Z}[i]^{\ast}, v_i \in \mathbb{Z}_+$ be such that $n_i = v_in$ with $n$ a primitive Gaussian integer. The terms with $\arg n_1 = \arg n_2$ then contribute
\begin{align}\label{eq:vsum}
T \sum_{\substack{\arg n_1 = \arg n_2 \\ \N(n_1), \N(n_2) \in [X, (1 + \eta)X]}} \alpha_{n_1}\alpha_{n_2} &\ll \frac{T}{X^2} \sum_{0<\N(n)\leq (1+\eta)X}  \left(\sum_{\substack{v \in \mathbb{Z}_+ \\ v^2 \in [X/\N(n), (1 + \eta)X/\N(n)] \\ \alpha_{vn} \neq 0}} 1\right)^2.
\end{align}

We now claim that for any $y \le X$ such that $\eta^{-20 \log \log y} \le y$ the number of $v \in [y, \sqrt{1 + \eta}y]$ such that all prime divisors of $v$ lie in $[z_1, z_1^2] \cup \cdots \cup [z_r, z_r^2] \cup [2X^{1/2}, \infty)$ is $O(\eta y / \log y)$. Denote the number of such $v$ by $V = V(y, \eta)$ and let $z_{r+1} = 2X^{1/2}$.  The number of positive integers $v\leq 2y$ which have $\Omega(v)\geq 10 \log \log y$ with $\Omega(v)$ the total number of prime factors of $v$ with multiplicities is by Shiu's bound~\cite[Theorem 1]{shiu} 
\begin{align*}
\leq 2^{-10\log \log y}\sum_{v\leq 2y}2^{\Omega(v)}\ll \frac{y}{(\log y)^2},    
\end{align*}
say, which is negligible.  We then fix $k < 10 \log \log y$ and write $v = p_1 \cdots p_k$, assuming $p_1 \le \ldots \le p_k$ and letting $0 = j_0 \le j_1 \le \ldots \le j_{r+1} = k$ denote the indices such that $p_i \in [z_t, z_t^2]$ for $j_t < i \le j_{t+1}$. There are $O(k^{r-1})$ possible values for $(j_1, \ldots , j_{r-1})$. Given $p_1, \ldots , p_{k-1}$, we bound the number of possible values of
$$p_k \in \left[\frac{y}{p_1 \cdots p_{k-1}}, \frac{\sqrt{1 + \eta}y}{p_1 \cdots p_{k-1}}\right]$$
by the Brun--Titchmarsh inequality and the bounds $p_1 \cdots p_{k-1} \le (2y)^{(k-1)/k}$, $k < 10 \log \log y$ and $\eta^{-20 \log \log y} \le y$ as
$$\ll \frac{\eta y / (p_1 \cdots p_{k-1})}{\log(\eta y / (p_1 \cdots p_{k-1}))} \ll k \frac{\eta y}{p_1 \cdots p_{k-1} \log y}.$$
Our claim follows by summing over $p_1, \ldots , p_{k-1}$: the number $V$ is bounded by
$$V \ll \frac{y}{(\log y)^2} + \sum_{k < 10 \log \log y} k^{r-1} \cdot k \cdot \sum_{p_1 \in I_1,\ldots,p_{k-1}\in I_{k-1}} \frac{\eta y}{p_1 \cdots p_{k-1} \log y}$$
for some intervals $I_i = I_i(k)$ that are of form $[z_j, z_j^2], 1 \le j \le r$, and this is bounded, for any $\varepsilon > 0$, by
$$\ll_{\varepsilon} \frac{y}{(\log y)^2} + \sum_{k \in \mathbb{N}} \frac{\eta y}{\log y} k^r  (\log 2 + \varepsilon)^k \ll \frac{\eta y}{\log y}.$$

Denote by $C_{\eta}$ the supremum of values of $y$ with $\eta^{-20 \log \log y} \ge y$. One can check that $C_{\eta} \ll \exp((\log \eta)^2/2)$.

We thus have, for given $n$ and $C_{\eta} \le z \le 2X$, that the number of $v \in [z, \sqrt{(1 + \eta)}z]$ with $\alpha_{vn} \neq 0$ is $O(\eta z/\log z)$. For $z \le C_{\eta}$ we use the trivial bound $O(C_{\eta})$ for the number of such $v$. Hence, the right hand side of~\eqref{eq:vsum} is bounded by
\begin{align*}
\ll \frac{T}{X^2} \sum_{0 < \N(n) \le (1+\eta)X} \left(\frac{\eta^2 X/\N(n)}{(\log 3X/\N(n))^2} + C_{\eta}^2\right)1_{\alpha_n \neq 0}.
\end{align*}
As above, the number of $n, \N(n) \le y$ with $\alpha_n \neq 0$ is $O(y/\log y)$. Thus, the contribution of the $C_{\eta}$ term is bounded by $\frac{TC_{\eta}^2}{X(\log X)}$, which suffices, as $T \le \eta^2\exp(-(\log \eta)^2) X/\log X$. The contribution of the rest is bounded by
\begin{align*}
 &\ll \frac{T \eta^2}{X} \sum_{0 < \N(n) \le (1+\eta)X} \frac{1}{\N(n)(\log 3X/\N(n))^2} 1_{\alpha_n \neq 0},
\end{align*}
to which we apply a dyadic decomposition. The contribution of $\N(n) \in [y, 2y]$ to the above sum is bounded by $\ll (\log y)^{-1}(\log(X/y))^{-2}$, and thus we obtain an upper bound of
$$\ll T \frac{\eta^2}{X/(\log X)}.$$
This suffices, as $T \le \eta^2 X/(\log X)$.

Hence, we may restrict to $\arg n_1 \neq \arg n_2$. We then note that for $n_1, n_2$ counted on the left-hand side of~\eqref{eq_argsum} we have $|\arg n_1 \overline{n_2}| \le 1/T$ and $\N(n_1n_2) \le 4X^2$, which in particular implies that 
$$|\Im(n_1 \overline{n_2})| = |n_1n_2||\sin\left(\arg n_1\overline{n_2}\right)| \le \frac{2X}{T}.$$
Writing $n_1 = a + bi, n_2 = c +di$ with $0 \leq  a, b, c, d \leq \sqrt{2X}$, we thus have $|ad - bc| \leq 2X/T$.

Let $\delta$ be as in Proposition~\ref{prop_divisorsum}. Next, we discard the contribution of the case $\min(a, b, c, d) \le X^{1/2 - \delta/10}$. This corresponds to $\min(|\arg n_1|, |\arg n_2|) \ll X^{-\delta/10}$, where we recall our convention on $\arg n$ being defined modulo $\pi/2$. We handle the case $0 \le \arg n_1 \ll X^{-\delta/10}$, the other case is similar. There are $O(X^{1 - \delta/10})$ Gaussian integers $n_1$ with $\N(n_1) \le 2X$ in the sector $0 \le \arg n_1 \ll X^{-\delta/10}$, and given $n_1$, the number of $n_2$ with $\N(n_2) \le 2X$, $0 < |\arg n_2 - \arg n_1| \le 1/T$ is by Lemma~\ref{lem:countSector} bounded by $O(X/T)$. Hence, the contribution of this case to the left-hand side of~\eqref{eq_argsum} is
$$\ll \frac{T}{X^2} X^{1 - \delta/10} \frac{X}{T} \ll X^{-\delta/10},$$
which is small enough.

Hence, we have reduced matters to bounding
\begin{align*}
T\sum_{\substack{X^{1/2 - \delta/10}<a,b,c,d\leq \sqrt{2X}\\0 < |ad-bc| \leq 2X/T  \\ a^2+b^2,c^2+d^2\in [X,(1+\eta)X]}}\alpha_{a+bi}\alpha_{c+di}.    
\end{align*}

Note that $\alpha_{a+bi} \neq 0$ in particular implies that $\N(a+bi)=a^2+b^2$ has no prime factors which lie in $[2, X^{1/4}] \setminus ([\sqrt{z_1}, z_1^2] \cup \ldots \cup [\sqrt{z_r}, z_r^2])$. Let the set of such integers be $\mathcal{Q}$. Hence, the previous sum is at most
\begin{align}\label{eq_a2b2}
\ll \frac{T}{X^2} \sum_{\substack{X^{1/2 - \delta/10}< a,b,c,d\leq \sqrt{2X}\\0 < |ad-bc|\leq 2X/T\\a^2+b^2,c^2+d^2\in [X,(1+\eta)X]}}1_{a^2+b^2\in \mathcal{Q}}1_{c^2+d^2\in \mathcal{Q}}.
\end{align}

To treat the conditions $a^2 + b^2, c^2 + d^2 \in [X, (1 + \eta)X]$, we perform a smoother-than-dyadic decomposition over $a, b, c, d$. Let $J_i$ be intervals of form $[(1 + \eta)^i, (1 + \eta)^{i+1}]$. Consider the set $\mathcal{J}$ of $4$-tuples $(J_{i_1}, J_{i_2}, J_{i_3}, J_{i_4})$ for which the contribution of $a \in J_1, b \in J_2, c \in J_3, d \in J_4$ to~\eqref{eq_a2b2} is non-zero.

For a fixed tuple $(J_{i_1}, J_{i_2}, J_{i_3}, J_{i_4})$, let $b_{i_1}, b_{i_2}, b_{i_3}, b_{i_4}$ be nonnegative smooth functions with the following properties for all $1 \le k \le 4$:
\begin{itemize}
\item $b_{i_k}$ satisfies the bound $b_{i_k}(t) \le 2\eta$ for all $t$ and $b_{i_k}(t) \ge \eta$ for $t \in J_{i_k}$
\item $b_{i_k}$ is supported in $J_{i_k - 1} \cup J_{i_k} \cup J_{i_k + 1}$.
\item $b_{i_k}$ satisfies the bounds $|b_{i_k}^{(h)}(t)| \ll_{h, \delta} (1+\eta)^{-i_kh}$ for all $h \in \mathbb{Z}_+$, where $b_{i_k}^{(h)}$ is the $h$th derivative of $b_{i_k}$.
\end{itemize}

With these choices we have
\begin{align}
\label{eq:b-sum}
&\frac{T}{X^2} \sum_{\substack{X^{1/2 - \delta/10} < a, b, c, d \le \sqrt{2X} \\ 0 < |ad - bc| \le 2X/T \\ a^2 + b^2, c^2 + d^2 \in [X, (1 + \eta)X]}} 1_{a^2+b^2 \in \mathcal{Q}}1_{c^2+d^2\in \mathcal{Q}}\nonumber \\
\le & \frac{T}{\eta^4 X^2} \sum_{0 < |k| \le 2X/T} \sum_{(J_{i_1}, J_{i_2}, J_{i_3}, J_{i_4}) \in \mathcal{J}} \sum_{\substack{a, b, c, d \in \mathbb{Z}_+ \\ ad - bc = k}} b_{i_1}(a)b_{i_2}(b)b_{i_3}(c)b_{i_4}(d)1_{(a^2+b^2)(c^2 + d^2) \in \mathcal{Q}}.\
\end{align}

By the fundamental lemma of sieve theory~\cite[Fundamental Lemma 6.3]{iw-kow}, there exists a sequence $\lambda_D\in [-1,1]$ of real numbers with the following properties:
\begin{enumerate}
    \item For $1\leq m\leq 4X^2$, we have
\begin{align}
\label{eq:bound_ind}
1_{m\in \mathcal{Q}}\leq \sum_{\substack{D\mid m}}\lambda_D;
\end{align}

    \item $(\lambda_D)_D$ is supported in
\begin{align*}
\{1\leq D\leq X^{\delta/10},\,\, D\mid \Pi\},\quad \Pi=\prod_{\substack{p\leq X^{\delta/10}\ \\ p \equiv 1 \pmod{4} \\p \not\in [z_1, z_1^2] \cup \ldots \cup [z_r, z_r^2]}}p,
\end{align*}
where $\delta$ is as in Proposition~\ref{prop_divisorsum};
\item For any multiplicative function $g:\mathbb{N}\to [0,1]$ with $0\leq g(p)<1$ for all $p\mid \Pi$ 
and with $g$ satisfying the dimension condition
\begin{align}
\label{eq:dim_cond}
\prod_{w\leq p <z}(1-g(p))^{-1}\leq \left(\frac{\log z}{\log w}\right)^{10}\left(1+\frac{K}{\log w}\right)  
\end{align}
for $2\leq w\leq z\leq X^{\delta/10}$,
we have
\begin{align}
\label{eq:sieve_bound}
\sum_{\substack{D}} \lambda_{D}g(D)\ll_K \prod_{p\mid \Pi}(1-g(p)).   
\end{align}
\end{enumerate}
Note that we may insert the condition $(D, k) = 1$ to the sums in~\eqref{eq:bound_ind} and~\eqref{eq:sieve_bound}: If~\eqref{eq:bound_ind} holds without the condition $(D, k) = 1$ for given $(\lambda_D)$, it also holds with the condition present, as for any $m$ we have
$$1_{m \in \mathcal{Q}} \le 1_{m/(m, k^{\infty}) \in \mathcal{Q}} \le \sum_{D \mid m/(m, k^{\infty})} \lambda_D = \sum_{\substack{D \mid m \\ (D, k) = 1}} \lambda_D,$$
where by $m/(m, k^{\infty})$ we denote the largest divisor of $m$ coprime with $k$. If~\eqref{eq:sieve_bound} holds without the condition $(D, k) = 1$ for all $g$ as in~\eqref{eq:dim_cond}, in order to recover~\eqref{eq:sieve_bound} one may then replace $g(D)$ by $g(D) \cdot 1_{(D, k) = 1}$.

Hence, we may upper bound~\eqref{eq:b-sum} by
\begin{align}
\label{eq:sieved}
\frac{T}{\eta^4 X^2} \sum_{0 < |k| \le 2X/T} \sum_{(J_{i_1}, J_{i_2}, J_{i_3}, J_{i_4}) \in \mathcal{J}} \sum_{\substack{D \le X^{\delta/10} \\ (D, k) = 1}} \lambda_D \sum_{\substack{a, b, c, d \\ ad - bc = k \\ D \mid (a^2 + b^2)(c^2 + d^2)}} b_{i_1}(a)b_{i_2}(b)b_{i_3}(c)b_{i_4}(d).
\end{align}
By M\"{o}bius inversion,
\begin{align*}
\sum_{\substack{a, b, c, d \\ ad - bc = k \\ D \mid (a^2 + b^2)(c^2 + d^2)}} b_{i_1}(a)b_2(b)b_3(c)b_4(d) &= 
\sum_{\substack{m_1, m_2 \mid D \\ D \mid m_1m_2}} \sum_{\substack{a, b, c, d \\ ad - bc = k \\ (a^2 + b^2, D) = m_1 \\ (c^2 + d^2, D) = m_2}} b_{i_1}(a)b_{i_2}(b)b_{i_3}(c)b_{i_4}(d) \\
&= \sum_{\substack{m_1, m_2, n_1, n_2 \\ m_1n_1 \mid D \\ m_2n_2 \mid D \\ D \mid m_1m_2}} \mu(n_1)\mu(n_2) \sum_{\substack{a, b, c, d \\ ad - bc = k \\ m_1n_1 \mid a^2 + b^2 \\ m_2n_2 \mid c^2 + d^2}} b_{i_1}(a)b_{i_2}(b)b_{i_3}(c)b_{i_4}(d).
\end{align*}
Noting that $\supp(b_{i_{j}}) \subset [X^{1/2 - \delta/10}, \sqrt{2X}]$, we may apply Proposition~\ref{prop_divisorsum} to evaluate the previous expression as
\begin{align*}
\begin{split}
& \sum_{\substack{m_1, m_2, n_1, n_2 \\ m_1n_1 \mid D \\ m_2n_2 \mid D \\ D \mid m_1m_2}}\int \left(\int \frac{b_{i_1}(s)b_{i_4}(t/s)}{s} \d s\right)\left(\int \frac{b_{i_2}(s)b_{i_3}(t/s)}{s} \d s\right) \d t \times \\
& \mu(n_1)\mu(n_2) \frac{6}{m_1n_1m_2n_2\pi^2} \prod_p f_p(k, m_1n_1, m_2n_2) + O(X^{1 - \delta/2}).
\end{split}
\end{align*}
Denote the value of the integral by $I = I_{i_1, i_2, i_3, i_4}$. By multiplicativity, we may write the above as
\begin{align}
\label{eq:b-summed}
\frac{6I}{\pi^2} \prod_{p \mid kD} g_p(k, D) + O_{\varepsilon}(X^{1 - \delta/2 + \varepsilon})
\end{align}
for
\begin{align}
\label{eq:def-g}
g_p(k, D) = \sum_{\substack{m_1, m_2, n_1, n_2 \\ m_1n_1 \mid (D, p) \\ m_2n_2 \mid (D, p) \\ (D, p) \mid m_1m_2}} \mu(n_1)\mu(n_2)\frac{1}{m_1n_1m_2n_2} f_p(k, m_1n_1, m_2n_2).
\end{align}
(Recall that $D$ is square-free.) Note that the value of $g_p(k, D)$ depends only on the exponents $v_p(k)$ and $v_p(D)$ of $p$ in $k$ and $D$ and that $g_p(k,1)=f_p(k,1,1)$. In particular, $g_p(1, 1) = 1$.

Plugging~\eqref{eq:b-summed} into~\eqref{eq:sieved}, we obtain
\begin{align}
\label{eq:calculated}
& \frac{T}{\eta^4 X^2} \sum_{0 < |k| \le 2X/T} \sum_{(J_{i_1}, J_{i_2}, J_{i_3}, J_{i_4}) \in \mathcal{J}} \sum_{\substack{D \le X^{\delta/10} \\ (D, k) = 1}} \lambda_D \left(\frac{6I}{\pi^2} \prod_{p \mid kD} g_p(k, D) + O(X^{1 - \delta/3 })\right) \nonumber \\
= & \frac{6T}{\pi^2 \eta^4 X^2} \left( \sum_{(J_{i_1}, J_{i_2}, J_{i_3}, J_{i_4}) \in \mathcal{J}} I \sum_{0 < |k| \le 2X/T} \left(\prod_{p \mid k} f_p(k, 1, 1)\right) \sum_{\substack{D \le X^{\delta/10} \\ (D, k) = 1}} \lambda_D \prod_{p \mid D} g_p(1, D)\right) + O(X^{-\frac{\delta}{5}}).
\end{align}
The error term is negligible when compared to the right-hand side of~\eqref{eq_argsum}.

To evaluate the sum over $D$ we apply the fundamental lemma. To do so, we have to check the dimension condition~\eqref{eq:dim_cond}. Fix $k$ and let $g(p) = g_p(1, p)$ for primes $p$, extending $g$ multiplicatively to all integers dividing $\Pi$. We compute, using~\eqref{eq:def-g} and the formulas for $f_p$ in Proposition~\ref{prop_divisorsum}, that for $p \mid D$, $p \equiv 1 \pmod{4}$ we have
\begin{align*}
g(p) = g_p(1, p) & = \sum_{\substack{m_1, m_2, n_1, n_2 \\ m_1n_1 \mid p \\ m_2n_2 \mid p \\ p \mid m_1m_2}} \mu(n_1)\mu(n_2)\frac{1}{m_1n_1m_2n_2} f_p(1, m_1n_1, m_2n_2) \\
& = 2\frac{1}{p}f_p(1, p, 1) - \frac{1}{p^2}f_p(1, p, p) \\
&= \frac{2}{p} \frac{2p}{p+1} - \frac{1}{p^2} \frac{2p}{p+1} \\
&= \frac{4p-2}{p(p+1)}.
\end{align*}
One easily checks that $g(p) < \min(10/p, 1)$, say. Hence, by Mertens's theorem, $g$ satisfies the dimension condition~\eqref{eq:dim_cond} (for some $K = O(1)$) and we have, by~\eqref{eq:sieve_bound},
\begin{align*}
\sum_{\substack{D \le X^{\delta/10} \\ (D, k) = 1}} \lambda_D \prod_{\substack{p \mid D \\ p \nmid k}} g_p(1, D) &\ll \prod_{\substack{p \mid \Pi \\ p \equiv 1 \pmod{4} \\ p \nmid k}} (1 - g(p)) = \prod_{\substack{p \mid \Pi \\ p \equiv 1 \pmod{4} \\ p \nmid k}} \left(1 - \frac{4p-2}{p(p+1)}\right).
\end{align*}
We bound this in a routine way using the prime number theorem in arithmetic progressions and the fact that $\prod_{z \le p \le z^2} (1 + 4/p) \ll 1$ for any $z \ge 1$, obtaining the bound
\begin{align*}
\ll \frac{1}{(\log X)^2} \prod_{\substack{p \mid k \\ p \equiv 1 \pmod{4}}} \left(1 + \frac{4}{p}\right).
\end{align*}

Plugging the obtained bound to~\eqref{eq:calculated}, we can upper bound the main term there by
\begin{align}
\label{eq:plugged}
\frac{T}{\eta^4 X^2(\log X)^2} \left(\sum_{(J_{i_1}, J_{i_2}, J_{i_3}, J_{i_4}) \in \mathcal{J}} I \sum_{0 < |k| \le 2X/T} \left(\prod_{p \mid k} f_p(k, 1, 1)\right) \prod_{\substack{p \mid k \\ p \equiv 1 \pmod{4}}} \left(1 + \frac{4}{p}\right)\right).
\end{align}
The sum over $k$ is bounded by routine methods. Note that $f_p(k, 1, 1) \le p/(p-1) \le 1 + 2/p$. Hence, if $\omega(m)$ denotes the number of distinct prime factors of $m \in \mathbb{Z}_+$, we have
\begin{align*}
\prod_{p \mid k} f_p(k, 1, 1) \prod_{\substack{p \mid k \\ p \equiv 1 \pmod{4}}} \left(1 + \frac{4}{p}\right) \ll \prod_{p \mid k} \left(1 + \frac{6}{p}\right) \le \sum_{m \mid k} \frac{6^{\omega(m)}}{m} \le \sum_{m \mid k} \frac{\tau(m)^3}{m}.
\end{align*}
Hence
\begin{align*}
& \sum_{0 < |k| \le 2X/T} \left(\prod_{p \mid k} f_p(k, 1, 1)\right) \prod_{\substack{p \mid k \\ p \equiv 1 \pmod{4}}} \left(1 + \frac{4}{p}\right) \\
\ll & \sum_{1 \le m \le 2X/T} \frac{\tau(m)^3}{m} \sum_{\substack{0 < |k| \le 2X/T \\ m \mid k}} 1 \\
\ll & \sum_{1 \le m \le 2X/T} \frac{1}{m^{1-\varepsilon}} \frac{X/T}{m} \\
\ll & \frac{X}{T}.
\end{align*}
Thus,~\eqref{eq:plugged} is bounded by
\begin{align*}
\frac{1}{X(\log X)^2} \sum_{(J_{i_1}, J_{i_2}, J_{i_3}, J_{i_4}) \in \mathcal{J}} \frac{I}{\eta^4}.
\end{align*}

We are left with estimating the sum of integrals. Recall that
\begin{align*}
I = I_{i_1, i_2, i_3, i_4} = \int \left(\int \frac{b_{i_1}(s)b_{i_4}(t/s)}{s} \d s\right)\left(\int \frac{b_{i_2}(s)b_{i_3}(t/s)}{s} \d s\right) \d t
\end{align*}
and that $|b_i(t)| \le 2\eta$ for any $i$ and $t$. Let $A, B, C, D$ be powers of $1 + \eta$ such that $\supp(b_{i_1}) \subset [A, (1 + \eta)^3A], \ldots , \supp(b_{i_4}) \subset [D, (1+\eta)^3D]$. We have
\begin{align*}
\frac{I}{\eta^4} &\ll  \int \left(\int \frac{1_{s \in [A, (1 + \eta)^3A]}1_{t/s \in [D, (1 + \eta)^3D]}}{s} \d s\right)\left(\int \frac{1_{s \in [B, (1 + \eta)^3B]}1_{t/s \in [C, (1 + \eta)^3C]}}{s} \d s\right)\d t.
\end{align*}
The integral over $t$ is supported in those values for which $AD \le t \le AD(1 + \eta)^6$ and $BC \le t \le BC(1 + \eta)^6$. In particular, in order for $I$ to be non-zero we must have $BC(1+\eta)^{-6}\leq AD\leq BC(1+\eta)^6$. The inner integrals are bounded by $\log((1 + \eta)^3)$, resulting in the bound
\begin{align*}
\frac{I}{\eta^4} \ll \int_{AD \le t \le AD(1 + \eta)^6} \log((1 + \eta)^3)^2 \d t \ll AD\eta^3.
\end{align*}
By symmetry, we also have the bound $I \ll BC\eta^3$, and thus
\begin{align}
\label{eq:int_bound}
\frac{I}{\eta^4} \ll \sqrt{ABCD} \eta^3.
\end{align}

Furthermore, as we consider only $(J_{i_1}, \ldots,  J_{i_4}) \in \mathcal{J}$, we must have $A^2 + B^2 \le X(1 + \eta)$ and $A^2(1 + \eta)^6 + B^2(1 + \eta)^6 \ge X$. Hence, in particular, the set $[A, A(1 + \eta)^3] \times [B, B(1 + \eta)^3] \subset \mathbb{R}^2$ is a subset of the annulus
$$\mathcal{A} = \{\mathbf{x} \in \mathbb{R}^2 : |\mathbf{x}|^2 \in [X/(1 + \eta)^6, X(1 + \eta)^7]\}.$$
The analogous result holds for $C$ and $D$.

Finally, note that the bound~\eqref{eq:int_bound} may be written as $I/\eta^4 \ll \eta \sqrt{AB\eta^2} \sqrt{CD\eta^2}$, the terms $AB\eta^2$ and $CD\eta^2$ corresponding to the areas of the rectangles $[A, A(1 + \eta)^3] \times [B, B(1 + \eta)^3]$ and $[C, C(1 + \eta)^3] \times [D, D(1 + \eta)^3]$.

All in all, we have
\begin{align*}
\sum_{(J_{i_1}, J_{i_2}, J_{i_3}, J_{i_4}) \in \mathcal{J}} \frac{I}{\eta^4} &\ll  \sum_{\substack{A, B, C, D \\ \log(A)/\log(1+\eta), \ldots , \log(D)/\log(1+\eta) \in \mathbb{Z} \\ [A, A(1 + \eta)^3] \times [B, B(1 + \eta)^3] \subset \mathcal{A} \\ [C, C(1 + \eta)^3] \times [D, D(1 + \eta)^3] \subset \mathcal{A} \\ BC(1 + \eta)^{-6} \le AD \le BC(1 + \eta)^6}} \eta \sqrt{AB\eta^2} \sqrt{CD\eta^2} \\
&\ll\sum_{\substack{r \\ \log(r)/\log(1 + \eta) \in \mathbb{Z}}} \sum_{-6 \le \ell \le 6} \eta \\ 
&\times \Big(\sum_{\substack{A, B \\ \log(A)/\log(1+\eta) \in \mathbb{Z} \\ \log(B)/\log(1+\eta) \in \mathbb{Z} \\ [A, A(1+\eta)^3] \times [B, B(1 + \eta)^3] \subset \mathcal{A} \\ A/B = r}} \sqrt{AB\eta^2}\Big)\Big(\sum_{\substack{C, D \\ \log(C)/\log(1+\eta) \in \mathbb{Z} \\ \log(D)/\log(1+\eta) \in \mathbb{Z} \\ [C, C(1+\eta)^3] \times [D, D(1 + \eta)^3] \subset \mathcal{A} \\ C/D = r(1 + \eta)^\ell}} \sqrt{CD\eta^2}\Big).
\end{align*}
Noting that in the inner sums $A, B, C$ and $D$ run over $O(1)$ values, we obtain
\begin{align*}
&\ll  \sum_{\substack{r \\ \log(r)/\log(1 + \eta) \in \mathbb{Z}}} \eta \left|\mathcal{A} \cap \{(x, y) \in \mathbb{R}^2 : x/y \in [r/(1 + \eta)^9, r(1 + \eta)^9]\}\right| \\
& \ll  \eta |\mathcal{A}| \\
& \ll \eta^2X,
\end{align*}
as desired.
\end{proof}

\section{An additive divisor problem -- proof of Proposition~\ref{prop_divisorsum}}\label{sec:divisor}

In this section we prove Proposition~\ref{prop_divisorsum}. As our argument closely follows the proof in~\cite{DI}, we are at times brief with the exposition, referring the reader to~\cite{DI} for details.

\subsection{Rephrasing}
\label{sec:rephrasing}

We first note a parametrization for the solutions of $x^2 + y^2 \equiv 0 \pmod{T}$ for square-free $T$. For a given $x$, let $g = (x, T)$. Then one has $(y, T) = g$ as well, and one may take the common factor $g$ out. For invertible $x', y'$, the solutions of $x'^2 + y'^2 \equiv 0 \pmod{T'}$ are given by a set of lines of form $y' \equiv tx' \pmod{T'}$, where $t$ varies over the solutions of $t^2 \equiv -1 \pmod{T'}$. (Indeed, if $y' \equiv tx' \pmod{T'}$ for such $t$, then clearly $x'^2 + y'^2 \equiv 0 \pmod{T'}$, and if $x'^2 + y'^2 \equiv 0 \pmod{T'}$, then $(y'/x')^2 \equiv -1 \pmod{T'}$ and hence we can write $y' \equiv tx' \pmod{T'}$ with $t \equiv y'/x' \pmod{T'}$.)

Note that $m_1m_2 - m_3m_4 = k$, $(k, T_1T_2) = 1$ and $T_1 \mid m_1^2 + m_3^2$ imply $(m_1, T_1) = (m_3, T_1) = 1$. Similarly $(m_2, T_2) = (m_4, T_2) = 1$. Hence, our task is to estimate for each $t_i \pmod{T_i}$ with $t_i^2 \equiv -1 \pmod{T_i}$ the sum
\begin{align}
\label{eq:rephrasing}
\sum_{\substack{m_1, m_2, m_3, m_4 \\ m_i \in \text{supp}(b_i) \\ m_1m_2 - m_3m_4 = k \\ m_3 \equiv t_1m_1 \pmod{T_1} \\ m_2 \equiv t_2m_4 \pmod{T_2}}} b_1(m_1)b_2(m_2)b_3(m_3)b_4(m_4).
\end{align}

\subsection{Eliminating \texorpdfstring{$m_2$}{m2}}
\label{sec:elim2}

We start by eliminating the variable $m_2$ in our sum. Note that by the mean value theorem and the bound $|b_2'(t)| \ll 1/M_2$, we have, for $m_i$ as in~\eqref{eq:rephrasing},
\begin{align*}
b_2(m_2) - b_2\left(\frac{m_3m_4}{m_1}\right) = b_2\left(\frac{m_3m_4 + k}{m_1}\right) - b_2\left(\frac{m_3m_4}{m_1}\right) \ll \frac{k}{M_1M_2}.
\end{align*}
From this and the divisor bound, we deduce that
\begin{align*}
& \sum_{\substack{m_1, m_2, m_3, m_4 \\ m_1m_2 - m_3m_4 = k \\ m_3 \equiv t_1m_1 \pmod{T_1} \\ m_2 \equiv t_2m_4 \pmod{T_2}}} b_1(m_1)b_2(m_2)b_3(m_3)b_4(m_4) \\
= & \sum_{\substack{m_1, m_2, m_3, m_4 \\ m_1m_2 - m_3m_4 = k \\ m_3 \equiv t_1m_1 \pmod{T_1} \\ m_2 \equiv t_2m_4 \pmod{T_2}}} b_1(m_1)b_2\left(\frac{m_3m_4}{m_1}\right)b_3(m_3)b_4(m_4) + O(kx^{\varepsilon})
\end{align*}
for any $\varepsilon > 0$. The error is negligible.

By elementary number theory,
$$m_2 = \frac{m_3m_4 + k}{m_1}$$
has a solution $m_2 \in \mathbb{Z}$ satisfying $m_2 \equiv t_2m_4 \pmod{T_2}$ if and only if $m_3m_4 + k \equiv t_2m_4m_1 \pmod{T_2m_1}$, i.e.
$$m_4(t_2m_1 - m_3) \equiv k \pmod{T_2m_1}.$$
This equation is solvable in $m_4 \in \mathbb{Z}$ if and only if $g \coloneqq (t_2m_1 - m_3, T_2m_1)$ divides $k$. In this case the solution set is
$$m_4 \equiv k/g  \cdot \overline{(t_2m_1 - m_3)/g} \pmod{T_2m_1/g},$$
where $\overline{(t_2m_1 - m_3)/g}$ is the inverse of $(t_2m_1 - m_3)/g$ modulo $T_2m_1/g$. For brevity, we denote this congruence by $m_4 \equiv R_{m_1, m_3} \pmod{T_{m_1, m_3}}$.

Hence,
\begin{align*}
& \sum_{\substack{m_1, m_2, m_3, m_4 \\ m_1m_2 - m_3m_4 = k \\ m_3 \equiv t_1m_1 \pmod{T_1} \\ m_2 \equiv t_2m_4 \pmod{T_2}}} b_1(m_1)b_2\left(\frac{m_3m_4}{m_1}\right)b_3(m_3)b_4(m_4) \\
= & \sum_{\substack{m_1, m_3 \\ m_3 \equiv t_1m_1 \pmod{T_1} \\ (t_2m_1 - m_3, T_2m_1) \mid k}} b_1(m_1)b_3(m_3) \sum_{m_4 \equiv R_{m_1, m_3} \pmod{T_{m_1, m_3}}} b_2(m_3m_4/m_1)b_4(m_4).
\end{align*}

\subsection{Eliminating \texorpdfstring{$m_4$}{m4}}

The argument is similar to~\cite[Section 3]{DI}, so we keep our exposition brief. By the Poisson summation formula, one is able to treat sums of form
$$\sum_{\substack{n \le x \\ n \equiv a \pmod{q}}} f(n)$$
for $C^1$ functions $f$. This leads to
\begin{align}
\label{eq:elim_m4}
& \sum_{\substack{m_1, m_3 \\ m_3 \equiv t_1m_1 \pmod{T_1} \\ (t_2m_1 - m_3, T_2m_1) \mid k}} b_1(m_1)b_3(m_3) \sum_{\substack{m_4 \equiv R_{m_1, m_3} \pmod{T_{m_1, m_3}}}} b_2(m_3m_4/m_1)b_4(m_4) \nonumber\\
= & \sum_{\substack{m_1, m_3 \\ m_3 \equiv t_1m_1 \pmod{T_1} \\ (t_2m_1 - m_3, T_2m_1) \mid k}} \frac{b_1(m_1)b_3(m_3)}{T_{m_1, m_3}} \int_{\supp( b_4)} b_2(tm_3/m_1)b_4(t) \d t +E, 
\end{align}
where
\begin{align}\begin{split}\label{eq:elim5}
E= & \sum_{\substack{m_1, m_3 \\ m_3 \equiv t_1m_1 \pmod{T_1} \\ (t_2m_1 - m_3, T_2m_1) \mid k}} b_1(m_1)b_3(m_3) \sum_{\substack{h \in \mathbb{Z} \\ h \neq 0}} \frac{1}{2\pi i h} e\left(\frac{-R_{m_1, m_3}h}{T_{m_1, m_3}}\right)  \\
& \times \left(-\int_{\supp b_4} (b_2(tm_3/m_1)b_4(t))'e\left(\frac{ht}{T_{m_1, m_3}}\right) \d t\right). \end{split}
\end{align}
The first term in~\eqref{eq:elim_m4} corresponds to a main term, while the $E$ term where the sum ranges over $h \neq 0$ corresponds to an error term.

Let us write in~\eqref{eq:elim5}
\begin{align*}
E=E_{\leq H}+E_{>H},
\end{align*}
where $E_{\leq H}$ corresponds to the summation condition $0<|h|\leq H$ and $E_{>H}$ corresponds to the summation condition $|h|>H$. 

We show that the sum over $h$ is small enough, first taking care of the tails $|h| > H \coloneqq x^{10\delta}$ (say), after which we consider small values of $h$.

\subsection{Estimation of the tails}

Write $g(t) = b_2(tm_3/m_1)b_4(t)$. Then the integral in~\eqref{eq:elim5} may be written as
\begin{align*}
\int_{\supp(g)} g'(t)e\left(\frac{ht}{T_{m_1, m_3}}\right) \d t,
\end{align*}
which, after partial integration and the triangle inequality, is bounded by
\begin{align*}
\ll_K\left(\frac{T_{m_1, m_3}}{h}\right)^K \int_{\supp(g)} |g^{(K+1)}(t)| \d t
\end{align*}
for any $K > 0$. One computes $|g^{(K)}(t)| \ll_K x^{K(-1 + 2\delta)}$. Since $|h| > x^{10\delta}$ and $T_{m_1, m_3} \ll x^{1 + \delta}$, by taking $K$ to be a large enough constant we obtain an upper bound of $h^{-2}x^{-10}$ (say) to the above. Plugging this into~\eqref{eq:elim5} gives us
\begin{align*}
E_{> H} \ll \sum_{1 \le m_1, m_3 \le x^{1+\varepsilon}} \sum_{|h| > H} h^{-2}x^{-10} \ll x^{-1}
\end{align*}
(say), which is sufficient.

\subsection{Estimation of contribution of small \texorpdfstring{$h$}{h}}

We are interested in bounding 
\begin{align}
\label{eq:small_h}\begin{split}
E_{\leq H}= -\sum_{m_1} b_1(m_1) \sum_{0 < |h| \le H} \frac{1}{2\pi i h} &\int  \sum_{\substack{m_3 \equiv t_1m_1 \pmod{T_1} \\ (t_2m_1 - m_3, T_2m_1) \mid k}} (b_2(tm_3/m_1)b_4(t))'b_3(m_3)\\
&\times e\left(\frac{h(t - R_{m_1, m_3})}{T_{m_1, m_3}}\right) \d t.
\end{split}
\end{align}
By the triangle inequality, for fixed $m_1, h$ and $t$ we reduce to bounding
\begin{align}
\label{eq:small_h_bound}
\left|\sum_{\substack{r \pmod{L} \\ r \equiv t_1m_1 \pmod{T_1} \\ (t_2m_1 - r, T_2m_1) \mid k}}
\sum_{\substack{m_3 \equiv r \pmod{L}}} (b_2(tm_3/m_1)b_4(t))'b_3(m_3)e\left(\frac{h(t - R_{m_1, m_3})}{T_{m_1, m_3}}\right)\right|,
\end{align}
where $L = \lcm(T_1, T_2m_1)$. Note that $m_3 \equiv r \pmod{L}$ implies $T_{m_1, m_3} = T_{m_1, r}$ and $R_{m_1, m_3} = R_{m_1, r}$. 

Similarly as when eliminating $m_4$, we apply the Poisson summation formula to the sum over $m_3$. We bound~\eqref{eq:small_h_bound} by
\begin{align}
\label{eq:int_and_sum}
&\Big|\frac{1}{L} \sum_{\substack{r \pmod{L} \\ r \equiv t_1m_1 \pmod{T_1} \\ (t_2m_1 - r, T_2m_1) \mid k}} e\left(\frac{h(t - R_{m_1, r})}{T_{m_1, r}}\right) \sum_{\ell \in \mathbb{Z}} e\left(-\frac{r\ell}{L}\right)\int (b_2(ts/m_1)b_4(t))'b_3(s)e\left(\frac{\ell s}{L}\right) \d s\Big| \nonumber \\
\le &\frac{1}{L} \sum_{\ell \in \mathbb{Z}} \Big|\int (b_2(ts/m_1)b_4(t))'b_3(s)e\left(\frac{\ell s}{L}\right) \d s\Big|\Big| \sum_{\substack{r \pmod{L} \\ r \equiv t_1m_1 \pmod{T_1} \\ (t_2m_1 - r, T_2m_1) \mid k}} e\left(-\frac{r\ell}{L}\right) e\left(\frac{h(t - R_{m_1, r})}{T_{m_1, r}}\right)\Big|,
\end{align}
where the integral is over the support of $b_3$. We consider the contribution of $|\ell| > x^{10\delta}$ and $|\ell| \le x^{10\delta}$ separately.

For large $|\ell| > x^{10\delta}$, the idea is to bound the sum over $r$ trivially as $L$ and estimate the integral by integrating by parts $K$ times for a large constant $K$. Write 
$$b(s) = \frac{\partial}{\partial t} \left(b_2(ts/m_1)b_4(t)\right)b_3(s) = \left(\frac{s}{m_1}b_2'(ts/m_1)b_4(t) + b_2(ts/m_1)b_4'(t)\right)b_3(s).$$
One sees that if $f_1$ and $f_2$ both are compactly supported functions satisfying the derivative bound $|f_i^{(k)}(s)| \ll_k C_is^{-k}$ in their domain for all $k \in \mathbb{Z}_{\geq 0}$ and some constants $C_i$ independent of $k$, then $f_1 + f_2$ and $f_1f_2$ satisfy such bounds as well with the corresponding factors $C_1 + C_2$ and $C_1C_2$. Since $s \mapsto s/m_1$, $s \mapsto b_2'(ts/m_1)$, $s \to b_4(s)$, $s \mapsto b_2(ts/m_1)$, $s \to b_4'(s)$ and $s \mapsto b_3(s)$ are such functions with $C = \max(M_1/M_3, M_3/M_1) \ll x^{\delta}$, it follows that
$$|b^{(K)}(s)| \ll_K x^{4K\delta}s^{-K}.$$
Hence, by integrating the integral over $s$ in~\eqref{eq:int_and_sum} by parts $K$ times and estimating the sum over $r$ trivially as $L$, we bound the contribution of $|\ell| > x^{10\delta}$ by
\begin{align*}
\frac{1}{L} \sum_{|\ell| > x^{10\delta}} \left|\int b^{(K)}(s) \left(\frac{L}{\ell}\right)^Ke\left(\frac{\ell s}{L}\right)\d s\right| \cdot L \ll \sum_{\ell > x^{10\delta}} \frac{1}{\ell^{100}} \ll x^{-100\delta}
\end{align*}
for $K$ a large enough constant.

We then consider the contribution of small $|\ell| \le x^{10\delta}$. In this case we estimate the integral in~\eqref{eq:int_and_sum} trivially as $O(x^{2\delta})$, and our task is to obtain a non-trivial bound for
\begin{align}
\label{eq:kloosterman_0}
S_{m_1, h} \coloneqq \left|\sum_{\substack{r \pmod{L} \\ r \equiv t_1m_1 \pmod{T_1} \\ \\ (t_2m_1 - r, T_2m_1) \mid k}} e\left(-\frac{r\ell}{L}\right) e\left(\frac{h(t - R_{m_1, r})}{T_{m_1, r}}\right)\right|.
\end{align}
The idea is that the sum in~\eqref{eq:kloosterman_0} is essentially a Kloosterman sum for which we have power-saving bounds. However, the details require some attention.

We begin by writing $S_{m_1, h}$ as
\begin{align}
\label{eq:kloosterman_1}
S_{m_1, h} &= \Big|\sum_{\substack{r \pmod{L} \\ r \equiv t_1m_1 \pmod{T_1} \\ (t_2m_1 - r, T_2m_1) \mid k}} e\left(\frac{-R_{m_1, r}h}{T_{m_1, r}}\right)e\left(\frac{th}{T_{m_1, r}}\right)e\left(-\frac{r\ell}{L}\right)\Big| \nonumber \\
&\le \sum_{g \mid (k, T_2m_1)} \Big|\sum_{\substack{r \pmod{L} \\ r \equiv t_1m_1 \pmod{T_1} \\ (t_2m_1 - r, T_2m_1) = g}} e\left(\frac{k/g \cdot \overline{(t_2m_1 - r)/g}h}{T_2m_1/g}\right)e\left(\frac{r\ell}{L}\right)\Big|,
\end{align}
where $\overline{a}$ denotes the inverse of $a$ modulo $T_2m_1/g$. We separate the condition $r \equiv t_1m_1 \pmod{T_1}$ by writing
$$1_{r \equiv t_1m_1 \pmod{T_1}} = \frac{1}{T_1} \sum_{v \pmod{T_1}} e\left(\frac{vr - vt_1m_1}{T_1}\right),$$
so that~\eqref{eq:kloosterman_1} turns into
\begin{align*}
S_{m_1, h} &\leq \frac{1}{T_1}\sum_{g \mid (k, T_2m_1)} \Big|\sum_{v \pmod{T_1}} \sum_{\substack{r \pmod{L} \\ (t_2m_1 - r, T_2m_1) = g}} e\left(\frac{k/g \cdot \overline{(t_2m_1 - r)/g}h}{T_2m_1/g}\right)e\left(\frac{r\ell}{L}+\frac{vr - vt_1m_1}{T_1}\right)\Big| \nonumber \\
&\le \frac{1}{T_1}\sum_{g \mid (k, T_2m_1)} \sum_{v \pmod{T_1}} \Big|\sum_{\substack{r \pmod{L} \\ (t_2m_1 - r, T_2m_1) = g}} e\left(\frac{k/g \cdot \overline{(t_2m_1 - r)/g}h}{T_2m_1/g}\right) e\left(\frac{r\ell}{L}\right)e\left(\frac{vr}{T_1}\right)\Big|.
\end{align*}

We perform the substitution $t_2m_1 - r \to r'$ in the inner sum above, obtaining
\begin{align}
\label{eq:kloosterman_2}
\frac{1}{T_1} \sum_{g \mid (k, m_1)} \sum_{v \pmod{T_1}} \Big|\sum_{\substack{r' \pmod{L} \\ (r', T_2m_1) = g}} e\left(\frac{k/g \cdot \overline{r'/g} h}{T_2m_1/g}\right) e\left(\frac{(-r' + t_2m_1)\ell}{L}\right)e\left(\frac{v(-r' + t_2m_1)}{T_1}\right)\Big|.
\end{align}
Note that the translations by $t_2m_1$ do not affect the absolute value of the sum. We then let $r' = gs$ in~\eqref{eq:kloosterman_2} to get
\begin{align*}
S_{m_1, h} \le \frac{1}{T_1} \sum_{g \mid (k, m_1)} \sum_{v \pmod{T_1}} \Big|\sum_{\substack{s \pmod{L/g} \\ (s, T_2m_1/g) = 1}} e\left(\frac{\overline{s}kh/g}{T_2m_1/g}\right)e\left(\frac{-s\ell}{L/g}\right)e\left(\frac{-svg}{T_1}\right)\Big|.
\end{align*}

Let $L_1$ denote the largest divisor of $L/g$ coprime with $T_2m_1/g$ and let $L_2 = L/(L_1g)$. Any $s \pmod{L/g}$ may be written uniquely as $L_2a + b$, where $b$ is an integer modulo $L_2$ and $a$ is an integer modulo $L_1$. Note that $L_2 \equiv 0 \pmod{T_2m_1/g}$ and that $(s, T_2m_1/g) = 1$ if and only if $b$ is invertible modulo $L_2$. Hence, the above may be written as
\begin{align}
\label{eq:kloosterman_3}
\frac{1}{T_1} \sum_{g \mid k} \sum_{v \pmod{T_1}} \Big|\sum_{0 \le a < L_1} \sum_{\substack{0 \le b < L_2 \\ (b, L_2) = 1}} e\left(\frac{\overline{b}kh/g}{T_2m_1/g}\right)e\left(\frac{-(L_2a + b)\ell}{L/g}\right)e\left(\frac{-(L_2a + b)vg}{T_1}\right) \Big|.
\end{align}

By Bezout's lemma, we may write $1/(L/g)$ as $c/L_1 + d/L_2$ for some $c, d \in \mathbb{Z}$. This gives $1/T_1 = c'/L_1 + d'/L_2$ for $c' = cL/(T_1g)$ and $d' = dL/(T_1g)$. Plugging these in~\eqref{eq:kloosterman_3} gives
\begin{align*}
S_{m_1, h} \le \frac{1}{T_1} \sum_{g \mid k} \sum_{v \pmod{T_1}} \Big|\sum_{\substack{0 \le b < L_2 \\ (b, L_2) = 1}} e\left(\frac{\overline{b}kh/g}{T_2m_1/g}\right)e\left(\frac{-db\ell}{L_2}\right)e\left(\frac{-d'bvg}{L_2}\right) \times \\
\sum_{0 \le a < L_1} e\left(\frac{-(L_2a + b)c\ell}{L_1}\right)e\left(\frac{-(L_2a + b)c'vg}{L_1}\right)\Big|.
\end{align*}
The value of the $a$-sum is independent of $b$ by the coprimality of $L_1$ and $L_2$, and it is bounded by $L_1$ in absolute value, so we obtain an upper bound
\begin{align*}
S_{m_1, h} &\le \frac{1}{T_1} \sum_{g \mid k} L_1 \sum_{v \pmod{T_1}} \Big|\sum_{\substack{0 \le b < L_2 \\ (b, L_2) = 1}} e\left(\frac{\overline{b}kh/g}{T_2m_1/g}\right)e\left(\frac{-b(d\ell + d'vg)}{L_2}\right)\Big| \\
&= \frac{1}{T_1} \sum_{g \mid k} L_1 \sum_{v \pmod{T_1}} \Big|\sum_{\substack{0 \le b < L_2 \\ (b, L_2) = 1}} e\left(\frac{\overline{b}kh/g \cdot L_2/(T_2m_1/g)}{L_2}\right)e\left(\frac{-b(d\ell + d'vg)}{L_2}\right)\Big|.
\end{align*}
This inner sum is finally a Kloosterman sum, to which we apply Weil's upper bound~\cite[Corollary 11.12]{iw-kow} to get, for any $\varepsilon > 0$,
\begin{align*}
S_{m_1, h} & \le  \frac{1}{T_1} \sum_{g \mid k} L_1 \sum_{v \pmod{T_1}} \gcd(L_2, khL_2/(T_2m_1), d\ell + d'vg)^{1/2}L_2^{1/2+\varepsilon} \\
& \ll \frac{1}{T_1} \sum_{g \mid k} L_1 \sum_{v \pmod{T_1}} \sqrt{\frac{khL_2}{T_2m_1}}L_2^{1/2+\varepsilon} \\
& \ll \sum_{g \mid k} L^{1+\varepsilon}\sqrt{\frac{kh}{T_2m_1}} \\
& \ll x^{10\delta}\sqrt{h}m_1^{1/2 + \varepsilon},
\end{align*}
where in the last step we used $L \le T_1T_2m_1 \le x^{2\delta}m_1$

Plugging this upper bound to~\eqref{eq:int_and_sum}, we bound $E_{\le H}$ in~\eqref{eq:small_h} by
\begin{align*}
& \sum_{m_1} b_1(m_1) \sum_{0 < |h| \le H} \frac{1}{2\pi h} \int_{M_3}^{2M_3} \frac{1}{L} S_{m_1, h} \cdot x^{2\delta} \d t \\
& \ll  \sum_{m_1} b_1(m_1) \sum_{0 < |h| \le H} \frac{1}{2\pi h} \int_{M_3}^{2M_3} \frac{1}{L} \sum_{|\ell| < x^{10\delta}} x^{10\delta}\sqrt{h}m_1^{1/2 + \varepsilon} \cdot x^{2\delta} \d t  \\ 
& \ll  x^{30\delta}M_1^{1/2 + \varepsilon}M_3,
\end{align*}
which is $\ll x^{1.6}$, say, for $\delta > 0$ small enough.

\subsection{Calculating the main terms}

We finally evaluate the main term
\begin{align}
\label{eq:main_term}
\sum_{\substack{m_1, m_3 \\ m_3 \equiv t_1m_1 \pmod{T_1} \\ (t_2m_1 - m_3, T_2m_1) \mid k}} \frac{b_1(m_1)b_3(m_3)}{T_{m_1, m_3}} \int b_2(tm_3/m_1)b_4(t) \d t
\end{align}
in~\eqref{eq:elim_m4}. Recall from Section~\ref{sec:elim2} that $T_{m_1, m_3} = T_2m_1/g = T_2m_1/(t_2m_1 - m_3, T_2m_1)$. We write~\eqref{eq:main_term} as
\begin{align}
\label{eq:main_term_2}
\int \sum_{\substack{m_1, m_3 \\ m_3 \equiv t_1m_1 \pmod{T_1} \\ (t_2m_1 - m_3, T_2m_1) \mid k}} \frac{b_1(m_1)b_2(t/m_1)b_3(m_3)b_4(t/m_3)}{m_1m_3T_2/(t_2m_1 - m_3, T_2m_1)} \d t.
\end{align}
We compute the sum inside the integral. First, by M\"{o}bius inversion,
\begin{align}
\label{eq:main_term_inversion}
&\sum_{\substack{m_1, m_3 \\ m_3 \equiv t_1m_1 \pmod{T_1} \\ (t_2m_1 - m_3, T_2m_1) \mid k}} \frac{b_1(m_1)b_2(t/m_1)b_3(m_3)b_4(t/m_3)}{m_1m_3T_2/(t_2m_1 - m_3, T_2m_1)} \nonumber \\
= &\frac{1}{T_2} \sum_{\substack{g \mid k}} g \sum_{e \in \mathbb{N}} \mu(e) \sum_{\substack{m_1, m_3 \\ m_3 \equiv t_1m_1 \pmod{T_1} \\ ge \mid t_2m_1 - m_3 \\ ge \mid T_2m_1}} \frac{b_1(m_1)b_2(t/m_1)b_3(m_3)b_4(t/m_3)}{m_1m_3} \nonumber \\
= &\frac{1}{T_2} \sum_{g \mid k} g \sum_{e \in \mathbb{N}} \mu(e) \sum_{\substack{m_1 \\ ge \mid T_2m_1}} \frac{b_1(m_1)b_2(t/m_1)}{m_1} \sum_{\substack{m_3 \equiv t_1m_1 \pmod{T_1} \\ ge \mid t_2m_1 - m_3}} \frac{b_3(m_3)b_4(t/m_3)}{m_3}.
\end{align}
Note that if $b$ is a smooth, compactly supported function, then by partial summation for any $a, q \in \mathbb{N}$ we have
\begin{align}
\label{eq:abel_sum}
\left|\sum_{n \equiv a \pmod{q}} b(n) - \frac{1}{q} \int b(t) \d t\right| \ll \int |b'(t)| \d t.
\end{align}
Hence
\begin{align*}
\sum_{\substack{m_3 \equiv t_1m_1 \pmod{T_1} \\ ge \mid t_2m_3 - m_3}} \frac{b_3(m_3)b_4(t/m_3)}{m_3} = \frac{1_{(ge, T_1) \mid (t_2 - t_1)m_1}}{\lcm(T_1, ge)} \int \frac{b_3(s)b_4(t/s)}{s}\d s+ O\left(\frac{1}{M_3}\right).
\end{align*}

Plugging this into~\eqref{eq:main_term_inversion}, summing the error over $g, e$ and $m_1$ (noting that we may restrict to $e \ll x^{1 + \delta}$) and integrating over $t$ in~\eqref{eq:main_term_2} gives a total error $\ll x^{1 + O(\delta)}$, which is acceptable. The main term in~\eqref{eq:main_term_inversion} then becomes
\begin{align*}
\frac{1}{T_2} \left(\int \frac{b_3(s)b_4(t/s)}{s} \d s\right) \sum_{g \mid k} g \sum_{e \in \mathbb{N}} \frac{\mu(e)}{\lcm(T_1, ge)} \sum_{\substack{m_1 \\ ge \mid T_2m_1 \\ (ge, T_1) \mid (t_2 - t_1)m_1}} \frac{b_1(m_1)b_2(t/m_1)}{m_1}.
\end{align*}
By another application of~\eqref{eq:abel_sum}, the inner sum here may be written as
\begin{align*}
&\sum_{\substack{m_1 \\ ge \mid T_2m_1 \\ (ge, T_1) \mid (t_2 - t_1)m_1}} \frac{b_1(m_1)b_2(t/m_1)}{m_1} \\
= &\frac{1}{\lcm(ge/(ge, T_2), (ge, T_1)/(ge, T_1, t_2 - t_1))}\int \frac{b_1(s)b_2(t/s)}{s}\d s+ O\left(\frac{1}{M_1}\right),
\end{align*}
and again the error is found to be negligible.

Thus, the main term~\eqref{eq:main_term_2} is (up to admissible errors) equal to
\begin{align}
\label{eq:main_term_3}
& \int \left(\int \frac{b_1(s)b_2(t/s)}{s} \d s\right)\left(\int \frac{b_3(s)b_4(t/s)}{s} \d s\right) \d t \nonumber \\
\times &\frac{1}{T_2}\sum_{g \mid k} g \sum_{e \in \mathbb{N}} \frac{\mu(e)}{\lcm(T_1, ge)} \cdot \frac{1}{\lcm(ge/(ge, T_2), (ge, T_1)/(ge, T_1, t_2 - t_1))}.
\end{align}
The integral over $t$, which agrees with the one given in Proposition~\ref{prop_divisorsum}, is a normalization factor depending only on the chosen functions $b_i$. We are left with computing the sum
\begin{align}
\label{eq:def_S}
S_{t_1, t_2} \coloneqq \frac{1}{T_2} \sum_{g \mid k} g\sum_{e \in \mathbb{N}} \frac{\mu(e)}{\lcm(T_1 ,ge)} \cdot \frac{1}{\lcm(ge/(ge, T_2), (ge, T_1)/(ge, T_1, t_2 - t_1))}.
\end{align}
and summing it over $t_i^2 \equiv -1 \pmod{T_i}$.

Some manipulation yields
\begin{align*}
S_{t_1, t_2} &= \frac{1}{T_2} \sum_{g \mid k} g \sum_{e \in \mathbb{N}} \frac{\mu(e)}{\lcm(T_1, ge)\lcm(ge/(ge, T_2), (ge, T_1)/(ge, T_1, t_2 - t_1))} \\
&= \frac{1}{T_2} \sum_{g \mid k} g \sum_{e \in \mathbb{N}} \frac{\mu(e)(T_1, ge)(ge/(ge, T_2), (ge, T_1)/(ge, T_1, t_2 - t_1))}{T_1g^2e^2(ge, T_1)/((ge, T_2)(ge, T_1, t_2 - t_1))} \\
&= \frac{1}{T_1T_2} \sum_{g \mid k} \frac{1}{g} \sum_{e \in \mathbb{N}} \frac{\mu(e)(ge(ge, T_1, t_2 - t_1), (ge, T_1)(ge, T_2))}{e^2}.
\end{align*}
At this point we invoke the assumption $(k, T_1T_2) = 1$, from which it follows that $(g, T_1T_2) = 1$. The sum simplifies to
\begin{align*}
S_{t_1, t_2} = \frac{1}{T_1T_2} \left(\sum_{g \mid k} \frac{1}{g}\right)\left(\sum_{e \in \mathbb{N}} \frac{\mu(e)(e(e, T_1, t_2 - t_1), (e, T_1)(e, T_2))}{e^2}\right).
\end{align*}
The sum over $e$ is multiplicative, and thus by Euler products
\begin{align*}
S_{t_1, t_2} = \frac{1}{T_1T_2} \left(\sum_{g \mid k} \frac{1}{g}\right) \prod_{p \text{ prime}} \left(1 - \frac{(p(p, T_1, t_2 - t_1), (p, T_1)(p, T_2))}{p^2}\right).
\end{align*}
Recalling that $T_1, T_2$ are square-free, for $e = p$ a prime the numerator equals $p^v$ for some $v \in \{0, 1, 2\}$. The case $v = 2$ occurs if and only if $p \mid T_1, T_2, t_1 - t_2$, and $v = 1$ occurs if $p$ divides $T_1T_2$ but not $(T_1, T_2, t_2 - t_1)$. Note that if $v = 2$ occurs for some $p$, then $S_{t_1, t_2}$ vanishes.

Hence we may write
\begin{align*}
S_{t_1, t_2} &= \frac{1_{(T_1, T_2, t_1 - t_2) = 1}}{T_1T_2}\left(\sum_{g \mid k} \frac{1}{g}\right) \prod_{\substack{p \mid T_1T_2}} \frac{p-1}{p} \prod_{p \nmid T_1T_2} \frac{p^2 - 1}{p^2} \\
&= \frac{6 \cdot 1_{(T_1, T_2, t_1 - t_2) = 1}}{T_1T_2\pi^2}\left(\sum_{g \mid k} \frac{1}{g}\right)\prod_{p \mid T_1T_2}  \frac{p}{p+1}
\end{align*}
using $\prod_p(1-p^{-2})=1/\zeta(2)=6/\pi^2$.

We now sum  $S_{t_1, t_2}$ over all $t_1 \pmod{T_1}, t_2 \pmod{T_2}$ satisfying $t_i^2 \equiv -1 \pmod{T_i}$. We have
\begin{align}
\label{eq:sum_S}
\begin{split}
&\sum_{\substack{t_1 \pmod{T_1}, \ t_2 \pmod{T_2} \\ t_i^2 \equiv -1 \pmod{T_i}}} S_{t_1, t_2} \\
&= \frac{6}{T_1 T_2 \pi^2} \prod_{p \mid T_1T_2} \frac{p}{p+1} \left(\sum_{g \mid k} \frac{1}{g}\right) \sum_{\substack{t_1 \pmod{T_1}, \ t_2 \pmod{T_2} \\ t_i^2 \equiv -1 \pmod{T_i}}} 1_{(T_1, T_2, t_1 - t_2) = 1}.  
\end{split}
\end{align}
One computes that the inner sum satisfies
\begin{align}
\label{eq:sum_ti}
\sum_{\substack{t_1 \pmod{T_1}, \ t_2 \pmod{T_2} \\ t_i^2 \equiv -1 \pmod{T_i}}} 1_{(T_1, T_2, t_1 - t_2) = 1} = \prod_{p \mid T_1T_2} g_p((p, T_1), (p, T_2)),
\end{align}
where
\begin{align*}
g_p(1, p) = g_p(p, 1) = \sum_{\substack{t \pmod{p} \\ t^2 \equiv -1 \pmod{p}}} 1 = \begin{cases} 1, \quad \text{if } p = 2 \\ 2, \quad \text{if } p \equiv 1 \pmod{4} \\ 0, \quad \text{if } p \equiv 3 \pmod{4}\end{cases}
\end{align*}
and
\begin{align*}
g_p(p, p) = \sum_{\substack{t_1 \pmod{p}, \ t_2 \pmod{p} \\ t_i^2 \equiv -1 \pmod{p} \\ t_1 \not\equiv t_2 \pmod{p}}} 1 =  \begin{cases} 0, \quad \text{if } p \not\equiv 1 \pmod{4} \\ 2, \quad \text{if } p \equiv 1 \pmod{4}\end{cases}
\end{align*}
(and $g_p(1, 1) = 1$). Combining~\eqref{eq:main_term_3},~\eqref{eq:sum_S} and~\eqref{eq:sum_ti} we conclude the proof of Proposition~\ref{prop_divisorsum}.

\section{Proof of Theorem~\ref{thm:E2}}
\label{sec:proof}

In view of Lemma~\ref{lem:reduct}, Theorem~\ref{thm_variance} for $k = 2$ (and thus Theorem~\ref{thm:E2}) follows from the following proposition.

\begin{proposition}
\label{prop:E2}
Let $\varepsilon>0$ be small enough and let $C=\C$. Let $X\geq 3, P_1=(\log X)^{C-1}$, and let
$$\beta_n = \begin{cases} 1, \quad n \equiv p_1p \quad \text{with} \quad P_1^{1 - \varepsilon} \le \N(p_1) \le P_1,  \\ 0, \quad \text{otherwise}.\end{cases}$$
and
$$F(m) = \sum_{X < \N(n) \le 2X} \beta_n \lambda^m(n).$$
Then
$$\sum_{0 < |m| \le X/(\log X)^{C-\varepsilon}} |F(m)|^2 = o\left(\frac{1}{(\log X)^2}\right).$$
\end{proposition}

For the proof of this proposition (as well as for Proposition~\ref{prop:E3} below), we need the following mean square estimate of prime Hecke polynomials; the strength of the exponents in this result determines our exponent $C$. 

\begin{proposition}[Sparse mean squares of Hecke polynomials over primes]\label{prop:sparse}
Let $\varepsilon>0$ be small but fixed. Let $X\geq X'\geq X/2\geq 2$, and let $P(m)=\sum_{X'\leq \N(p)\leq X}\lambda^m(p)/\N(p)$. Let $\mathcal{T}\subset [-X,X]\cap \mathbb{Z}$ satisfy 
\begin{align}\label{eq:Tbound2}
|\mathcal{T}|\ll X^{\constant + \varepsilon},    
\end{align}
and suppose that for some $F\in [X^{\varepsilon/2},X^{2\varepsilon}]$ and some Hecke polynomial $$F(m)=\sum_{F^{1-\varepsilon^2}\leq \N(n)\leq F}a_n\lambda^m(n)/\N(n)$$ 
with $\sum_{n}|a_n|^2\ll F$ we have
\begin{align*}
\mathcal{T}\subset \{m\in \mathbb{Z}:\,\, |F(m)|\geq F^{-\frac{5}{141}-10\varepsilon^2}\}.
\end{align*}
Then for any $A\geq 1$ we have 
\begin{align*}
\sum_{m\in \mathcal{T}}|P(m)|^2\ll_A(\log X)^{-A}.     
\end{align*}
\end{proposition}

Let us first see how Proposition~\ref{prop:sparse} implies Proposition~\ref{prop:E2}.

\begin{proof}[Proof of Proposition~\ref{prop:E2} assuming Proposition~\ref{prop:sparse}]
Let $\varepsilon>0$ be small enough. Write $T = X/(\log X)^{C-\varepsilon}$ and let $\eta$ be a parameter tending to $0$ slowly in terms of $X$. Applying Lemma~\ref{lem:factorize} with $H = 1/\log(1 + \eta)$, we obtain, with $I = ((1 - \varepsilon) (\log P_1) H- 1, (\log P_1)H]$,
\begin{align}
\label{eq:fact_applied}
\sum_{0 < |m| \le T} |F(m)|^2 \ll |I| \sum_{v \in I} \sum_{0 < |m| \le T} |A_{v, H}(m)B_{v, H}(m)|^2 + T\sum_{\substack{|\arg n_1 - \arg n_2| \le 1/T \\ \N(n_1), \N(n_2) \in [X/(1 + \eta), X(1 + \eta)] \text{ or} \\ \N(n_1), \N(n_2) \in [2X, 2X(1 + \eta)]}} c_{n_1}c_{n_2},
\end{align}
where
$$c_n = \frac{1}{\N(n)} \sum_{\substack{n \equiv p_1p \\ P_1^{1 - \varepsilon} \le \N(p_1) \le P_1}} 1$$
and
$$A_{v, H}(m) = \sum_{e^{v/H} < \N(p_1) \le e^{(v+1)/H}} \frac{\lambda^m(p_1)}{\N(p_1)}, \quad B_{v, H}(m) = \sum_{\N(p) \sim Xe^{-v/H}} \frac{\lambda^m(p)}{\N(p)}.$$
The second sum on the right of~\eqref{eq:fact_applied} is $\ll \eta^2/(\log X)^2 = o(1/(\log X)^2)$ by Proposition~\ref{prop:mvtapplication}. For the first sum on the right of~\eqref{eq:fact_applied} we take the maximum over $v$. Let the maximum be attained by $v = v_0$, and denote $P_1'(m) = A_{v_0, H}(m)$ and $P(m) = B_{v_0, H}(m)$, so that
$$P_1'(m) = \sum_{P_1' \le \N(p_1) < (1 + \eta)P_1'} \frac{\lambda^m(p_1)}{\N(p_1)}, \quad P(m) = \sum_{\N(p) \sim X/P_1'} \frac{\lambda^m(p)}{\N(p)}$$
for some $P_1'\in [P_1^{1-\varepsilon}/2,2P_1]=[(\log X)^{(1-\varepsilon)(C-1)}/2,2(\log X)^{C-1}]$.
Our goal is to show that
\begin{align}
\label{eq:target_E2}
\sum_{0 < |m| \le T} |P_1'(m)P(m)|^2 = o\left(\frac{\eta^2}{(\log X)^2(\log P_1)^2}\right).
\end{align}
We shall in fact prove a bound of $\ll (\log X)^{-2-\varepsilon^2}$. Note that $P_1' = (\log X)^{C-1 + O(\varepsilon)}$.

Let 
\begin{align*}
 \mathcal{T}_1=\{0 < |m| \le T:\,\, |P_1'(m)| \le P_1'^{-\varepsilon^2}\} \quad \text{and} \quad \mathcal{T}= ([-T,T]\cap \mathbb{Z}) \setminus(\{0\}\cup \mathcal{T}_1).   
\end{align*}
The contribution of $\mathcal{T}_1$ to the sum in~\eqref{eq:target_E2} is bounded via the pointwise bound $|P_1(m)| \le P_1^{-\varepsilon^2}$ and the improved mean value theorem (Lemma~\ref{lem:IMVT}), yielding
\begin{align*}
\sum_{m \in \mathcal{T}_1} |P_1'(m)P(m)|^2 \ll P_1'^{-2\varepsilon^2}\left(T\sum_{\N(n)\sim N}|a_n|^2+ T \sum_{\substack{|\arg n_1 - \arg n_2| \le 1/T \\ n_1\neq n_2\\\N(n_1), \N(n_2) \sim N}} |a_{n_1}a_{n_2}|\right),
\end{align*}
where $N=X/P_1'\in [X/P_1, 2X/P_1^{1 - \varepsilon}]$ and $a_n=1/\N(n)$ if $n$ is a Gaussian prime and $a_n=0$ otherwise. We estimate this sum using Proposition~\ref{prop:mvtapplication} and see that the previous expression is
$$\ll P_1'^{-2\varepsilon^2} \left(\frac{T}{(X/P_1')(\log X)}+\frac{1}{(\log X)^{2}}\right) \ll (\log X)^{-2-\varepsilon^2},$$
since $T=X/(\log X)^{C-\varepsilon}$ and $P_1'>\log X$.

We are left with the contribution of $\mathcal{T}$. Since $|P_1'(m)|\ll 1$, to prove~\eqref{eq:target_E2} it suffices to show that
\begin{align}
\label{eq:log-saving}
\sum_{m \in \mathcal{T}} |P(m)|^2\ll_A(\log X)^{-A}
\end{align}
for any fixed $A > 0$. In order to deduce~\eqref{eq:log-saving} from Proposition~\ref{prop:sparse}, we need some properties of the set $\mathcal{T}$.

Firstly, 
by Corollary~\ref{cor:newLVT} we have
 \begin{align}\label{eq:Tbound}
|\mathcal{T}| \ll X^{1/((3/2)(C-1)-3) + O(\varepsilon^2)}\ll X^{\constant + \varepsilon},
\end{align}
for $C=\C$ if $\varepsilon>0$ is small enough. 

Secondly, note that if  
\begin{align*}
F(m):=P_1'(m)^k/k!^{1/2},\quad F=(2P_1')^k,\quad k=\left\lfloor \varepsilon\frac{\log T}{\log(2P_1')}\right\rfloor    
\end{align*}
and if $b_n$ are the coefficients of $F(m)$, then $b_n$ are supported in $F^{1-\varepsilon^2}\leq \N(n)\leq F$ and
\begin{align*}
\sum_{\N(n)\leq F}|b_n|^2\leq \sum_{\N(n)\leq F}\sum_{\substack{n\equiv p_1\cdots p_k\\\N(p_1)\sim P_1'\ldots, \N(p_k)\sim P_1'}}1 \leq  P_1^k\leq F. 
\end{align*}

Finally, note that, since $F=T^{\varepsilon+o(1)}$, for $m\in \mathcal{T}$ we have
\begin{align*}
|F(m)|\geq P_1'^{-2\varepsilon^2k}/e^{(k/2)\log k}\gg F^{-2\varepsilon^2-o(1)-1/(2(C-1))-o(1)}\gg F^{-1/(2(C-1))-3\varepsilon^2}\gg F^{-\frac{5}{141}-3\varepsilon^2}
\end{align*}
for $C = 15.1$.

In view of these properties of $\mathcal{T}$, we may apply Proposition~\ref{prop:sparse} to deduce~\eqref{eq:log-saving}.
\end{proof}

We then turn to the proof of Proposition~\ref{prop:sparse}.

\begin{proof}[Proof of Proposition~\ref{prop:sparse}]
\textbf{Step 1: Applying Heath-Brown's decomposition.} We apply Heath-Brown's decomposition (Lemma~\ref{lem:HB}) with parameters $k = 100$ and $B\geq 1$ large to get
\begin{align}
\label{eq:apply_HB}
\sum_{m \in \mathcal{T}} |P(m)|^2 \ll_B (\log X)^D\sum_{m \in \mathcal{T}}|M_1(m) \cdots M_{J}(m)|^2+(\log X)^{-B}
\end{align}
for some constant $D=D_B > 0$ and some Hecke polynomials $$M_j(m)=\sum_{\N(n)\sim M_j}a_{n,j}\lambda^m(n)$$
with  $1\leq j\leq J \le 200$, $M_j\geq 1$ and $X/\exp(2(\log(2X))^{19/20})\leq M_1\cdots M_{J}\leq X$. Here $|a_{n,j}|\leq 1$, and 
$$|M_j(m)|\ll \exp(-(\log(2M_j))^{1/10})$$ 
for $m\in [1,T]\cap \mathbb{Z}$ and $M_j(m)$ is smooth (in the sense of Definition~\ref{def:smooth}) if $M_j\geq X^{1/100}$. 

In what follows, let $\varepsilon>0$ be a small enough constant. 
We bound any $M_{j}$ shorter than $\exp((\log X)^{99/100})$ appearing on the right-hand side of~\eqref{eq:apply_HB} trivially by $|M_j(m)|\ll 1$. Hence, on redefining $J$ and relabeling the $M_j$, we may assume all of $M_1, \ldots , M_{J}$ are long enough so that by Lemma~\ref{lem:vinogradov}
\begin{align}\label{eq:Mj}
|M_j(m)| \ll \exp(-(\log X)^{1/20}) \end{align}
and that we have $$X^{1-o(1)} \ll M_1 \cdots M_{J} \le X.$$
We may also assume that $M_j\geq X^{1/100}$ for all $1\leq j\leq J$, since otherwise by applying~\eqref{eq:Mj} to $M_j$ and Proposition~\ref{prop:HM}(ii) with $F(m) = \prod_{i\neq j}M_i(m)$ we have
\begin{align*}
\sum_{m \in \mathcal{T}}|M_1(m) \cdots M_{J}(m)|^2\ll_A (\log X)^{-2A+O(1)}(1+|\mathcal{T}|X^{5/6-99/100+o(1)})\ll (\log X)^{-A}    \end{align*}
 for small enough $\varepsilon>0$, since by~\eqref{eq:Tbound2} we have $|\mathcal{T}|\ll X^{\constant+\varepsilon}$. We may also assume that $J \ge 2$, as otherwise by Proposition~\ref{prop:zeta}(ii) we have
$$\sum_{m \in \mathcal{T}} |M_1(m)|^2 \ll |\mathcal{T}|\left(T^{1/3}X^{-1/2 + \varepsilon} + X^{-3/8 + \varepsilon}\right),$$
which again is sufficient by $|\mathcal{T}| \ll X^{\constant+\varepsilon}$.

In order to make Proposition~\ref{prop:zeta}(i) and Lemma~\ref{le:density1}(i) applicable, we write
\begin{align*}
M_j(m)&=\widetilde{M}_j(m)+E_j(m),
\end{align*}
where 
\begin{align*}
\widetilde{M}_j(m)=\sum_{\substack{\N(n)\sim M_j\\\arg(n)\not \in I_1\cup \cdots \cup I_r}}a_{n,j}\lambda^{m}(n),\quad E_j(m)=\sum_{\substack{\N(n)\sim M_j\\\arg(n) \in I_1\cup \cdots \cup I_r}}a_{n,j}\lambda^{m}(n)     \end{align*}
with 
\begin{align}\begin{aligned}\label{eq:Idef}
I_i&=[\alpha_i-X^{-\varepsilon^4},\alpha_i+X^{-\varepsilon^4}]\textnormal{ for } i=1,\ldots, r,\\
 \{\alpha_1,\ldots, \alpha_r\}&=\{t\in [0,\frac{\pi}{2}]:\,\exists\, k\in \{1,\ldots,R\}:\, \Im((1+i\tan(t))^k)\cdot \Im((1+i\tan(2t))^k)=0\},
 \end{aligned}
\end{align}
with $R$ a large enough constant. 
Then by the mean value theorem (Lemma~\ref{lem:MVT}), the divisor bound, and the fact that $|a_{n,j}|\leq \log \N(n)$ we have 
\begin{align*}
&\sum_{m\in \mathcal{T}}|E_1(m)|^2|\widetilde{M}_2(m)|^2\cdots |\widetilde{M}_J(m)|^2\\
&\ll  (\log X)^{O(1)}\frac{T+X}{(M_1\cdots M_J)^2}\sum_{M_1\cdots M_J\leq \N(n)\leq 2^JM_1\cdots M_J}\left(\sum_{\substack{n\equiv n_1\cdots n_J\\\N(n_j)\sim M_j\,\forall\, j\leq J\\\arg(n_1)\in I_1\cup\cdots \cup I_r}}1\right)^2\\
&\ll X^{o(1)}\frac{T+X}{(M_1\cdots M_J)^2}\sum_{M_1\cdots M_J\leq \N(n)\leq 2^JM_1\cdots M_J}\sum_{\substack{n\equiv n_1\cdots n_J\\\N(n_j)\sim M_j\,\forall j\leq J\\\arg(n_1)\in I_1\cup\cdots \cup I_r}}1\ll X^{-\varepsilon^4/2},     
\end{align*}
by Lemma~\ref{lem:countSector} if $\varepsilon>0$ is small enough. Arguing similarly, we see that for all $1\leq k\leq J$ we have
\begin{align*}
 \sum_{m\in \mathcal{T}}\prod_{j=1}^k|E_j(m)|^2\prod_{j=k+1}^{J}|\widetilde{M}_j(m)|^2\ll X^{-\varepsilon^4/2}.     
\end{align*}
 Hence,  it suffices to show that 
\begin{align*}
\sum_{m\in \mathcal{T}}|\widetilde{M}_1(m)|^2\cdots |\widetilde{M}_J(m)|^2\ll_A(\log X)^{-A}.    
\end{align*}
For any interval $\mathcal{J}\subset [0,\pi/2],$ let $\widetilde{M}_{j,\mathcal{J}}(n)$ be the same sum as $\widetilde{M}_j$, but with the additional summation condition $\arg(n)\in \mathcal{J}$.
By the pigeonhole principle, there exist some intervals $\mathcal{J}_1,\ldots, \mathcal{J}_J$ of length $\asymp X^{-2\varepsilon^4}$ such that
\begin{align*}
\sum_{m\in \mathcal{T}}|\widetilde{M}_1(m)|^2\cdots |\widetilde{M}_J(m)|^2\ll  X^{4J\varepsilon^4} \sum_{m\in \mathcal{T}}|\widetilde{M}_{1,\mathcal{J}_1}(m)|^2\cdots |\widetilde{M}_{J,\mathcal{J}_J}(m)|^2.    
\end{align*}
Now, by permuting the indices if necessary, it suffices to show that
\begin{align*}
\sum_{m\in \mathcal{T}'}|\widetilde{M}_{1,\mathcal{J}_1}(m)|^2\cdots |\widetilde{M}_{J,\mathcal{J}_J}(m)|^2\ll X^{-\varepsilon^3},    
\end{align*}
say, where
\begin{align*}
\mathcal{T}'=\{m\in [-T,T]\cap \mathbb{Z}\:\,\, |\widetilde{M}_{1,\mathcal{J}_1}(m)|\geq |\widetilde{M}_{2,\mathcal{J}_2}(m)|\geq \cdots \geq |\widetilde{M}_{J,\mathcal{J}_J}(m)|\}.    
\end{align*}

Let us write
\begin{align*}
N_1(m)=\widetilde{M}_{1,\mathcal{J}_1}(m),\quad N_2(m)=\widetilde{M}_{2,\mathcal{J}_2}(m)\cdots \widetilde{M}_{J,\mathcal{J}_J}(m),\quad N_1=M_1,\quad N_2=M_2\cdots M_J.    
\end{align*}
With this notation, it suffices to show that
\begin{align}
\sum_{m\in \mathcal{T}'}|N_1(m)|^2|N_2(m)|^2\ll X^{-\varepsilon^3}.   
\end{align}
 We have now decomposed our Hecke polynomial in the desired manner. We recall here for convenience that by the above analysis we have the constraints
$$X^{1/100}\leq N_1\leq X^{99/100},\quad X^{1-o(1)}\ll N_1N_2\leq X.$$
Moreover, for later use we note the following important properties of $N_1(m)$:
\begin{enumerate}[(a)]

\item The coefficients of $N_1(m)$ are supported in $\arg(n)\in \mathcal{I}$, where $\mathcal{I}$ is some interval that is $\gg X^{-\varepsilon^4}$ away from all the solutions to $\Im((1+i\tan(t))^k)=0$ with $k=1,\ldots, R$ (this follows directly from the construction of $I_1,\ldots, I_r$ in~\eqref{eq:Idef}).

\item The coefficients of $N_1(m)^2$ are supported in $\arg(n)\in \mathcal{I}'$, where $\mathcal{I}'$ is some interval that is $\gg X^{-\varepsilon^4}$ away from all the solutions to $\Im((1+i\tan(t))^k)=0$ with $k=1,\ldots, R$ (this is because if $\mathcal{J}_1=[\alpha-\delta,\alpha+\delta]$, the coefficients of $N_1(m)^2$ are supported in $\arg(n)\in [2\alpha-2\delta,2\alpha+2\delta]$, and by the construction of $I_1,\ldots, I_r$ in~\eqref{eq:Idef} the interval $[2\alpha-2\delta,2\alpha+2\delta]$ is $\gg X^{-\varepsilon^4}$ away from all the solutions to $\Im((1+i\tan(t))^k)=0$ with $k=1,\ldots, R$.). 
\end{enumerate}

\textbf{Step 2: Splitting of the summation range and conclusion.} Define
\begin{align*}
\mathcal{T}_{\sigma}=\{m\in \mathcal{T}':\,\, |N_1(m)|\sim N_1^{-\sigma}\}.    \end{align*}
The definition of $\mathcal{T}'$ tells us that for $m\in \mathcal{T}_{\sigma}$ we also have
\begin{align}\label{eq:N_2}
|N_2(m)|\leq N_2^{-\sigma}.     
\end{align}

By property (a), Proposition~\ref{prop:zeta}(i) and Remark~\ref{rem:pointwise}, the polynomial $N_1(m)$ admits a power-saving bound and thus the set  $\mathcal{T}_{\sigma}$ is empty unless
\begin{align}\label{eq:delta0}
\sigma\geq \delta_0   \end{align}
for some small absolute constant $\delta_0$. By dyadic decomposition, it suffices to show that 
\begin{align*}
\sum_{m\in\mathcal{T}_{\sigma}} |N_1(m)N_2(m)|^2 &\ll X^{-2\varepsilon^3},   
\end{align*}
say. Recalling~\eqref{eq:N_2}, this bound follows if we show that
\begin{align}\label{eq:goal2}
|\mathcal{T}_{\sigma}|\ll N_1^{2\sigma}N_2^{2\sigma}X^{-2\varepsilon^3},     
\end{align}
Observe for later use that, since $N_1N_2\gg X^{1-o(1)}$ and $\varepsilon>0$ is small, by~\eqref{eq:delta0} we have~\eqref{eq:goal2} if  
\begin{align}\label{eq:goal}
|\mathcal{T}_{\sigma}|\ll T^{2(1-\varepsilon^2)\sigma},     
\end{align}
say. 

Note that by~\eqref{eq:Tbound} we  have~\eqref{eq:goal} unless
\begin{align}\label{eq:minbound}
\sigma\leq \frac{\constant}{2}+O(\varepsilon). \end{align}
We split the proof of~\eqref{eq:goal2} into cases depending on the size of $N_1$.

\textbf{Case 1: $N_1\in [X^{1/100},X^{1/3-\varepsilon}]\cup [X^{3/8},X^{1/2-\varepsilon}]\cup [X^{3/4},X^{99/100}]$.} Write $N_1=X^{\beta}$. 
Let $1\leq \ell\leq 100$ be an integer such that $\beta\ell\in [3/4-\varepsilon,1-\varepsilon]$. 
By Lemma~\ref{le:density1}(ii) applied to $N_1(m)^{\ell}$, we have~\eqref{eq:goal}, provided that 
$$C_0\varepsilon\leq \sigma\leq \frac{1}{32}-O(\varepsilon)$$
for some absolute constant $C_0$. 
The first inequality above holds for $\varepsilon>0$ small enough since $\sigma\geq \delta_0$ by~\eqref{eq:delta0}, and the second inequality holds by~\eqref{eq:minbound} since for $\varepsilon>0$ small enough we have
\begin{align}\label{eq:case2.2}
\frac{1}{32}>\frac{\constant}{2}+O(\varepsilon).
\end{align}
Hence,~\eqref{eq:goal} holds in this case. 

\textbf{Case 2: $N_1\in (X^{1/2-\varepsilon},X^{3/4})$.} By property (a), the Hecke polynomial $N_1(m)$ is the restriction of a smooth Hecke polynomial to a region where Proposition~\ref{prop:zeta}(i) is applicable. By Proposition~\ref{prop:zeta}(i) with the exponent pair $(\kappa,\lambda)=(0.02381, 0.8929)$ as in Lemma~\ref{lem:exp_pairs},
we see that $\mathcal{T}_{\sigma}$ is empty unless
\begin{align}\label{eq:sigma2lower}
\sigma\geq \frac{1-3\kappa-\lambda}{2}+O(\varepsilon)\geq 0.0178.    
\end{align}

Write
\begin{align}\label{eq:betabounds}
N_1=X^{\beta},\quad \frac{1}{2}-\varepsilon\leq \beta\leq \frac{3}{4}.    
\end{align}
By Lemma~\ref{le:density2} with $\delta=0.7509-\beta$ and $A=(5/141)^{-1}/2=14.1$, we have~\eqref{eq:goal} if
\begin{align}\label{eq:beta2}
\frac{0.7509-\beta}{14.1(2-2\beta)}+O(\varepsilon)\leq \sigma \leq \frac{0.12635 -\frac{3.0036-4\beta}{14.1}}{8\beta-2}-O(\varepsilon).    
\end{align}

In the range~\eqref{eq:betabounds}, the left-hand side of~\eqref{eq:beta2} is maximized at $\beta=1/2-\varepsilon$ and the right-hand side of~\eqref{eq:beta2} is minimized also at $\beta=1/2-\varepsilon$. Hence,~\eqref{eq:goal} holds if
\begin{align}\label{eq:case1.1}
\frac{0.2509}{14.1}+O(\varepsilon)\leq \sigma \leq 0.063175-\frac{0.5018}{14.1}-O(\varepsilon).    
\end{align}
Combining this with~\eqref{eq:sigma2lower} and~\eqref{eq:minbound} and taking $\varepsilon>0$ small, it now suffices to note that
\begin{align}\begin{aligned}\label{eq:check1}
\frac{0.2509}{14.1}&<0.0178,\\
0.063175-\frac{0.5018}{14.1}&> \frac{\constant}{2}.
\end{aligned}
\end{align}
Hence~\eqref{eq:check1} holds, so~\eqref{eq:goal} follows.  

\textbf{Case 3: $N\in (X^{1/3-\varepsilon},X^{3/8})$.} Now
\begin{align}\label{eq:beta3}
N_1=X^{\beta},\quad \frac{1}{3}-\varepsilon\leq \beta\leq \frac{3}{8}.    
\end{align}
Applying Lemma~\ref{le:density2} to $N_1(m)^2$ as in Case 2, we have~\eqref{eq:goal} if~\eqref{eq:case1.1} holds. Hence, we may assume that
\begin{align*}
\sigma<\frac{0.2509}{C-1}-O(\varepsilon).    
\end{align*}

Now we apply Lemma~\ref{le:density1}(i) to $N_1(m)^2$ with the exponent pair $(\kappa,\lambda)=(0.05, 0.825)$ as in Lemma~\ref{lem:exp_pairs} (noting that by property (b) the coefficients of $N_1(m)^2$ are supported in the set required for the application of Lemma~\ref{le:density1}(i)). We deduce that~\eqref{eq:goal} holds provided that
\begin{align*}
\sigma\leq \frac{2\beta\frac{\kappa-\lambda+1}{2\kappa}-1}{2\beta(2+2/\kappa)-2}-O(\varepsilon).    
\end{align*}
In the range~\eqref{eq:beta3}, the right-hand side is minimized at $\beta=1/3-\varepsilon$, in which case the previous inequality implies 
\begin{align*}
\sigma\leq 0.01923.
\end{align*}
Now note that
\begin{align}\label{case2.3}
0.01923 >\frac{0.2509}{C-1}+O(\varepsilon)    
\end{align}
since $C>14.1$. Hence, we must have~\eqref{eq:goal}.

Combining all the above cases,~\eqref{eq:goal} follows, and this was enough to complete the proof of Proposition~\ref{prop:E2}.
\end{proof}

\section{Proof of Theorem~\ref{thm:E3}}\label{sec:proof2}

By Lemma~\ref{lem:reduct}, Theorem~\ref{thm_variance} for $k = 3$ (and thus Theorem~\ref{thm:E3}) will follow from the following proposition.

\begin{proposition}
\label{prop:E3}
Let $\varepsilon>0$ be small enough and $C=\Co$. Let $X\geq 3, P_1=(\log \log X)^{C-1}$, $P_2=(\log X)^{\varepsilon^{-1}},$ and let
\begin{align*}
\beta_n = \begin{cases}  \frac{1}{\N(n)}, \quad \text{if} \quad n \equiv p_1p_2p, \quad P_i^{1 - \varepsilon} \le \N(p_i) \le P_i, \\ 0, \quad \text{otherwise}\end{cases}
\end{align*}
and
$$F(m) = \sum_{X < \N(n) \le 2X} \beta_n \lambda^m(n).$$
Then
$$\sum_{0 < |m| \le X(\log \log X)^{\varepsilon - C}/(\log X)} |F(m)|^2 = o\left(\frac{1}{(\log X)^2}\right).$$
\end{proposition}

\begin{proof}
Write $T = X(\log \log X)^{\varepsilon-C}/(\log X)$. Similarly as in the proof of Proposition~\ref{prop:E2}, we apply Lemma~\ref{lem:factorize} to obtain
\begin{align}
\label{eq:fact_applied_E3}
\sum_{0 < |m| \le T} F(m) \ll |I|^2 \sum_{0 < |m| \le T} |A_{v_0, H}(m)B_{v_0, H}(m)|^2 + T\sum_{\substack{|\arg n_1 - \arg n_2| \le 1/T \\ \N(n_1), \N(n_2) \in [X/(1 + \eta), X(1 + \eta)] \text{ or} \\ \N(n_1), \N(n_2) \in [2X, 2X(1 + \eta)]}} c_{n_1}c_{n_2}
\end{align}
with $I = [(1 - \varepsilon)\log P_1 /\log(1 + \eta) - 1, \log P_1 / \log (1 + \eta))$, $v_0 \in I$, $\eta \to 0$ slowly in terms of $X$,
$$c_n = \frac{1}{\N(n)} \sum_{\substack{n = p_1p_2p \\ P_i^{1 - \varepsilon} \le \N(p_i) \le P_i}} 1,$$
and
$$A_{v_0, H}(m) = \sum_{e^{v_0/H} \le \N(p_1) < e^{(v_0+1)/H}} \frac{\lambda^m(p_1)}{\N(p_1)}, \quad B_{v, H}(m) = \sum_{\substack{n = p_2p \\ \N(n) \sim Xe^{-v_0/H} \\ P_2^{1 - \varepsilon} \le \N(p_2) \le P_2}} \frac{\lambda^m(n)}{\N(n)}.$$
As before, the second term in~\eqref{eq:fact_applied_E3} is $o\left(1/(\log X)^2)\right)$ by Proposition~\ref{prop:mvtapplication}. We again denote
$$P_1'(m) = A_{v_0, H}(m) = \sum_{P_1' \le \N(p_1) \le (1 + \eta)P_1'} \frac{\lambda^m(p_1)}{\N(p_1)},$$
where $P_1' \in [P_1^{1-\varepsilon}, P_1]$, and so we wish to show
$$\sum_{0 < |m| \le T} |P_1'(m)B_{v_0, H}(m)|^2 = o\left(\frac{\eta^2}{(\log X)^2(\log P_1)^2}\right).$$

Let $\alpha_1 = \varepsilon$ and let 
$$\mathcal{T}_1=\{0<|m|\leq T:\,\, |P_1'(m)| \le P_1'^{-\alpha_1}\}$$
and
$$\mathcal{T}=([-T,T]\cap \mathbb{Z})\setminus (\{0\}\cup \mathcal{T}_1).$$

For bounding the contribution of $\mathcal{T}_1$, we use the improved mean value theorem (Lemma~\ref{lem:IMVT}) and Proposition~\ref{prop:mvtapplication} as in the proof of Proposition~\ref{prop:E2} to get
\begin{align*}
\sum_{m \in \mathcal{T}_1} |P_1'(m)B_{v_0, H}(m)|^2 &\ll P_1'^{-2\alpha_1} \left(\frac{T}{(X/P_1')(\log X)}+\frac{1}{(\log X)^{2}}\right)\\
&\ll (\log X)^{-2}P_1'^{-\varepsilon},
\end{align*}
which is sufficient.

For bounding the contribution of $\mathcal{T}$, we further factorize the polynomial $B_{v_0, H}$. Again via Lemma~\ref{lem:factorize} and Proposition~\ref{prop:mvtapplication} we reduce to showing
\begin{align}\label{eq:P_1P_2P}
\sum_{m \in \mathcal{T}} |P_1'(m)P_2'(m)P(m)|^2 = o\left(\frac{\eta^4}{(\log X)^2(\log P_1)^2(\log P_2)^2}\right),
\end{align}
where
$$P_i'(m) = \sum_{P_i' \le \N(p_i) \le (1 + \eta)P_i'} \frac{\lambda^m(p_i)}{\N(p_i)} \quad \text{and} \quad P(m) = \sum_{\N(p) \sim X/(P_1'P_2')} \frac{\lambda^m(p)}{\N(p)}$$
for some $P_i' \in [P_i^{1 - \varepsilon}, P_i]$. We will in fact obtain an upper bound of $\ll (\log X)^{-2 - \varepsilon}$ for~\eqref{eq:P_1P_2P}.

Write $\alpha_2 = \frac{1}{2(C-1)} + 2\varepsilon$, and let
$$\mathcal{T}_2 = \{m \in \mathcal{T} : |P_2'(m)| \le P_2'^{-\alpha_2}\}$$
and
$$\mathcal{T}_3 = \mathcal{T} \setminus \mathcal{T}_2.$$

Let $\ell = \lceil \log P_2' / \log P_1' \rceil$ and note that for any $m \in \mathcal{T}_2$ we have
$$|P_2'(m)|^2 \leq P_2^{-2\alpha_2}(P_1^{\alpha_1}|P_1(m)|)^{2\ell},$$
so
\begin{align}
\label{eq:high_moment}
\sum_{m \in \mathcal{T}_2} |P_1'(m)P_2'(m)P(m)|^2 &\ll \sum_{m \in \mathcal{T}_2} |P_2'(m)P(m)|^2 \ll P_2'^{-2\alpha_2}P_1'^{2\ell \alpha_1} \sum_{m \in \mathcal{T}_2} |P_1(m)|^{2\ell}|P(m)|^2 \nonumber \\
&\ll P_2^{-2\alpha_2}P_1'^{2\ell \alpha_1} \ell\sum_{m \in \mathcal{T}_2} |A(m)|^2,
\end{align}
where 
$$A(m) = \sum_{\N(n) \sim Y} \frac{A_n}{\N(n)}\lambda^m(n),$$
for some 
\begin{align}
\label{eq:Y_bound}
P_1'^{\ell} \frac{X}{P_1'P_2'} \le Y \le 2^{\ell}P_1'^{\ell} \frac{X}{P_1'P_2'}.
\end{align}
Here the coefficients satisfy $A_n$ are bounded by 
$$|A_n| \le \sum_{\substack{n \equiv  p_1 \cdots p_{\ell}p \\ \N(p_i) \in [P_1', (1 + \eta)P_1'] \\ \N(p) \sim X/(P_1'P_2')}}  1,$$
and, in particular, by unique factorization we then have $|A_n| \le (\ell+1)!$.

By the mean value theorem (with the remark~\eqref{eq_an2}) we may bound~\eqref{eq:high_moment} by
\begin{align*}
P_2'^{-2\alpha_2} P_1'^{2\ell \alpha_1} \ell Y \sum_{\N(n) \sim Y} \frac{|A_n|^2 \tau(n)}{\N(n)^3}.
\end{align*}
Here
\begin{align*}
Y \sum_{\N(n) \sim Y} \frac{|A_n|^2}{\N(n)^3}\tau(n) &\ll Y \frac{(\ell + 1)!2^{\ell}}{Y^2} \sum_{\N(n) \sim Y} \frac{|A_n|}{\N(n)} \\
&\ll \frac{(\ell!)^{1 + o(1)}}{Y} \sum_{\substack{p_1, \ldots , p_{\ell}, p \\ \N(p_i) \in [P_1', (1 + \eta)P_1'] \\ \N(p) \sim X/(P_1'P_2')}} \frac{1}{\N(p_1) \cdots \N(p_{\ell})\N(p)} \\
&\ll \frac{(\ell!)^{1 + o(1)}}{Y \log Y}.
\end{align*}

Thus,~\eqref{eq:high_moment} is bounded by
\begin{align*}
\ll P_2'^{-2\alpha_2}P_1'^{2\ell \alpha_1}(\ell!)^{1 + o(1)} \frac{T+Y}{Y \log Y}.
\end{align*}
By~\eqref{eq:Y_bound} we have $Y = X(\log X)^{o(1)}$, and in particular $T \ll Y$. Hence, the previous expression is
\begin{align*}
\ll (\log X)^{\varepsilon^{-1}(2\alpha_1 - 2(1-\varepsilon)\alpha_2) + \varepsilon^{-1}/(C-1) - 1 + o(1)}.
\end{align*}
By our choice of $\alpha_1, \alpha_2$ we have $2\alpha_1 - 2\alpha_2 = -1/(C-1) - 2\varepsilon$, and hence the above is bounded by $(\log X)^{-2-\varepsilon}$, which is sufficient.

For bounding the contribution of $\mathcal{T}_3$, note that by Remark~\ref{rem:LVT} applied to $P_2(m)$ we obtain $$|\mathcal{T}_3| \ll X^{2\alpha_2 + 2\varepsilon + o(1)}\ll X^{1/(C-1)+O(\varepsilon)}.$$

Bounding trivially $|P_1'(m)|,|P_2'(m)|\ll 1$, it suffices to show that
\begin{align*}
\sum_{m\in \mathcal{T}_3}|P(m)|^2\ll_A (\log X)^{-A}
\end{align*}
for any fixed $A > 0$.

Denoting 
\begin{align*}
F(m):=P_2'(m)^k/k!^{1/2},\quad F=P_2'^k,\quad k=\left\lfloor \varepsilon\frac{\log T}{\log P_2'}\right\rfloor,    
\end{align*}
we see as in the proof of Proposition~\ref{prop:E2} that
the coefficients $b_n$ of $F(m)$ are supported in $F^{1-\varepsilon^2}\leq \N(n)\leq F$ and satisfy
\begin{align*}
\sum_{\N(n)\leq F}|b_n|^2\leq F. 
\end{align*}
Moreover, since $F=T^{\varepsilon+o(1)}$, for $m \in \mathcal{T}_3$ we have
\begin{align*}
|F(m)|\geq P_2'^{-\alpha_2 k}/e^{(k/2)\log k}\gg F^{-\alpha_2 -o(1)}F^{-\varepsilon-o(1)}\gg F^{-1/(2(C-1))-2\varepsilon}.    
\end{align*}

In view of these properties of $\mathcal{T}_3$, we may apply Proposition~\ref{prop:sparse} to deduce that~\eqref{eq:log-saving} holds since
\begin{align*}
\frac{1}{C-1}+O(\varepsilon)<\constant    
\end{align*}
for $C=\Co$ and $\varepsilon>0$ small enough. This completes the proof.
\end{proof}

\appendix

\section{Appendix: Exponent pairs}\label{app:A}

In this appendix, we prove Lemma~\ref{le_gk}, following closely Ivi\'c's argument from~\cite[Chapter 2]{ivic-book}.

\begin{proof}[Proof of Lemma~\ref{le_gk}]

(A) We claim that if $(\kappa,\lambda)$ is an exponent pair of degree $R\geq 1$, then $A(\kappa,\lambda)=(\kappa_1,\lambda_1)$ is an exponent pair of degree $\leq R+1$. We clearly have $0\leq \kappa_1\leq 1/2\leq \lambda_1\leq 1$.

We first note that $(1/2,1/2)$ is an exponent pair of degree $\leq 2$, since by applying~\cite[Lemma 2.4]{ivic-book} (a truncated Poisson formula) and using~\cite[Lemma 2.2]{ivic-book} (the second derivative test) to estimate the exponential integrals appearing in it, for $f\in \mathscr{F}_I(A,B,M,2)$ we have
\begin{align}
\label{eq:half-half}
\left|\sum_{n\in I}e(f(n))\right|\ll M^{3/2}(AB)^{1/2}.     
\end{align}

Then let $f\in \mathscr{F}_I(A,B,M,R+1)$. If $A<B^{1/2}$, we apply~\eqref{eq:half-half} to obtain
\begin{align*}
\left|\sum_{n\in I}e(f(n))\right|&\ll M^{3/2}(AB)^{1/2}=M^{3/2}A^{1/2}B^{1/2+\lambda/(2\kappa+2)}B^{-\lambda/(2\kappa+2)}\\
&\ll M^{3/2} A^{(1+\kappa-2\lambda)/(2\kappa+2)}B^{1/2+\lambda/(2\kappa+2)}\ll M^{3/2}A^{\kappa_1}B^{\lambda_1}.     
\end{align*}

Suppose then that $B^{1/2}\leq A$. By~\cite[Lemma 2.5]{ivic-book} (which is a Weyl-differencing inequality), for any $H>0$ we have
\begin{align*}
\left|\sum_{n\in I}e(f(n))\right|^2\ll B^2H^{-1}+H^2+BH^{-1}\sum_{1\leq j\leq H}\left|\sum_{n\in I\cap (I-j)}e(g_j(n))\right|,     
\end{align*}
where 
$$g_j(t):=f(t+j)-f(t).$$
By the mean value theorem, for $t\in I$ we have
\begin{align*}
g_j^{(r)}(t)=jf^{(r+1)}(t+\theta),    
\end{align*}
where $\theta\in [0,j]$ depends on $j$, $r$ and $t$. Hence, for $t\in I\cap (I-j)$ and $0\leq r\leq R$ we have
\begin{align*}
 jM^{-1}AB^{-r}\leq |g_j^{(r)}(t)|\leq jMAB^{-r}.   
\end{align*}
Now, by applying the existing exponent pair $(\kappa,\lambda)$, for $B/(2AM)<j\leq H$ we have
\begin{align}\label{eq:jsum1}
\left|\sum_{n\in I\cap (I-j)}e(g_j(n))\right|\ll M^{C_{\kappa,\lambda}} (jAB^{-1})^{\kappa}B^{\lambda}\ll  M^{C_{\kappa,\lambda}} (HAB^{-1})^{\kappa}B^{\lambda}  
\end{align}
for some constant $C_{\kappa,\lambda}$. 
To obtain an estimate that works in the remaining range $j\leq B/(2AM)$, note that in this range $|g_j'(t)|\leq 1/2$ and $|g_j''(t)|\geq jM^{-1}AB^{-2}$, so by~\cite[Lemmas 1.2 and 2.2]{ivic-book} (partial summation and the second derivative test) we have
\begin{align}\label{eq:jsum2}
\begin{split}
BH^{-1}\sum_{1\leq j\leq B/(2AM)}\left|\sum_{n\in I\cap (I-j)}e(g_j(n))\right|&\ll BH^{-1}\cdot M^{1/2} A^{-1/2}B\sum_{1\leq j\leq B/(2AM)}j^{-1/2}\\
&\ll B^{5/2}A^{-1}H^{-1}\ll B^2H^{-1}, 
\end{split}
\end{align}
since $B^{1/2}\leq A$. 

Combining~\eqref{eq:jsum1} and~\eqref{eq:jsum2}, we obtain
\begin{align*}
\left|\sum_{n\in I}e(f(n))\right|^2&\ll  B^2H^{-1}+H^2+M^{C_{\kappa,\lambda}}B(HAB^{-1})^{\kappa}B^{\lambda}.  
\end{align*}
Choosing 
\begin{align*}
H=B^{(1+\kappa-\lambda)/(\kappa+1)}A^{-\kappa/(\kappa+1)}    
\end{align*}
and performing some elementary manipulation (cf.~\cite[Lemma 2.8]{ivic-book}), we see that $(\kappa_1,\lambda_1)$ is an exponent pair of degree $\leq R+1$. 

(B) We claim that if $(\kappa,\lambda)$ is an exponent pair of degree $R$ with $\kappa+2\lambda\geq 3/2$, then $B(\kappa,\lambda)$ is an exponent pair of degree $\leq \max\{R+1,4\}$. Note first that $(\kappa_2,\lambda_2):=B(\kappa,\lambda)$ satisfies $0\leq \kappa_2\leq 1/2\leq \lambda_2\leq 1$. 

Let $f\in \mathscr{F}_I(A,B,M,\max\{R+1,4\})$. We may assume that $M\geq 2$, and by symmetry we may assume that $f''(a)<0$. It suffices to prove the claim for intervals of the form $I=[B,B+h]$ with $h\leq B/(2M^2)$, since any interval $I\subset [B,2B]$ is a union of $\ll M^2$ such intervals. Now apply~\cite[Lemma 2.7]{ivic-book} (van der Corput's B-transformation) with $a=B$, $b=B+h$, $m_2=|f''(a)|$, $m_3=(m_2m_4)^{1/2}$ and $m_4=M^3AB^{-3}$. Note that this is an admissible choice, since by the mean value theorem for $y\in [0,h]$ we have 
$$|f''(a+y)-f''(a)|=y|f^{(3)}(a+\theta)|\leq \frac{1}{2}M^{-1}AB^{-1}\leq \frac{1}{2}|f''(a)|$$
for some $\theta\in [0,y]$, and since
\begin{align*}
MAB^{-2}\ll (|f''(a)|M^3AB^{-3})^{1/2}= m_3. 
\end{align*}
We obtain from~\cite[Lemma 2.7]{ivic-book} the estimate
\begin{align*}
\sum_{n\in I}e(f(n))&=e\left(\frac{1}{8}\right)\sum_{\substack{f'(b)<\nu\leq f'(a)\\\nu\in \mathbb{Z}}}|f''(x_{\nu})|^{-1/2}e(g(\nu))+O(M^{1/2}A^{-1/2}B^{1/2})\\
&+O(\log((A+2)M)))+O(M(AB)^{1/3}),    
\end{align*}
where 
$$g(\nu):=f(x_{\nu})-\nu x_{\nu}$$
and $x_{\nu}$ is defined by $f'(x_{\nu})=\nu$. By using Leibniz's rule for the $r$th derivative of a product as in~\cite[proof of Lemma 2.9]{ivic-book}, we see that for $t\in [f'(b),f'(a)]$ we have
\begin{align*}
|g^{(r)}(t)|\ll M^{O_r(1)}BA^{1-r}    
\end{align*}
for all $0\leq r\leq R$ and 
\begin{align*}
|g^{(r)}(t)|\gg M^{-O_r(1)}BA^{1-r}     
\end{align*}
for $0\leq r\leq 3$. 

By partial summation and the fact that $(\kappa,\lambda)$ is an exponent pair of degree $R$, we then have
\begin{align*}
\left|\sum_{\substack{f'(b)<\nu\leq f'(a)\\\nu \in \mathbb{Z}}}|f''(x_{\nu})|^{-1/2}e(g(\nu))\right|\ll M^{O_{\kappa,\lambda}(1)} (A/B)^{-1/2}B^{\kappa}A^{\lambda}    
\end{align*}
and therefore
\begin{align*}
\left|\sum_{n\in I}e(f(n))\right|\ll M^{O_{\kappa,\lambda}(1)}( A^{\kappa_2}B^{\lambda_2}+(AB)^{1/3}).    
\end{align*}
As in~\cite[proof of Lemma 2.9]{ivic-book}, elementary manipulation and the condition $\kappa+2\lambda\geq 3/2$ imply
\begin{align*}
(AB)^{1/3}\leq A^{\kappa_2}B^{\lambda_2},   
\end{align*}
so the claim follows.
\end{proof}

\bibliography{refs}
\bibliographystyle{plain}

\end{document}